\documentclass[11pt, leqno, a4paper,twoside]{amsart} 
\usepackage{lipsum}
\usepackage{amsfonts}
\usepackage{graphicx}
\usepackage{epstopdf}
\usepackage{algorithmic}
\usepackage{calligra}
\usepackage{amsfonts,amsmath,amsthm,amssymb}
\usepackage{mathtools}
\usepackage{hyperref}
\usepackage{autonum}
\usepackage{hhline}
\usepackage{array}
\usepackage{diagbox}
\usepackage{tcolorbox}
\usepackage{mdframed}
\usepackage{multicol}
\usepackage{graphicx}
\usepackage{subcaption}
\usepackage{moreverb}
\usepackage{bbm}
\usepackage[margin=1.38in]{geometry}
\usepackage{todonotes}
\usepackage{scalerel,amssymb}
\allowdisplaybreaks
\usepackage{mathrsfs}  
\usepackage{lineno}
\usepackage{todonotes}
\usepackage{tikz}
\usepackage{pgfplots}
\usetikzlibrary{arrows.meta}
\usepackage{siunitx}
\usepackage[numbers,sort&compress]{natbib}
\definecolor{mygreen}{HTML}{43a047}
\usepackage{subcaption}
\usepackage{doi}
\usepackage{alphalph}
\usepackage{booktabs}
\usepackage{makecell}
\def\dmp{d}

\def\opsi{\overline{\psi}}

\newcommand{\Om}{\Omega}
\newcommand{\D}{\Delta}
\newcommand{\pn}{\psi^n}

\newcommand{\kernel}{\mathfrak{k}}
\newcommand{\kerML}{\mathfrak{k}}
\newcommand{\kersing}{\mathfrak{g}}
\newcommand{\Dt}{\textup{D}_t} 
\newcommand{\Dtal}{{\textup{D}}_t^{2-\alpha}}
\def\elltil{\tilde{\ell}}
\newcommand{\pt}{\psi_t}
\newcommand{\ptn}{\psi_t^n}
\newcommand{\ptt}{\psi_{tt}}
\newcommand{\pttn}{\psi_{tt}^n}
\newcommand{\ptttn}{\psi_{ttt}^n}
\newcommand{\pttt}{\psi_{ttt}}
\newcommand{\ddt}{\frac{\textup{d}}{\textup{d}t}}
\newcommand{\Igamma}{{\textup{I}}^\gamma}
\newcommand{\Is}{{\textup{I}}^s}
\newcommand{\Igammas}{{\textup{I}}^{\gamma+s}}
\newcommand{\Ione}{{\textup{I}}^{1}}
\newcommand{\Ialpha}{{\textup{I}}^\alpha}
\newcommand{\Italpha}{{\textup{I}}^{2\alpha}}
\newcommand{\Dal}{{\textup{D}}_t^\alpha}
\newcommand{\Doal}{{\textup{D}}_t^{1-\alpha}}
\newcommand{\bmu}{\boldsymbol{\mu}}
\newcommand{\bxi}{\boldsymbol{\xi}}
\newcommand{\dt}{\, \textup{d} t}
\newcommand{\ds}{\, \textup{d} s }

\newcommand{\intt}{\int_0^t}

\newcommand{\nLtwo}[1]{\|#1\|_{L^2}}

\newcommand{\nLfour}[1]{\|#1\|_{L^4}}

\newcommand{\nHtwo}[1]{\|#1\|_{H^2}}
\newcommand{\nLtwoLtwo}[1]{\|#1\|_{L^2 (L^2)}}
\newcommand{\nLtwotLtwo}[1]{\|#1\|_{L_t^2 (L^2)}}
\newcommand{\nLtwoLfour}[1]{\|#1\|_{L^2 (L^4)}}

\newcommand{\nLinfLtwo}[1]{\|#1\|_{L^\infty (L^2)}}

\newcommand{\nLinfLfour}[1]{\|#1\|_{L^\infty (L^4)}}
\newcommand{\nLinfLinf}[1]{\|#1\|_{L^\infty (L^\infty)}}

\newcommand{\nLtwotLinf}[1]{\|#1\|_{L^2_t(L^\infty)}}
\newcommand{\nLtwotLfour}[1]{\|#1\|_{L^2_t(L^4)}}
\newcommand{\nLinftLtwo}[1]{\|#1\|_{L^\infty_t(L^2)}}
\newcommand{\spacelow}{X^{\textup{low}}_{\textup{fMGT III}}}
\newcommand{\spacelowIII}{X^{\textup{low}}_{\textup{fMGT III}}}
\newcommand{\spacehighIII}{X^{\textup{high}}_{\textup{fMGT III}}}
\newcommand{\spaceI}{X_{\textup{fMGT I}}}

\newcommand{\prodLtwo}[2]{(#1, #2)_{L^2}}
\newcommand{\R}{\mathbb{R}} 
\newcommand{\N}{\mathbb{N}} 
\newcommand{\Honetwo}{{H_\diamondsuit^2(\Omega)}}
\newcommand{\Honethree}{{H_\diamondsuit^3(\Omega)}}

\newcommand{\spaceKhigh}{X_{\textup{fMGT--K III}}}

\def\LoneLtwo{L^1(L^2)}

\def\LinfLtwo{L^\infty(L^2)}

\newcommand{\CHone}{C_{H^1, L^4}}

\newcommand{\CHtwo}{C_{H^2, L^\infty}}

\newcommand{\CPF}{C_{\textup{PF}}}

\newtheorem{theorem}{Theorem}
\newtheorem{lemma}{Lemma}
\newtheorem{proposition}{Proposition}
\newtheorem*{assumption*}{Assumptions}

\newtheorem{remark}{Remark}
\numberwithin{lemma}{section}
\numberwithin{proposition}{section}
\numberwithin{theorem}{section}
\numberwithin{equation}{section}
\makeatletter
\newcommand{\leqnomode}{\tagsleft@true}
\newcommand{\reqnomode}{\tagsleft@false}
\makeatother
\newcommand{\TK}{\mathcal{T}}

\newcommand{\fmgt}{\textup{fMGT}}
\newcommand{\fmgti}{\textup{fMGT I}}

\definecolor{grey}{rgb}{0.5,0.5,0.5}

\title[Time-fractional MGT equations]{Time-fractional Moore--Gibson--Thompson equations}
\subjclass[2010]{35L05, 35L72}

\keywords{fractional MGT equations, nonlinear acoustics, well-posedness, limiting behavior}

\author[B. Kaltenbacher and V. Nikoli\'c]{\bfseries Barbara Kaltenbacher$^1$ and Vanja Nikoli\'c$^2$}

\address{  
	Department of Mathematics\\  
	Alpen-Adria-Universit\"at Klagenfurt 
	\\ Universitätsstr. 65--67, A-9020 Klagenfurt, Austria}
\email{barbara.kaltenbacher@aau.at}

\address{ 
	Department of Mathematics \\ 
	Radboud University   \\ 
	Heyendaalseweg 135,
	6525 AJ Nijmegen, The Netherlands}
\email{vanja.nikolic@ru.nl} 

\begin{document}
	\maketitle     
	\vspace*{-4mm}      
	\begin{center}
		{\footnotesize
			$^1$Department of Mathematics, Alpen-Adria-Universit\"at Klagenfurt, Austria \\
			$^2$Department of Mathematics, Radboud University, The Netherlands
		}
\end{center}
\vspace*{4mm}
\begin{abstract}
	In this paper, we consider several time-fractional generalizations of the Jordan--Moore--Gibson--Thompson (JMGT) equations in nonlinear acoustics as well as their linear Moore--Gibson--Thompson (MGT) versions. Following the procedure described in Jordan (2014), these time-fractional acoustic equations are derived from four fractional versions of the Maxwell--Cattaneo law in Compte and Metzler (1997). Additionally to providing well-posedness results for each of them, we also study the respective limits as the fractional order tends to one, leading to the classical third order in time (J)MGT equation.
\end{abstract}
	
\section{Introduction}
It is well-known that using the Fourier temperature flux law, given by
\begin{equation} \label{fourier_law}
	\boldsymbol{q}= -\kappa \nabla \theta,
\end{equation}
in the derivation of second-order models of nonlinear acoustics may lead to the so-called paradox of infinite speed of propagation; see~\cite{kuznetsov1971equations, kaltenbacher2009global, kaltenbacher2007numerical, jordan2008nonlinear, jordan2016survey}. As a remedy, the Maxwell--Cattaneo law may be used instead
\[
\boldsymbol{q}+\tau \boldsymbol{q}_t = - \kappa \nabla \theta,
\]
whereby a time lag $\tau>0$ is introduced between the heat flux and the temperature induced by it. This change within the governing equations leads to the third-order in time sound propagation described by a family of Moore--Gibson--Thompson (MGT) equations in linear acoustics:
\begin{equation}
	\begin{aligned}
		\tau \psi_{ttt}+\psi_{tt}-c^2 \Delta \psi- (\tau c^2+\delta) \Delta \psi_t = 0
	\end{aligned}
\end{equation}
or Jordan--Moore--Gibson--Thompson (JMGT) equations in nonlinear acoustics:
\begin{equation}
	\begin{aligned}
		\tau \psi_{ttt}+\psi_{tt}-c^2 \Delta \psi- (\tau c^2+\delta) \Delta \psi_t = f(\psi_t, \psi_{tt}, \nabla \psi, \nabla \psi_{t});
	\end{aligned}
\end{equation}
see the works of Moore and Gibson~\cite{moore1960propagation}, Thompson~\cite{thompson}, and Jordan~\cite{jordan2014second, jordan2008nonlinear} for a detailed insight into their derivation and physical background and~\cite{kaltenbacher2011wellposedness, KaltenbacherNikolic, bongarti2020vanishing, kaltenbacher2012well, pellicer2020uniqueness} for a selection of results on their mathematical analysis.\\
\indent However, a drawback of using the hyperbolic heat equation is that it may violate the second law
of thermodynamics; see, for example,~\cite{zhang2014time, fabrizio2017modeling, ferrillo2018comparing}. Fractional generalizations of the heat flux law have emerged in the literature as a way of interpolating between the properties of the two flux laws; see, e.g.,~\cite{povstenko2011fractional, compte1997generalized, fabrizio2015some, atanackovic2012cattaneo} and the references contained therein. In~\cite{compte1997generalized}, Compte and Metzler proposed several generalized time-fractional heat-flux laws in the following form:
\begin{equation} \label{fractional_law}
	\begin{aligned}
		(1+\tau^{\alpha_1} \Dt^{\alpha_1})\boldsymbol{q}(t) = -\kappa \Dt^{\alpha_2} \nabla \theta,
	\end{aligned}
\end{equation} 
where the choice of $(\alpha_1, \alpha_2)$ arises from a particular anomalous diffusion process in complex media. In the present work, we derive and analyze the time-fractional (J)MGT equations that arise from the use of fractional temperature laws \eqref{fractional_law} in place of the standard heat-flux law within the governing equations.\\
\indent One such model coming from the choice of fractional orders $(\alpha_1, \alpha_2)=(1, 1-\alpha)$ in the generalized Cattaneo law \eqref{fractional_law} is given by
\begin{equation} \label{general_fJMGT_III}
	\tau \psi_{ttt}+\psi_{tt}-c^2 \Delta \psi -(\tau c^2+\delta \Doal)\Delta \psi_{t}=f(\psi_t, \psi_{tt}, \nabla \psi, \nabla \psi_{t}), \ 0< \alpha \leq 1,
\end{equation}
whereas the choice  $(\alpha_1, \alpha_2)=(\alpha, 1-\alpha)$ leads to 
\begin{equation} \label{general_fJMGT_I}
	\tau^\alpha \Dal \psi_{tt}+\psi_{tt}-c^2 \Delta \psi -\tau^\alpha c^2 \Dal \Delta \psi- \delta \Dtal \D\psi=f(\psi_t, \psi_{tt}, \nabla \psi, \nabla \psi_{t}),
\end{equation}
where we assume that $\alpha \in (1/2, 1)$. We refer to Section~\ref{Sec:Modeling} below for the definition of $\textup{D}_t^\alpha$ and details on the modeling and to Tables~\ref{table:fJMGT} and \ref{table:fMGT} for a complete list of the fractional models that are considered in this work. In particular, we analyze the time-fractional JMGT equations in terms of local-in-time solvability and the limiting behavior of their solutions as $\alpha \rightarrow 1^{-}$.\\
\indent To the best of our knowledge, this is the first work dealing with the mathematical analysis of time-fractional MGT models. We point out that, on the other hand, (J)MGT equations with memories that involve smooth kernel functions represent an active field of research; see, e.g.,~\cite{lasiecka2017global, bucci2019regularity, dell2017moore, alves2018moore, dell2016moore, bounadja2020decay} and the references contained therein. \\
\indent Our exposition is organized as follows. In Section~\ref{Sec:Modeling} we derive four fractional versions of JMGT based on the four instances of \eqref{fractional_law} elaborated on in~\cite{compte1997generalized}. After a short Section~\ref{Sec:Preliminaries} with mathematical notation and tools, we first in Section~\ref{Sec:Analysis_fJMGT_W_III} focus on the version \eqref{general_fJMGT_III} of fixed highest order three. We prove its well-posedness in the linear as well as in the nonlinear case without gradient nonlinearity and justify the limit $\alpha\to 1^{-}$. Next, in Section~\ref{Sec:Analysis_fJMGT_W_others} we provide a similar analysis for the other models, that have in common a $2+\alpha$ leading derivative. This analysis works out with one exception, where the damping term is too weak to allow for varying coefficients or nonlinearities and whose linear version is analyzed separately in Section~\ref{Sec:WaveEq_Memory} based on a reformulation as a second-order wave equation. Before doing so, we return to \eqref{general_fJMGT_III} in Section~\ref{Sec:Analysis_fJMGT_K_III} and provide well-posedness and the limit $\alpha\to1^{-}$ in its full version, including the gradient nonlinearity, which requires higher-order energy estimates.
\section{Modeling with generalized heat-flux equations} \label{Sec:Modeling}

In this section, we consider the four general versions of the constitutive equation \eqref{fourier_law} proposed by Compte and Metzler in~\cite{compte1997generalized} and discuss the resulting acoustic equations. These time-fractional general flux equations (GFE) are as follows:
\begin{alignat}{3}
	\hspace*{-1.5cm}\text{\small (GFE)}\hphantom{III}&&\qquad \qquad  (1+\tau^\alpha \Dal)\boldsymbol{q}(t) =&&\,-\kappa \nabla \theta;\hphantom{\Doal}\\[1mm]
	\hspace*{-1.5cm}\text{\small(GFE I)}\hphantom{II}&& \qquad \qquad(1+\tau^\alpha \Dal)\boldsymbol{q}(t) =&&\, -\kappa \Doal \nabla \theta;\\[1mm]
	\hspace*{-1.5cm}\text{\small(GFE II)}\, \hphantom{I}&&\qquad \qquad (1+\tau^\alpha \Dal)\boldsymbol{q}(t) =&&\, -\kappa \Dt^{\alpha-1} \nabla \theta;\\[1mm]
	\hspace*{-1.5cm}\text{\small(GFE III)}\,\, && \qquad \qquad (1+\tau \partial_t)\boldsymbol{q}(t) =&&\, -\kappa \Doal \nabla \theta,
\end{alignat}
where $\boldsymbol{q}$ denotes the flux vector, $\theta$ the absolute temperature, and $\kappa$ is the thermal conductivity. A numerical study and comparison of the four resulting fractional heat equations has been performed in~\cite{zhang2014time} in a one-dimensional setting. Although they can all predict negative temperatures, the fractional heat equation based on using (GFE I) appears to avoid this nonphysical behavior for $\alpha \in (1/2, 1)$ close enough to $1/2$.    \\
\indent Note that while Compte and Metzler~\cite{compte1997generalized} state the equations using the Riemann--Liouville fractional derivative, in the present work $\Dt^{\gamma}$ always denotes the Caputo--Djrbashian fractional derivative:
	\[
	\Dt^{\gamma}w(t)=\frac{1}{\Gamma(1-\gamma)}\int_0^t (t-s)^{-\gamma}\Dt^{\lceil\gamma\rceil} w(s) \, \textup{d} s, \quad -1<\gamma <1;
	\]
	see, for example,~\cite[\S 1]{kubica2020time} and~\cite[\S 2.4.1]{podlubny1998fractional} for its definition.
	Here $n=\lceil\gamma\rceil$, $n\in\{0,1\}$ is the integer obtained by rounding up $\gamma$ and $\Dt^n$ is the zeroth or first derivative operator.  \\

%%%%%%%%%%%%%%%%%%%%%%%%%%%%%%%%%%%%%%%%%%%%%%%%%%%%%%%%%%%%%%%%%%%%%%%%%%%%%%%%%%%%%%%%%
\noindent {\small \bf (fJMGT)} We begin by discussing the modeling with the first option; that is
\begin{equation} \label{heat_flux_fractional_GFE}
	\begin{aligned}
		(1+\tau^\alpha \Dal)\boldsymbol{q}(t) = -\kappa \nabla \theta;
	\end{aligned}
\end{equation} 
cf.~\cite[Eq. (9)]{zhang2014time}. We note that this modification of the heat-flux law is introduced \emph{ad hoc} in~\cite{compte1997generalized} and then disregarded, however numerical studies of the resulting heat equation in~\cite{zhang2014time} incorporate it as well, and so we include it here. \\
\indent The derivation of the acoustic equation follows the steps taken in~\cite[\S 4]{jordan2014second} with now \eqref{heat_flux_fractional_GFE} in place of the Maxwell--Cattaneo law. This derivation employs a weakly-nonlinear approximation, which for our purposes can be restated as
\begin{equation} \label{weakly_nl_assumptions_alpha}
	\epsilon <<1,\quad \theta=O(\epsilon),\quad \tilde{K}=O(\epsilon), \quad \tau^\alpha=O(\epsilon), \quad |\mathfrak{e}|=O(\epsilon^2).
\end{equation}
Here $\epsilon$ is the Mach number, $\tilde{K}$ is the dimensionless thermal diffusivity, and $\mathfrak{e}$ the dimensionless entropy. Note that, compared to~\cite{jordan2014second}, the condition $\tau^\alpha=O(\epsilon)$ replaces $\tau=O(\epsilon)$ here. \\
\indent It is assumed that the sound wave propagates through a thermally conductive and relaxing liquid or gas with negligible viscosity. Starting from a one-dimensional setting, the governing system is first approximated by
\begin{equation} \label{eq_1_GFE}
	\begin{aligned}
		\psi_{tt}+\tfrac12 \epsilon \partial_t(\psi_x)^2-(1+(\gamma-2)s)[\psi_xs_x+(1+s)\psi_{xx}]=-\epsilon^{-1}\mathfrak{e}_t,
	\end{aligned}
\end{equation}
where $\psi$ is the acoustic velocity potential, $\gamma$ the adiabatic index, and $s$ is known as the condensation; see~\cite[Eq. (44)--(49) and Eq. (53)]{jordan2014second}. Upon employing $s \approx -\epsilon \psi_t$, one arrives at
\begin{equation} \label{eq_1}
	\begin{aligned}
		\psi_{tt}+\tfrac12 \epsilon \partial_t(\psi_x)^2-(1-(\gamma-2)\epsilon\psi_t)[-\epsilon \psi_x \psi_{tx}+(1-\epsilon \psi_t)\psi_{xx}]=-\epsilon^{-1}\mathfrak{e}_t;
	\end{aligned}
\end{equation}
cf. \cite[Eq. (49)]{jordan2014second}. From the entropy production law
$$\tilde{\kappa}{\mathfrak{e}_t}=-\tilde{K} \boldsymbol{q}_x,$$
with $\tilde{\kappa}$ being the dimensionless thermal conductivity,  and the general heat flux law \eqref{heat_flux_fractional_GFE} in a dimensionless version
\[
(1+\lambda_{\alpha} \Dal)\boldsymbol{q}(t) = -\tilde{\kappa} \nabla \theta,
\]
we then have the following entropy equation:
$$   (1+\lambda_{\alpha}  \Dal) \mathfrak{e}_t=\tilde{K}  \theta_{xx}.$$
After utilizing that $\theta \approx - \epsilon (\gamma-1) \psi_t$, we can rewrite it as
\begin{equation} \label{e_t_GFE}
	\begin{aligned}
		(1+\lambda_{\alpha}  \Dal)\mathfrak{e}_t=-\epsilon\tilde{K}(\gamma-1)\psi_{txx};
	\end{aligned}
\end{equation}
cf.~\cite[Eq. (57) and (58)]{jordan2008nonlinear}. Applying the relaxation operator $(1+\lambda_{\alpha} \Dal)$ to \eqref{eq_1} and using \eqref{e_t_GFE} to eliminate $\mathfrak{e}$ then leads to
\begin{equation} \label{eq_2_I}
	\begin{aligned}
		\begin{multlined}[t] (1+\lambda_{\alpha}  \Dal)\left\{\psi_{tt}+\tfrac12 \epsilon \partial_t(\psi_x)^2-(1-(\gamma-2)\epsilon \psi_t)[-\epsilon \psi_x \psi_{tx}+(1-\epsilon \psi_t)\psi_{xx}]\right\}\\
			= \,  \tilde{K}(\gamma-1) \psi_{txx}. \end{multlined}
	\end{aligned}
\end{equation}
With $\lambda_{\alpha} =O(\epsilon)$, by neglecting the $O(\epsilon^2)$ terms in the equation above, we arrive at
\begin{equation} \label{eq_3}
	\begin{aligned}
		(1+\lambda_{\alpha}  \Dal)\psi_{tt}-\lambda_{\alpha} \Dal \psi_{xx}-(1-\epsilon (\gamma-1)\psi_t)\psi_{xx}+\epsilon \partial_t(\psi_x)^2= \tilde{K}(\gamma-1) \psi_{txx}.
	\end{aligned}
\end{equation}
Dividing this equation by $(1-\epsilon (\gamma-1)\psi_t)$, using $(1-\epsilon (\gamma-1)\psi_t)^{-1} \approx 1+\epsilon (\gamma-1)\psi_t$ for $\epsilon<<1$ and neglecting all $O(\epsilon^2)$ terms, yields
\begin{equation} \label{final_eq_2}
	\begin{aligned}
		\begin{multlined}[t]\lambda_{\alpha}  \Dal \psi_{tt}+(1+\epsilon (\gamma-1)\psi_t)\psi_{tt}-\psi_{xx}-\lambda_{\alpha}  \Dal \psi_{xx}\\\hspace*{5cm}-\tilde{K}(\gamma-1) \psi_{txx}+\epsilon \partial_t(\psi_x)^2= 0. \end{multlined}
	\end{aligned}
\end{equation}
Extrapolating to a dimensionalized 3D model in a mathematically general form gives
\begin{equation} \label{fJMGTK}
	\begin{aligned}
		\begin{multlined}[t]\tau^\alpha \Dal \psi_{tt}+(1+2\tilde{k}\psi_t)\psi_{tt}-c^2 \Delta \psi -\tau^\alpha c^2 \Dal \Delta \psi- \delta \D\psi_{t}+ \elltil\partial_t |\nabla \psi|^2=0. \end{multlined}
	\end{aligned}
\end{equation}
Since the quadratic gradient nonlinearity present in this model corresponds to the one in the second-order Kuznetsov equation~\cite{kuznetsov1971equations}, we will henceforth refer to \eqref{fJMGTK} as the fractional JMGT--Kuznetsov equation, or the fJMGT--K equation for short.  \\
\indent Assuming local nonlinear effects can be neglected so that 
\begin{equation}
	|\nabla \psi|^2 \approx c^{-2}\psi_t^2,
\end{equation}
we obtain
\begin{equation} \label{fJMGTW}
	\begin{aligned}
		\begin{multlined}[t]
			\tau^\alpha \Dal \psi_{tt}+(1+2k\psi_t)\psi_{tt}-c^2 \Delta \psi -\tau^\alpha c^2 \Dal \Delta \psi- \delta \D\psi_{t}=0.
		\end{multlined}
	\end{aligned}
\end{equation}
The above approximation corresponds to the one commonly used when reducing the Kuznetsov equation to the Westvervelt second-order model of nonlinear acoustics; cf.~\cite[\S 2.3]{jordan2016survey}. For this reason, we will refer to \eqref{fJMGTW} as the fractional Jordan--Moore--Gibson--Thompson--Westervelt equation, or the fJMGT--W equation for short. This approximation is appropriate when cumulative nonlinear effects dominate the local ones, which is the case, e.g., for sound propagation sufficiently far from the source in terms of wavelengths; see the discussion in \cite[Ch. 3, Section 6]{hamilton1998nonlinear}.\\

%%%%%%%%%%%%%%%%%%%%%%%%%%%%%%%%%%%%%%%%%%%%%%%%%%%%%%%%%%%%%%%%%%%%%%%%%%%%%%%%%%%%%%%%%%%%%%%%%%%%%%%%%%%%%%%%%%%%
\noindent {\small \bf (fJMGT I)} As a second option, we employ the general heat-flux model given by
\begin{equation} \label{heat_flux_fractional}
	\begin{aligned}
		(1+\tau^\alpha \Dal)\boldsymbol{q}(t) = -\kappa \Doal \nabla \theta;
	\end{aligned}
\end{equation} 
see~\cite[Eq. (14)]{compte1997generalized} and~\cite[Eq. (10)]{zhang2014time}. The use of this flux law is motivated in~\cite{compte1997generalized} stochastically by fractal time random walks. Retracing the derivation steps from before leads to the following equation:
\begin{equation} \label{eq_2_II}
	\begin{aligned}
		\begin{multlined}[t] (1+\lambda_{\alpha} \Dal)\left\{\psi_{tt}+\tfrac12 \epsilon \partial_t(\psi_x)^2-(1-(\gamma-2)\epsilon \psi_t)[-\epsilon \psi_x \psi_{tx}+(1-\epsilon \psi_t)\psi_{xx}]\right\}\\
			= \,  \tilde{K}(\gamma-1)  \Doal\psi_{txx} \end{multlined}
	\end{aligned}
\end{equation}
in place of \eqref{eq_2_I}. Neglecting the $O(\epsilon^2)$ terms then yields
\begin{equation} \label{eq_3}
	\begin{aligned}
		(1+\lambda_{\alpha} \Dal)\psi_{tt}-(1-\epsilon (\gamma-1)\psi_t)\psi_{xx}- \lambda_{\alpha} \Dal \psi_{xx}+\epsilon \partial_t(\psi_x)^2= \tilde{K}(\gamma-1)   \Doal\psi_{txx}.
	\end{aligned}
\end{equation}
Analogously to before, dividing by $(1-\epsilon (\gamma-1)\psi_t)$ leads to
\begin{equation} 
	\begin{aligned}
		\begin{multlined}[t]\lambda_{\alpha} \Dal \psi_{tt}+(1+\epsilon (\gamma-1)\psi_t)\psi_{tt}-\psi_{xx}-\lambda_{\alpha} \Dal \psi_{xx}-\tilde{K} (\gamma-1)  \Doal\psi_{txx}\\+\epsilon \partial_t(\psi_x)^2= 0 \end{multlined}
	\end{aligned}
\end{equation}
in place of \eqref{final_eq_2}. Then extrapolating to a general 3D equation gives
\begin{equation} \label{fJMGTK_I}
	\begin{aligned}
		\begin{multlined}[t]\tau^\alpha \Dal \psi_{tt}+(1+2\tilde{k}\psi_t)\psi_{tt}-c^2 \Delta \psi -\tau^\alpha c^2 \Dal \Delta \psi- \delta \Doal \D\psi_{t}+  \elltil\partial_t |\nabla \psi|^2=0, \end{multlined}
	\end{aligned}
\end{equation}
which we will call the fractional Jordan--Moore--Gibson--Thompson--Kuznetsov equation of type I, or the fJMGT--K I equation for short. By assuming local nonlinear effects can be neglected as before, we arrive at
\begin{equation} \label{fJMGTW_I}
	\begin{aligned}
		\begin{multlined}[t]
			\tau^\alpha \Dal \psi_{tt}+(1+2k\psi_t)\psi_{tt}-c^2 \Delta \psi -\tau^\alpha c^2 \Dal \Delta \psi- \delta \Doal \D\psi_{t}=0,
		\end{multlined}
	\end{aligned}
\end{equation}
which we will refer to as the fractional Jordan--Moore--Gibson--Thompson--Westervelt equation of type I, or just the fJMGT--W I equation. \\

\noindent {\small \bf (fJMGT II)}
Thirdly, we employ the heat-flux model given by
\begin{equation} \label{heat_flux_fractional}
	\begin{aligned}
		(1+\tau^\alpha \Dal)\boldsymbol{q}(t) = -\kappa \Dt^{\alpha-1} \nabla \theta;
	\end{aligned}
\end{equation} 
cf.~\cite[Eq. (14)]{compte1997generalized} and \cite[Eq. (11)]{zhang2014time}. This flux law is motivated in~\cite{compte1997generalized} by nonlocal transport theory with memory effects; that is, a nonlocal relation between the flux $\boldsymbol{q}$ and temperature $\theta$:
\[
\boldsymbol{q}(x,t)= \int_0^t \kernel(t-s)\nabla \theta (x,s)\ds,
\]
with a suitable choice of the kernel. Analogously to before, we can derive the following general fractional model:
\begin{equation} \label{fJMGTK_II}
	\begin{aligned}
		\begin{multlined}[t]\tau^\alpha \Dal \psi_{tt}+(1+2\tilde{k}\psi_t)\psi_{tt}-c^2 \Delta \psi -\tau^\alpha c^2 \Dal \Delta \psi- \delta \Dal \D\psi+\elltil\partial_t |\nabla \psi|^2=0, \end{multlined}
	\end{aligned}
\end{equation}
which we will from now on refer to as the fractional Jordan--Moore--Gibson--Thompson--Kuznetsov equation of type II, or the fJMGT--K II equation for short.  If the local nonlinear effects can be neglected, we obtain
\begin{equation} \label{fJMGTW_II}
	\begin{aligned}
		\tau^\alpha \Dal \psi_{tt}+(1+2k\psi_t)\psi_{tt}-c^2 \Delta \psi -(\tau^\alpha c^2 +\delta)\Dal \Delta \psi=0,
	\end{aligned}
\end{equation}
which we will refer to as the fractional Jordan--Moore--Gibson--Thompson--Westervelt equation of type II, or the fJMGT--W II equation for short.\\

%%%%%%%%%%%%%%%%%%%%%%%%%%%%%%%%%%%%%%%%%%%%%%%%%%%%%%%%%%%%%%%%%%%%%%%%%%%%%%%%%%%%%%%%%
\noindent {\small \bf (fJMGT III)} Finally, we consider the wave-like acoustic models resulting from using the following flux law:
\begin{equation} \label{heat_flux_fractional}
	\begin{aligned}
		(1+\tau \partial_t)\boldsymbol{q}(t) = -\kappa \Doal \nabla \theta;
	\end{aligned}
\end{equation} 
see~\cite[Eq. (18)]{compte1997generalized} and~\cite[Eq. (12)]{zhang2014time}. In~\cite{compte1997generalized}, this law is motivated by a delayed equation that may connect the flux to a generalized force
\[
\boldsymbol{q}(t+\tau) = -\kappa \Doal \nabla \theta.
\]
Here, weakly-nonlinear acoustic approximation is based on assuming that 
\begin{equation} \label{weakly_nl_assumptions}
	\epsilon <<1,\quad \theta=O(\epsilon),\quad \tilde{K}=O(\epsilon), \quad \tau=O(\epsilon), \quad |\mathfrak{e}|=O(\epsilon^2).
\end{equation}
Retracing our previous derivation steps then quickly leads to
\begin{equation} 
	\begin{aligned}
		\begin{multlined}[t]\tau \psi_{ttt}+(1+\epsilon (\gamma-1)\psi_t)\psi_{tt}-\psi_{xx}-\tau \psi_{txx}-\tilde{K} (\gamma-1)  \Doal\psi_{txx}\\+\epsilon \partial_t(\psi_x)^2= 0. \end{multlined}
	\end{aligned}
\end{equation}
Extrapolating to a dimensionalized 3D model yields
\begin{equation} \label{fJMGTK_III}
	\begin{aligned}
		\begin{multlined}[t]\tau \psi_{ttt}+(1+2\tilde{k}\psi_t)\psi_{tt}-c^2 \Delta \psi -\tau c^2 \Delta \psi_{t}- \delta \Doal\D\psi_{t}+ \elltil\partial_t |\nabla \psi|^2=0, \end{multlined}
	\end{aligned}
\end{equation}
which we will henceforth refer to it as the fractional Jordan--Moore--Gibson--Thompson--Kuznetsov equation of type III, or the fJMGT--K III equation for short. If the local nonlinear effects can be neglected, we obtain
\begin{equation} \label{fJMGTW_III}
	\begin{aligned}
		\begin{multlined}[t]
			\tau \psi_{ttt}+(1+2k\psi_t)\psi_{tt}-c^2 \Delta \psi -\tau c^2\Delta \psi_{t}- \delta \Doal\D\psi_{t}=0.
		\end{multlined}
	\end{aligned}
\end{equation}
We will refer to this model as the fractional Jordan--Moore--Gibson--Thompson--Westervelt equation of type III, or the fJMGT--W III equation for short. \\

\indent We collect all discussed time-fractional acoustic equations in Table~\ref{table:fJMGT} for convenience and state them with a general source function $f$. Note that the constant $\delta>0$ for models I--III no longer has the dimension of usual sound diffusivity. \\
\indent We assume that $\alpha \in (0,1]$ in the fJMGT II and III equations, whereas we perform the analysis of the fJMGT and fJMGT I equations  under the assumption that $\alpha \in (1/2,1]$.   Formally letting $\alpha \rightarrow 1^{-}$ in these equations leads to the Jordan--Moore--Gibson--Thompson equations, either in the Westervelt or Kuznetsov forms; cf.~\cite{jordan2014second}.

\begin{table}[h]
	\captionsetup{width=.94\linewidth}
	\begin{center} \small
		\begin{tabular}{|m{1.3cm}||m{10.7cm}|}
			\hline
			\vspace*{2mm}
			{fJMGT--}                                     & \vspace*{2mm} \textbf{Nonlinear time-fractional acoustic equations}                                                                                                      \\[6pt]
			\Xhline{2\arrayrulewidth} \hline \vspace*{4mm}
			\centering	K & $\tau^\alpha \Dal \psi_{tt}+(1+2\tilde{k}\psi_t)\psi_{tt}-c^2 \Delta \psi -\tau^\alpha c^2 \Dal \Delta \psi- \delta \D\psi_{t}+ \elltil\partial_t |\nabla \psi|^2=f$   \\[3mm]
			\Xhline{0.02\arrayrulewidth} \vspace*{3mm}
			\centering	W     & $\tau^\alpha \Dal \psi_{tt}+(1+2k\psi_t)\psi_{tt}-c^2 \Delta \psi -\tau^\alpha c^2 \Dal \Delta \psi- \delta \D\psi_{t}=f$                                                \\[3mm] \hline\hline
			\vspace*{4mm}
			\centering	K I                                & $\tau^\alpha \Dal \psi_{tt}+(1+2\tilde{k}\psi_t)\psi_{tt}-c^2 \Delta \psi -\tau^\alpha c^2 \Dal \Delta \psi- \delta \Dtal \D\psi+ \elltil\partial_t |\nabla \psi|^2=f$ \\[3mm]
			\Xhline{0.02\arrayrulewidth} \vspace*{3mm}
			\centering	W I   & $\tau^\alpha \Dal \psi_{tt}+(1+2k\psi_t)\psi_{tt}-c^2 \Delta \psi -\tau^\alpha c^2 \Dal \Delta \psi- \delta \Dtal \D\psi=f$                                              \\[3mm] \hline\hline
			\vspace*{4mm}
			\centering	K II                               & $\tau^\alpha \Dal \psi_{tt}+(1+2\tilde{k}\psi_t)\psi_{tt}-c^2 \Delta \psi -\tau^\alpha c^2 \Dal \Delta \psi- \delta \Dal \D\psi+\elltil\partial_t |\nabla \psi|^2=f$   \\[2mm]
			\Xhline{0.02\arrayrulewidth} \vspace*{3mm}
			\centering	W II  & $\tau^\alpha \Dal \psi_{tt}+(1+2k\psi_t)\psi_{tt}-c^2 \Delta \psi -(\tau^\alpha c^2 +\delta)\Dal \Delta \psi=f$                                                          \\[3mm] \hline\hline
			\vspace*{4mm}
			\centering K III                              & $\tau \psi_{ttt}+(1+2\tilde{k}\psi_t)\psi_{tt}-c^2 \Delta \psi -\tau c^2 \Delta \psi_{t}- \delta \Dtal\D\psi+ \elltil\partial_t |\nabla \psi|^2=f$                     \\[3mm]
			\Xhline{0.02\arrayrulewidth} \vspace*{3mm}
			\centering	W III & $\tau \psi_{ttt}+(1+2k\psi_t)\psi_{tt}-c^2 \Delta \psi -\tau c^2\Delta \psi_{t}- \delta \Dtal\D\psi=f$                                                                   \\[3mm] \hline
		\end{tabular}
	\end{center}
	\caption{Nonlinear fJMGT models in the Kuznetsov and Westervelt forms.} \label{table:fJMGT}
\end{table} 
We will also study the linearizations of these equations (obtained by setting $k=\tilde{k}=\elltil=0$), which we will refer to as fractional Moore--Gibson--Thompson (fMGT) equations; cf. Table~\ref{table:fMGT} below.
%%%%%%%%%%%%%%%%%%%%%%%%%%%%%%%%%%%%%%%%%%%%%%%%%%%%%%%
\section{Theoretical preliminaries} \label{Sec:Preliminaries}
In this section, we gather several theoretical results from fractional calculus that will be useful later on. To simplify the notation, we often omit the spatial domain and the time interval when writing norms; for example, $\|\cdot\|_{W_t^{p,q} (L^r)}$ denotes the norm on $W^{p,q}(0,t;L^r(\Omega))$ and  $\|\cdot\|_{W^{p,q} (L^r)}$ denotes the norm on $W^{p,q}(0,T;L^r(\Omega))$.\\
\indent Throughout the paper, we assume that $\Om \subset \R^{n}$ is an open, bounded, and sufficiently smooth set, where $n \in \{1, 2, 3\}$. When writing solution spaces for $\psi$, we use the following notational convention:
\begin{equation} \label{sobolev_withtraces}
\begin{aligned}
\Honetwo=&\,H_0^1(\Omega)\cap H^2(\Omega),\\
\Honethree=&\, \left\{\psi \in H^3(\Omega)\,:\, \mbox{tr}_{\partial\Omega} \psi = 0, \  \mbox{tr}_{\partial\Omega} \D \psi = 0\right\}.
\end{aligned}
\end{equation}
In the analysis, we will rely on the continuous embeddings 
$H^1(\Omega)\hookrightarrow L^4(\Omega)$ and $H^2(\Omega) \hookrightarrow L^\infty(\Omega)$:
\begin{equation}\label{embeddigs}
\begin{aligned}
&\|v\|_{L^4(\Omega)}\leq \CHone \|\nabla v\|_{L^2(\Omega)}, \quad && v \in H^1_0(\Omega)\\
&\|v\|_{L^\infty(\Omega)}\leq C_{H^2, L^\infty} \|\D v\|_{L^2(\Omega)}, \quad && v \in \Honetwo.
\end{aligned}
\end{equation}
\indent We often write $x \lesssim y$ instead of $x \leq C y$. In such cases, $C>0$ represents a generic constant that may depend on the medium parameters and the final time $T$, but does not depend on the order of differentiation $\alpha$. \\
\indent Throughout the paper we make the following assumptions on the (constant) medium parameters:
\begin{equation}
\tau >0, \quad c>0, \quad \delta>0, \quad k,\, \tilde{k},\, \elltil\, \in \R. 
\end{equation}
\subsection*{Coercivity estimates}
When performing energy analysis, we will rely on the following two coercivity estimates.
\begin{itemize}
	\item
	\cite[Lemma 1]{Alikhanov:11}: For any absolutely continuous function $w$,
	\begin{equation}\label{eqn:Alikhanov_1}
	{w(t)}\textup{D}_t^{\gamma}w(t)\geq \tfrac12(\textup{D}_t^{\gamma} w^2)(t).
	\end{equation}
\end{itemize}
\begin{itemize}	
	\item
	\cite[Lemma 2.3]{Eggermont1987}; see also \cite[Theorem 1]{VoegeliNedaiaslSauter2016}: For any $w\in H^{-(1-\alpha)/2}(0,t)$,
	\begin{equation}\label{coercivityI}
	\int_0^t \langle \textup{I}^{1-\alpha} w(s),  w(s) \rangle \ds \geq \cos ( \tfrac{\pi(1-\alpha)}{2} ) \| w \|_{H^{-(1-\alpha)/2}(0,t)}^2, 
	\end{equation}	 
	where $\textup{I}^{\gamma}$ denotes the Abel integral operator:
	\[
	\textup{I}^{\gamma}w (t)= \frac{1}{\Gamma(\gamma)}\int_0^t (t-s)^{\gamma-1} w(s)\ds, \quad \gamma \in (0,1)
	\]
and the negative norm is defined by
\[
\| w \|_{H^{-\gamma}(0,t)}^2 = \int_{\mathbb{R}} (1+\omega^2)^{-\gamma}|\hat{w}(\omega)|^2\, \textup{d} \omega, \quad \gamma >0,
\]
with $\hat{w}$ being the Fourier transform of the extension by zero of $w$ to all of $\mathbb{R}$.
\end{itemize}
\subsection*{The Kato--Ponce inequality} The following product rule estimate holds:
\begin{equation}\label{prodruleest}
\|fg\|_{W^{\rho,r}(0,T)}\lesssim \|f \|_{W^{\rho,p_1}(0,T)} \|g\|_{L^{q_1}(0,T)}
+ \| f \|_{L^{p_2}(0,T)} \|g\|_{W^{\rho,q_2}(0,T)}
\end{equation}
for $0\leq\rho\leq\overline{\rho}<1$, 
$1<r<\infty$, $p_1$, $p_2$, $q_1$, $q_2\in(1,\infty]$, with $\frac{1}{r}=\frac{1}{p_i}+\frac{1}{q_i}$, $i=1,2$; see, e.g., \cite{GrafakosOh2014}. 
We note that the fractional derivative $\partial_t^{\rho}=(I-D_{tt})^{\rho/2}$ 
used in \cite{GrafakosOh2014} is an isomorphism between $H^\rho(0,T)$ and $L^2(0,T)$ for any $\rho\in[0,1]$.

\subsection*{Limits with respect to the order of differentiation}
Let $w: [0,T] \mapsto X$, where $X$ is a Banach space equipped with the norm $\|\cdot\|_X$. It is known that the mapping $\alpha\mapsto \Dal w(t)$ is left-sided continuous at any integer $\alpha$, while it is  discontinuous from the right side unless $w(0)=0$. This property is due to the following two identities that follow by integration by parts:
\begin{equation}\label{limitalphaC2}
\begin{aligned}
&\lim_{\alpha \rightarrow 1^{-}}\left\|\Dal w(t) -w_t(t)\right\|_X\\
=& \lim_{\alpha \rightarrow 1^{-}} \left\|\frac{1}{\Gamma(1-\alpha)}\int_0^t  (t-s)^{-\alpha}  w_t(s) \ds -w_t(t) \right\|_X\\
=&\lim_{\alpha \rightarrow 1^{-}} \left\|\frac{t^{1-\alpha}w_t(0)}{\Gamma(2-\alpha)}+\frac{1}{\Gamma(2-\alpha)}\int_0^t (t-s)^{1-\alpha} w_{tt}(s) \ds -w_t(t) \right\|_X \\
=&\, \left\| w_t(0)+\int_0^t w_{tt}(s)\ds -w_t(t)\right\|_X 
=\, 0,
\end{aligned}
\end{equation}
as well as
\begin{equation}
\begin{aligned}
\lim_{\alpha \rightarrow 0^+}\left\|\Dal w(t) -w(t)\right\|_X =&\, \lim_{\alpha \rightarrow 0^+} \left\|\frac{1}{\Gamma(1-\alpha)}\int_0^t  (t-s)^{-\alpha}  w_t(s) \ds -w(t) \right\|_X\\
=&\, \left\| w(0)\right\|_X;
\end{aligned}
\end{equation}
see, for example,~\cite[\S 2.4.1]{podlubny1998fractional}. These limits extend to $L^p(0,T;X)$ under relaxed smoothness assumptions on $w$. Here and below we use the notation 
\begin{align} \label{kersing_notation}
\kersing_\alpha(t)=\frac{t^{-\alpha}}{\Gamma(1-\alpha)},\quad \kersing_0(t)=1.
\end{align}
\begin{lemma}\label{Lemma:Limit}
Let $1 \leq p < \infty$. For any $w\in W^{1,p}(0,T;X){\, \cap \, W^{2,1}(0,T;X)}$, it holds that
\begin{equation}\label{limitalphaL2_1}
\limsup_{\alpha \rightarrow 1^{-}}\left\|\Dal w -w_t\right\|_{L^p(0,T;X)}=0.
\end{equation}
For any $w\in W^{1,1}(0,T;X)$,
\begin{equation}\label{limitalphaL2_2}
\Dal w \rightharpoonup w_t \mbox{ in } L^1(0,T;X)
\mbox{ as }\alpha \rightarrow 1^{-},
\end{equation}
\noindent and 
\begin{equation}\label{limitalphaL2}
\limsup_{\alpha \rightarrow 0^+}\left\|\Dal w -w+w(0)\right\|_{L^p(0,T;X)} =0\,. 
\end{equation}
\end{lemma}
\begin{proof} 
To prove \eqref{limitalphaL2}, we rely on Young's convolution inequality and Lebesgue's Dominated Convergence theorem, which yields 
\[ \|\kersing_\alpha-\kersing_0\|_{L^p(0,T)}\rightarrow 0 \ \text{as} \ \alpha\to0^+,\]  with the $L^1$ dominating function $t\mapsto \dfrac{t^{-p\alpha_0}}{\Gamma(1-\alpha_0)^p}$ for $\alpha\leq\alpha_0<1/p$ (without loss of generality since $\alpha$ tends to zero). Therefore, \eqref{limitalphaL2} holds via 
\[
\begin{aligned}
\limsup_{\alpha \rightarrow 0^+}\left\|\Dal w -w+w(0)\right\|_{L^p(0,T;X)} =&\, 
\limsup_{\alpha \rightarrow 0^+} \left\|(\kersing_\alpha-\kersing_0)*w_t \right\|_{L^p(0,T;X)}\\
\leq&\,\limsup_{\alpha \rightarrow 0^+} \left\|\kersing_\alpha-\kersing_0\right\|_{L^p(0,T)} \left\|w_t \right\|_{L^1(0,T;X)}\,=0.
\end{aligned}
\]
To show \eqref{limitalphaL2_2}, we use the fact that for any $w\in W^{1,1}(0,T;X)$ and $v\in L^\infty(0,T;X^*)$ where $X^*$ is the dual space of $X$, 
\begin{equation}\label{testv}
\begin{aligned}
&\int_0^T \left\langle \int_0^t \kersing_\alpha(t-s)w_t(s)\ds -w_t(t), v(t)\right\rangle_X\dt\\
=&\, \int_0^T \int_0^t \langle \kersing_\alpha(t-s)w_t(s)\ds, v(t)\rangle_X\dt -\int_0^T \langle w_t(t), v(t)\rangle_X\dt\\
=&\, \int_0^T \left\langle w_t(s), \int_s^T   \kersing_\alpha(t-s) v(t) \dt \right\rangle_X  \ds-\int_0^T \langle w_t(r), v(r)\rangle_X\dt\\
=&\, \int_0^T \left\langle w_t(r), \Bigl(\int_r^T   \kersing_\alpha(t-r) v(t) \dt -v(r)\Bigr)\right\rangle_X \textup{d}r
\ =: \int_0^T \psi(r) \textup{d}r,
\end{aligned}
\end{equation}
where $\langle\cdot,\cdot\rangle_X$ denotes the dual pairing. From the identity
\begin{equation}
\begin{aligned}
&\int_0^T \left\langle \phi(r), \Bigl(\int_r^T   \kersing_\alpha(t-r) v(t) \dt -v(r)\Bigr)\right\rangle_X\\
=&\,\int_0^T \left\langle \int_0^t \kersing_\alpha(t-s)\phi(s)\ds -\phi(t), v(t)\right\rangle_X\dt\\
=&\,\int_0^T \frac{1}{\Gamma(2-\alpha)}\left\langle t^{1-\alpha} \phi(0)+\int_0^t (t-s)^{1-\alpha} \phi_t(s) \ds -\phi(t), v(t)\right\rangle_X\dt \ \to0\mbox{ as }\alpha\to1^-
\end{aligned}
\end{equation}
for any $\phi\in C^1(0,T;X)$, we conclude that $\Bigl(\int_r^T \kersing_\alpha(t-r) v(t) \dt -v(r)\Bigr)$ tends to zero pointwise almost everywhere on $(0,T)$ and therefore so does the integrand $\psi$ in \eqref{testv}. \\
\indent On the other hand, $\psi$ can be bounded by the $L^1(0,T)$ function 
\[r\mapsto (C_T+1)\|w_t(r)\|_X \|v\|_{L^\infty(0,T;X^*)},\] where we have estimated 
\begin{equation}
\begin{aligned}
&\left\|\int_r^T \kersing_\alpha(t-r) v(t) \dt \right\|_{X^*} \\
\leq&\, \|v\|_{L^\infty(0,T;X^*)} \int_r^T \kersing_\alpha(t-r)\, \textup{d}r
= \|v\|_{L^\infty(0,T;X^*)} \frac{(T-r)^{1-\alpha}}{\Gamma(2-\alpha)} 
\end{aligned}
\end{equation}
and set \[C_T:=\sup_{\alpha\in(\frac12,1]}\frac{T^{1-\alpha}}{\Gamma(2-\alpha)}<\infty.\]
Thus by Lebesgue's Dominated Convergence theorem, \eqref{testv} tends to zero as $\alpha\to1^-$. Since $v\in L^\infty(0,T;X^*)$ is arbitrary, we obtain weak convergence of $\Dal w -w_t$ to zero in $L^1(0,T;X)$.
If $X$ is a Hilbert space, this computation can be repeated replacing dual pairings with inner products and yields weak convergence in the Hilbert space sense.
\\
\indent Finally, in case $w\in W^{1,p}(0,T;X)\cap W^{2,1}(0,T;X)$, \eqref{limitalphaL2_1} follows from 
\[
\begin{aligned}
&\limsup_{\alpha \rightarrow 1^{-}}\left\|\Dal w -w_t\right\|_{L^p(0,T;X)}\\
=&\, \limsup_{\alpha \rightarrow 1^{-}} \Bigl(\int_0^T{\|(\kersing_{\alpha}*w_t)(t)-w_t(t)\|_X}^p\dt\Bigr)^{1/p}=0\,,
\end{aligned}
\]
where we have again used Lebesgue's Dominated Convergence theorem.
Indeed, $(\kersing_{\alpha}*w_t)-w_t$ (and therefore the $p$th power of its norm) tends to zero almost everywhere on $(0,T)$ as $\alpha\to1^{-}$, on account of the previous discussion.
Furthermore, 
\[\|(\kersing_{\alpha}*w_t)-w_t\|_X^p\leq 2^{p-1}\|\kersing_\alpha*w_t\|_X^p + 2^{p-1}\|w_t\|_X^p\] %(Young's Convolution Ineq)
can be further dominated by the $L^1(0,T)$ function 
\[2^{p}\max\{1,T\}^p \Bigl(\|w_t(0)\|_X+ \|w_{tt}\|_{L^1(0,T;X)}\Bigr)^p+ 2^{p-1}\|w_t\|_X^p\]
due to 
\[
\begin{aligned}
&\|(\kersing_{\alpha}*w_t)(t)\|_X=\frac{t^{1-\alpha}}{\Gamma(2-\alpha)}\left\|w_t(0)
+\int_0^t (1-\tfrac{s}{t})^{1-\alpha} w_{tt}(s) \ds\right\|_X\\
&\leq C_T \left(\|w_t(0)\|_X+ \|w_{tt}\|_{L^1(0,T;X)}\right);
\end{aligned}
\]
cf. \eqref{limitalphaC2} with $C_T$ as above.
\end{proof}
%%%%%%%%%%%%%%%%%%%%%%%%%%%%%%%%%%%%%%%%%%%%%%%%%%%%%%%
\section{Analysis of the equation with the third-order leading term} \label{Sec:Analysis_fJMGT_W_III}
We begin our analytical considerations by looking at the fJMGT--W III equation 
\begin{equation}
\tau \psi_{ttt}+(1+2k\psi_t)\psi_{tt}-c^2 \Delta \psi -\tau c^2\Delta \psi_{t}- \delta \Dtal\D\psi=f,
\end{equation}
\noindent which has an integer-order leading term; cf. Table~\ref{table:fJMGT}. We intend to analyze it by setting up a fixed-point mapping. To this end, we, first of all, study the following linearization:
\begin{equation} \label{fMGT_III_sigma}
\tau \psi_{ttt} + (1+\sigma(x,t)) \psi_{tt} - c^2\Delta\psi -\tau c^2\Delta \psi_t - \delta \Dtal \Delta \psi = f, \quad 0< \alpha \leq 1.
\end{equation}
This is the linear fMGT III equation with a variable coefficient, which we assume is uniformly bounded; cf. Table~\ref{table:fMGT} below. More precisely, we assume that there exist $\underline{\sigma}$, $\overline{\sigma}>0$, such that
\begin{equation}\label{nondegeneracy_sigma}
\underline{\sigma} \leq \sigma(x,t)\leq \overline{\sigma} \ \mbox{ for all } \ x\in\Omega, \,  t\in(0,T).
\end{equation}
Note that since we study the local-in-time behavior in this work, we do not impose a non-degeneracy condition on $1+\sigma$. In the upcoming analysis, the crucial estimate involving fractional derivatives will be the following:
\begin{equation} \label{fractional_est_w}
\begin{aligned}
\int_0^t \prodLtwo{\Dtal w}{w_{tt}} \ds 
\geq&\,     \cos(\pi \alpha/2)    \|w_{tt} \|_{H^{-\alpha/2}(0,t; L^2(\Omega))}^2,
\end{aligned}
\end{equation}
which follows by \eqref{coercivityI}. To formulate the first well-posedness result, we introduce the solution space
\begin{equation} \label{def_X_fMGTIII}
\begin{aligned}
\quad X^{\textup{low}}_{\textup{fMGT III}}=\,\left\{\vphantom{H^{-\alpha/2}(0,T; H_0^1(\Om))}\right.&\psi \in L^{\infty}(0,T; H_0^1(\Omega)): \psi_t \in L^{\infty}(0,T; H_0^1(\Omega)),\\&\, \psi_{tt} \in L^\infty(0,T; L^2(\Omega)) \cap H^{-\alpha/2}(0,T; H_0^1(\Om)),\\&\left.\, \textup{I}^{1-\alpha} \|\nabla \Doal \pt\|^2_{L^2} \in L^\infty(0,T), \ \psi_{ttt} \in L^2(0,T; H^{-1}(\Omega)) \vphantom{H^{-\alpha/2}(0,T; H_0^1(\Om))}\right\}  
\end{aligned}
\end{equation}
for $ \alpha \in (0,1)$, and
\begin{equation} \label{def_X_fMGTIII_alphaone}
\begin{aligned}
X^{\textup{low}}_{\textup{fMGT III}}=  W^{1, \infty}(0,T; H_0^1(\Omega)) \cap W^{2, \infty}(0,T; L^2(\Om))\cap H^3(0,T; H^{-1}(\Om)) 
\end{aligned}
\end{equation}
in case $\alpha=1$. We denote by $\|\cdot\|_{X^{\textup{low}}_{\textup{fMGT III}}}$ the corresponding norm on this space. We claim that the fMGT III equation \eqref{fMGT_III_sigma} has a unique solution in this space under suitable assumptions on the data and the variable coefficient.
\begin{proposition}[Well-posedness of the fMGT III equation] \label{Prop:fMGT_III_lower} Let $\alpha \in (0,1]$ and $\sigma \in L^\infty(0,T; L^\infty(\Omega))$. Given $f \in L^2(0,T; L^2(\Om))$ and \[(\psi_0, \psi_1, \psi_2) \in (H_0^1(\Om), H_0^1(\Om), L^2(\Om)),\] there exists a unique $\psi$ in $X^{\textup{low}}_{\textup{fMGT III}}$, such that
	\begin{equation} \label{fMGT1_sigma}
	\begin{aligned}
	\begin{multlined}[t] \langle \tau \psi_{ttt}, v\rangle_{H^{-1}, H_0^1} + \prodLtwo{(1+\sigma) \psi_{tt}}{v} \\+ \prodLtwo{c^2\nabla\psi +\tau c^2\nabla \psi_t +\delta \Dtal \nabla \psi}{\nabla v} = \prodLtwo{f}{v} \end{multlined}
	\end{aligned}
	\end{equation}
	for all $v \in H_0^1(\Omega)$, a.e. in $(0,T)$, with $(\psi, \psi_t, \psi_{tt})\vert_{t=0}=(\psi_0, \psi_1, \psi_2)$. Furthermore, the solution satisfies the following estimate:
	\begin{equation} \label{energy_est1_fJMGT_W_III}
	\begin{aligned}
	&\,\begin{multlined}[t]
	\|\psi\|^2_{W^{1,\infty}(H^1)}+\nLinfLtwo{\ptt}^2 +\|\pttt\|^2_{L^2(H^{-1})}\\[2mm] \hspace*{1cm}+\cos(\alpha \pi/2)\|\ptt\|^2_{H^{-\alpha/2}(H^1)}+ \sup_{t \in (0,T)} \textup{I}^{1-\alpha} \|\nabla \Doal \pt\|^2_{L^2}
	\end{multlined} \\
	\lesssim&\, \nLtwoLtwo{f}^2+\nLtwo{\nabla \psi_0}^2+\nLtwo{\nabla \psi_1}^2+\nLtwo{\psi_2}^2,
	\end{aligned}
	\end{equation}	
where for $\alpha=1$, the $\cos(\alpha \pi/2)$ term should be omitted.	
\end{proposition}
\begin{proof}
	We focus in the proof on the case $\alpha \in (0,1)$ since the case $\alpha=1$ follows in a more straightforward manner. We perform the analysis by employing the standard Galerkin procedure to discretize the problem in space~\cite[\S 7]{evans2010partial}, with alterations needed to accommodate the third-order derivative and the fractional term. We approximate the solution by
	\begin{equation}
	\begin{aligned}
	\psi^n(x,t) =\sum_{i=1}^n \xi^n_i(t)\phi_i(x),
	\end{aligned}
	\end{equation}
	where $\{\phi_i\}_{i=1}^\infty$ are the eigenfunctions of the Dirichlet-Laplacian operator:
	\[
	-\Delta \phi_i = \lambda_i  \phi_i \ \text{ in } \Omega, \qquad \phi_i=0 \ \text{ on } \partial \Om.
	\]
	Denote $V_n=\textup{span}\{\phi_1, \ldots, \phi_n\}$. The semi-discrete problem is given by
	\begin{equation} \label{semi-discrete_III}
	\begin{aligned}
	\begin{multlined}[t] (\tau \ptttn, \phi_j)_{L^2} + \prodLtwo{(1+\sigma) \pttn}{\phi_j}\\ + \prodLtwo{c^2\nabla\pn +\tau c^2\nabla \ptn +\delta \Dtal \nabla \pn}{\nabla \phi_j} = \prodLtwo{f}{\phi_j}, \end{multlined}
	\end{aligned}
	\end{equation}
	for all $j=1, \ldots, n$, with approximate initial conditions $(\psi^n, \psi_t^n, \psi_{tt}^n)\vert_{t=0}=(\psi_{0}^n, \psi_{1}^n, \psi_{2}^n)$ chosen as $L^2$ projections of $(\psi_0, \psi_1, \psi_2)$ onto $V_n$.  In other words,	
	\begin{equation}
	\begin{aligned}
	\psi_{0}^n= \sum_{i=1}^n \xi^n_{0, i} \phi_i(x), \quad \psi_{1}^n= \sum_{i=1}^n \xi^n_{1, i} \phi_i(x), \quad \psi_{2}^n= \sum_{i=1}^n \xi^n_{2, i} \phi_i(x)
	\end{aligned}
	\end{equation}
	with 
	\begin{equation}
	\xi^n_{0, i} =(\psi_0, \phi_i)_{L^2}, \quad \xi^n_{1, i} =(\psi_1, \phi_i)_{L^2}, \quad \xi^n_{2, i} =(\psi_2, \phi_i)_{L^2}.
	\end{equation}
	Then we know that
	\begin{equation}
	\begin{aligned}
	\|\psi^n_0\|_{H^1} \leq&\, \|\psi_0\|_{H^1} \ \text{ and } \ &&\psi^n_0 \rightarrow \psi_0 \ \text{ in } \ H_0^1(\Omega),\\
	\|\psi^n_1\|_{H^1} \leq&\, \|\psi_1\|_{H^1} \ \text{ and } \ &&\psi^n_1 \rightarrow \psi_1 \ \text{ in } \ H_0^1(\Omega),\\
	\|\psi^n_2\|_{L^2} \leq&\, \|\psi_2\|_{L^2} \ \, \text{ and } \ &&\psi^n_2 \rightarrow \psi_2 \ \text{ in } \ L^2(\Omega);
	\end{aligned}
	\end{equation}
	cf.~\cite[\S 7, Lemma 7.5]{robinson2001infinite}. \\
	
	\noindent (I) \emph{Existence of an approximate solution.} We first show that for a given $n$, a unique approximate solution exists. With $\boldsymbol{\xi}=[\xi^n_1 \ \xi^n_2 \ \ldots \ \xi^n_n]^T$, the approximate problem can be rewritten in matrix form
	\begin{equation} \label{matrix_eq}
	\begin{aligned}
	\tau M\boldsymbol{\xi}_{ttt}+M_{\sigma}(t)\boldsymbol{\xi}_{tt}+c^2 K\boldsymbol{\xi}+\tau c^2 K \boldsymbol{\xi}_t+\delta K \Dtal \boldsymbol{\xi}= \boldsymbol{f}
	\end{aligned}
	\end{equation}
	with the entries of matrices $M=[M_{ij}]$, $M_\sigma(t)=[M_{\sigma, ij}(t)]$, $K=[K_{ij}]$, and the vector $\boldsymbol{f}(t)=[f_i(t)]$ given by
	\begin{equation} \label{matrices}
	\begin{aligned}
	&M_{ij}= (\phi_i, \phi_j)_{L^2}, \quad &&M_{\sigma, ij}(t)= ((1+\sigma(t))\phi_i, \phi_j)_{L^2}, \\
	&K_{ij}= (\nabla \phi_i, \nabla \phi_j)_{L^2}, \quad && f_i(t)=(f(t), \phi_i)_{L^2}.
	\end{aligned}
	\end{equation}
	We also introduce the vectors of coordinates of the approximate initial data in the basis: 
	\[
	\boldsymbol{\xi}_0=[\xi^n_{0,1} \ \xi^n_{0,2} \ \ldots \ \xi^n_{0,n}]^T, \quad \boldsymbol{\xi}_1=[\xi^n_{1,1} \ \xi^n_{1,2} \ \ldots \ \xi^n_{1,n}]^T, \quad \boldsymbol{\xi}_2=[\xi^n_{2,1} \ \xi^n_{2,2} \ \ldots \ \xi^n_{2,n}]^T. 
	\]
	Then by setting $\boldsymbol{\mu}=\boldsymbol{\xi}_{ttt}$, we have
	\begin{equation} \label{eq_xi}
	\boldsymbol{\xi}(t)=\boldsymbol{\xi}_0 +t \boldsymbol{\xi}_1+\frac{1}{2}t^2\boldsymbol{\xi}_2+\int_0^t \frac12(t-s)^2 \boldsymbol{\mu}(s)\ds,
	\end{equation}
	and we can restate the semi-discrete problem as a system of Volterra integral equations given by
	\begin{equation}
	\begin{aligned}
	\begin{multlined}[t]
	\tau M \boldsymbol{\mu}(t)+M_{\sigma}(t)\left(\int_0^t \boldsymbol{\mu}(s)\ds +\boldsymbol{\xi}_{2}\right)+c^2K \left(\boldsymbol{\xi}_0 +t \boldsymbol{\xi}_1+\frac{t^2}{2}\boldsymbol{\xi}_2+\int_0^t \frac12(t-s)^2 \boldsymbol{\mu}(s)\ds\right)\\+\tau c^2 K\left(\boldsymbol{\xi}_1 +t \boldsymbol{\xi}_2+\int_0^t (t-s) \boldsymbol{\mu}(s)\ds\right)+ \frac{\delta}{\Gamma(\alpha)}K\int_0^t(t-s)^{\alpha-1}\left(\int_0^s\boldsymbol{\mu}(r)\, \textup{d}r+\boldsymbol{\xi}_2\right)\textup{d}s\\ =\boldsymbol{f}(t)\hphantom{fill} \end{multlined}
	\end{aligned}
	\end{equation}
	for $t \in (0,T)$. By using Dirichlet's formula 
	\begin{align} \label{exchange_integrals}
	{\frac{\delta}{\Gamma(\alpha)}}\int_0^t(t-s)^{\alpha-1}\left(\int_0^s\boldsymbol{\mu}(r)\,\textup{d}r\right) \ds = {\frac{\delta}{\Gamma(\alpha)}}\int_0^t \left(\int_r^t (t-s)^{\alpha-1}\,\textup{d}s\right) \boldsymbol{\mu}(r) \textup{d}r,
	\end{align}
	we arrive at an equivalent reformulation
	\begin{equation} \label{eq_mu}
	\begin{aligned}
	\boldsymbol{\mu}(t) =\tilde{\boldsymbol{f}}(t)+\int_0^t K_{\alpha}(t,s)\boldsymbol{\mu}(s)\ds,
	\end{aligned}
	\end{equation}
	where the first term on the right is defined as
	\begin{equation}
	\begin{aligned}
	\tilde{\boldsymbol{f}}(t)=&\,\begin{multlined}[t]-\frac{1}{\tau}M^{-1} \left\{M_{\sigma}(t)\boldsymbol{\xi}_2+c^2K (\boldsymbol{\xi}_0+t \boldsymbol{\xi}_1+\frac12 t^2 \boldsymbol{\xi}_2)+\tau c^2 K(\boldsymbol{\xi}_1+t\boldsymbol{\xi}_2)-\boldsymbol{f}(t) \right.\\ \left.
	+\frac{\delta}{\Gamma(\alpha)}K\int_0^t (t-s)^{\alpha-1} \boldsymbol{\xi}_2\ds\right\} \end{multlined}
	\end{aligned}
	\end{equation}
	and the kernel function is given by
	\begin{equation}
	\begin{aligned}  
	K_{\alpha}(t,s)=&\,\begin{multlined}[t] -\frac{1}{\tau}M^{-1}\left(M_{\sigma}(t)+\frac12 c^2K(t-s)^2+\tau c^2K (t-s)\right) \\
	-\frac{1}{\tau} \frac{\delta}{\Gamma(\alpha+1)}M^{-1} K (t-s)^\alpha.
	\end{multlined}
	\end{aligned}
	\end{equation}
{To arrive at the kernel expression, we have employed
	\[
	\frac{\delta}{\Gamma(\alpha)}K\int_s^t (t-r)^{\alpha-1}\, \textup{d}r=\frac{\delta}{\alpha \Gamma(\alpha)}K(t-s)^\alpha.
	\] }
\noindent Due to the $L^\infty(0,T)$ regularity of the kernel $K_\alpha$ on $D=\{(t,s): \ 0 \leq s \leq t \leq T\}$ and the fact that function $\tilde{\boldsymbol{f}}$ belongs to $L^2(0,T)$, vector equation \eqref{eq_mu} has a unique solution $\boldsymbol{\mu} \in L^2(0,T)$. This claim directly follows by considering (systems of) integral equations in $L^2(0, T)$ instead of $C[0,T]$ in~\cite[Theorem 2.1.7]{brunner2004collocation}; see also \cite[Theorem 4.2, p. 24 in \S 9]{GLS90}. From \eqref{eq_xi}, taking into account initial data, a unique $\boldsymbol{\xi} \in H^3(0,T)$ and, in turn, $\pn \in H^3(0,T; V_n)$ exists. \\

	\noindent (II) \emph{A priori energy analysis.} We next focus on the derivation of the energy estimate, which goes through by testing the semi-discrete problem by $\psi^n_{tt}(t) \in V_n$. More precisely, we test \eqref{semi-discrete_III} with $\xi^n_{tt}(t)$ and sum over $j=1, \ldots, n$. After integrating over $(0,t)$, we at first obtain the identity
	\begin{equation} \label{energy_id}
	\begin{aligned}
	&\begin{multlined}[t]\frac12 \tau\nLtwo{\pttn(t)}^2 \Big \vert_0^t
	+ \frac12\tau c^2\nLtwo{\nabla \ptn(s)}^2 \Big \vert_0^t +\delta  \int_0^t \prodLtwo{\Dt^{2-\alpha}\nabla \pn}{\nabla \pttn} \ds
	\end{multlined}\\
	=&\begin{multlined}[t] - \int_0^t(1+\sigma)\| \pttn\|^2_{L^2}\ds+\int_0^t \prodLtwo{f}{\pttn}\ds - c^2 \prodLtwo{\nabla \psi^n}{\nabla \ptn} \Big \vert_0^t 
	+  c^2 \int_0^t \nLtwo{\nabla \ptn}^2 \ds. \end{multlined}
	\end{aligned}
	\end{equation}
	By Young's $\varepsilon$-inequality, we have
	\begin{equation}
	\begin{aligned}
	c^2\int_0^t \prodLtwo{\nabla \pn}{\nabla \ptn} \ds \leq \frac{1}{4\varepsilon}c^2(C(T)\|\nabla \ptn\|_{L^2_t(L^2)}+\nLtwo{\nabla \psi^n_0})^2 +\varepsilon c^2 \nLtwo{\nabla \ptn(t)}^2.
	\end{aligned}
	\end{equation}
	For $0<\varepsilon < \tau/2$, employing estimate \eqref{fractional_est_w} and Gronwall's inequality leads to   
	\begin{equation} \label{energy_est1_fJMGT_W_III_discrete}
	\begin{aligned}
	&\begin{multlined}[t] \nLtwo{\pttn(t)}^2  
	+ \nLtwo{\nabla \ptn(t)}^2+ \cos(\pi \alpha/2)  \| \nabla \pttn \|_{H^{-\alpha/2}(0,t; L^2(\Omega))}^2
	\end{multlined}\\
	\lesssim&\,\begin{multlined}[t] \nLtwoLtwo{f}^2+\nLtwo{\nabla \psi_0}^2+\nLtwo{\nabla \psi_1}^2+\nLtwo{\psi_2}^2. \end{multlined}
	\end{aligned}
	\end{equation}
	Note that we can also use the estimate
	\[
	\nLtwo{\pn(t)} \leq \CPF\nLtwo{\nabla \pn(t)} \lesssim  T \sup_{t \in (0,T)}\nLtwo{\nabla \ptn(t)}+{\nLtwo{\nabla \pn_0}}.
	\]	
	Additionally, standard arguments (cf.~\cite[\S 7]{evans2010partial}) lead to the bound
	\begin{equation}
	\begin{aligned}
	\int_0^t\|\psi_{ttt}^n\|^2_{H^{-1}}\ds \lesssim&\, \begin{multlined}[t] \int_0^t (1+\overline{\sigma})\nLtwo{\pttn}^2\ds+
	\int_0^t \nLtwo{\nabla \pn}^2\ds\\+\int_0^t \nLtwo{\nabla \ptn}^2\ds+\int_0^t \nLtwo{\Dtal\nabla \pn}^2\ds+\int_0^t \nLtwo{f}^2\ds.
	\end{multlined}
	\end{aligned}
	\end{equation}	
	We can further estimate the fractional term on the right as follows:
	\begin{equation} \label{est_Dtal}
	\begin{aligned}
	\|\Dtal\nabla \pn\|^2_{L^2(0,t; L^2)} = \|\textup{I}^{\alpha} \nabla \pttn\|^2_{L^2(0,t; L^2)} \lesssim&\, \|\nabla \pttn\|^2_{X_{-\alpha}(0,t; L^2)} \\
	\lesssim&\,\|\nabla \pttn\|^2_{H^{-\alpha/2}(0,t; L^2)},
	\end{aligned}
	\end{equation}
	where the first inequality follows by~\cite[Theorem 1]{gorenflo1999operator}. {Note that $\{X_{\beta}\}_{\beta \in \R}$ represents a scale of Hilbert spaces of functions $(0,t) \mapsto L^2(\Omega)$; cf.~\cite[\S 5]{baumeister1987stable} and~\cite[Lemma 8]{gorenflo1999operator}.} The second inequality follows by the fact that \[X_{\alpha}(0,t; L^2(\Om))\subseteq X_{\alpha/2}(0,t; L^2(\Om))= H_0^{\alpha/2}(0,t; L^2(\Om))\] for $\alpha<1$ and therefore $\alpha/2<\frac12$, together with duality. Note also that, on account of estimate \eqref{eqn:Alikhanov_1}, we know that
	\begin{equation}\label{Alikhanov_Galerkin}
	\begin{aligned}
	(\nabla \Doal \pn_{t}, \nabla \pttn)_{L^2}=&\,	(\nabla \Doal \pn_{t}, \Dt^\alpha\nabla \Dt^{1-\alpha} \ptn)_{L^2}\\\geq&\,  \frac12 \Dt^\alpha\|\nabla \Doal \ptn\|^2_{L^2}\\
	=&\, \frac12\ddt \, \textup{I}^{1-\alpha} \|\nabla \Doal \ptn\|^2_{L^2}.
	\end{aligned}
	\end{equation}
	Employing this estimate instead of \eqref{fractional_est_w} in the above derivation yields a uniform bound on $\displaystyle \sup_{t \in (0,T)} \textup{I}^{1-\alpha} \|\nabla \Doal \ptn\|^2_{L^2}$, which will be needed in the proof of uniqueness.  \\
	
	\noindent (III) \emph{Passing to the limit.} Thanks to the uniform bounds and \eqref{est_Dtal}, we have weak convergence of a subsequence, which we do not relabel, in the following sense:
	\begin{equation} \label{weak_limits1}
	\begin{alignedat}{4} 
	\ptttn  &\relbar\joinrel\rightharpoonup \pttt &&\ \text{ weakly}  &&\text{ in } &&L^2(0,T; H^{-1}(\Omega)),  \\
	\pttn  &\relbar\joinrel\rightharpoonup \ptt &&\ \text{ weakly-$\star$}  &&\text{ in } &&L^\infty(0,T; L^2(\Omega)),  \\
	\ptn &\relbar\joinrel\rightharpoonup \pt &&\ \text{ weakly-$\star$} &&\text{ in } &&L^\infty(0,T; H_0^1(\Omega)), \\
	\pn &\relbar\joinrel\rightharpoonup \psi &&\ \text{ weakly-$\star$} &&\text{ in } &&L^\infty(0,T; H_0^1(\Omega)).
	\end{alignedat} 
	\end{equation}
	Furthermore, we know that
		\begin{equation} \label{weak_limits2}
	\begin{alignedat}{4} 
	\nabla \pttn &\relbar\joinrel\rightharpoonup \nabla \ptt &&\ \ \text{  weakly} &&\text{ in } &&H^{-\alpha/2}(0,T; L^2(\Omega)), \\
	\Dtal\nabla \pn &\relbar\joinrel\rightharpoonup \Dtal\nabla \psi &&\ \ \text{  weakly} &&\text{ in } &&L^2(0,T; L^2(\Omega)), \\
	\textup{I}^{1-\alpha} \|\nabla \Doal \ptn\|^2_{L^2} &\relbar\joinrel\rightharpoonup \textup{I}^{1-\alpha} \|\nabla \Doal \pt\|^2_{L^2} &&\ \ \text{  weakly-$\star$} &&\text{ in } &&L^\infty(0,T).
	\end{alignedat} 
	\end{equation}
	We can thus pass to the weak limit in the usual way to conclude that $\psi$ solves \eqref{fMGT1_sigma}. Further, weak/weak-$\star$ lower semi-continuity of norms implies
	\[	
	\begin{aligned}
	\| \nabla \ptt \|_{H^{-\alpha/2}(0,t; L^2(\Omega))}^2\leq&\,  \liminf_{n \rightarrow \infty} \| \nabla \pttn \|_{H^{-\alpha/2}(0,t; L^2(\Omega))}^2, \\
	 \sup_{t \in (0,T)}  \textup{I}^{1-\alpha} \|\nabla \Doal \pt\|^2_{L^2} \leq&\, \liminf_{n \rightarrow \infty} \sup_{t \in (0,T)} \textup{I}^{1-\alpha} \|\nabla \Doal \ptn\|^2_{L^2}.
	\end{aligned} 
	\]
	and thus by passing to the limit in the energy estimate for $\pn$, we conclude that $\psi$ satisfies \eqref{energy_est1_fJMGT_W_III}. \\

	\noindent (IV) \emph{Attainment of the initial conditions.} We next show that $\psi$ attains its initial conditions. By \eqref{weak_limits1} and~\cite[Lemma 3.1.7]{zheng2004nonlinear}, we know that
	\[
	\pn (0) \relbar\joinrel\rightharpoonup \psi(0) \text{ weakly}  \text{ in } H_0^1(\Omega),  
	\]
	and since $\psi^n(0) \rightarrow \psi_0$ in $H_0^1(\Omega)$, we have $\psi(0)=\psi_0$. Further,
	\[
	\ptn (0) \relbar\joinrel\rightharpoonup \psi_t(0) \text{ weakly}  \text{ in } L^2(\Omega),  
	\]
	and thus $\psi_t(0)=\psi_1$ as an equality in $L^2(\Omega)$. Similarly, $\psi_{tt}(0)=\psi_2$ as an equality in $H^{-1}(\Omega)$; that is, 
	\[\langle \psi_{tt}(0  ), v \rangle_{H^{-1}, H^1} = \langle \psi_{2}, v \rangle_{H^{-1}, H^1} , \quad \forall v \in H_0^1(\Omega).\]

		\noindent (V) \emph{Uniqueness.} To prove uniqueness, we should show that the only solution of 
	\begin{equation} \label{homogeneous_eq}
	\begin{aligned}
	\tau \psi_{ttt}+(1+\sigma)\ptt-\tau c^2\Delta \psi_t-c^2 \Delta \psi-\delta \Doal \D \psi_t=0
	\end{aligned}
	\end{equation}
	with $\psi_0=\psi_1=\psi_2=0$ in $\spacelowIII$ is $\psi=0$. The issue, however, is that at this point we are not allowed to directly test \eqref{fMGT1_sigma} with $\ptt$ due to its low regularity. Instead, in the spirit of~\cite[\S 2.4]{temam2012infinite}, we will prove that such $\psi$ satisfies
	\begin{equation} \label{unique_III_low}
	\begin{aligned}
	&(\tau \psi_{ttt}-\tau c^2\Delta \psi_t-c^2 \Delta \psi-\delta \Doal \D \psi_t, \psi_{tt})_{L^2} \\
\gtrsim &\,\begin{multlined}[t] \ddt \left\{\frac12\tau \nLtwo{\ptt}^2+\frac12\tau c^2\nLtwo{\nabla \pt}^2  +c^2(\nabla \psi, \nabla \pt)_{L^2}\right. \\ \left.\hspace*{3cm}+\frac{\delta}{2}\textup{I}^{1-\alpha} \|\nabla \Doal \psi_{t}\|^2_{L^2}\right\} - c^2\nLtwo{\nabla \pt}^2.\end{multlined}
	\end{aligned}
	\end{equation}
	This estimate combined with \eqref{homogeneous_eq} implies that
	\begin{equation} \label{unique_III_low_b}
	\begin{aligned}
	\begin{multlined}[t] \ddt \left\{\frac12\tau \nLtwo{\ptt}^2+\frac12\tau c^2\nLtwo{\nabla \pt}^2+c^2(\nabla \psi, \nabla \pt)_{L^2} \right. \\ \left.\hspace*{3cm}+\frac{\delta}{2}\textup{I}^{1-\alpha} \|\nabla \Doal \psi_{t}\|^2_{L^2}\right\} - c^2\nLtwo{\nabla \pt}^2
	\lesssim -((1+\sigma \psi_{tt}, \psi_{tt})_{L^2} , \end{multlined}
	\end{aligned}
	\end{equation}
	after which we can proceed as in the previous energy analysis to arrive at $\psi=0$. We note that if $\psi \in \spacelowIII$ solves \eqref{homogeneous_eq}, then  
	\begin{equation} \label{reg1}
	\begin{aligned}
	\psi_t \in L^2(H_0^1(\Om)), \quad \ptt \in L^2(0,T; L^2(\Om)) \cap H^{-\alpha/2}(0,T; H_0^1(\Omega)).
	\end{aligned}
	\end{equation}
	Furthermore, a bootstrap argument yields
	\begin{equation} \label{reg2}
	\begin{aligned}
	\tau \psi_{ttt}-\tau c^2\Delta \psi_t-c^2 \Delta \psi-\delta \Doal \D \psi_t= - (1+\sigma)\ptt \in L^2(L^2).
	\end{aligned}
	\end{equation}
	We next construct a regularization of $\psi$ which satisfies \eqref{unique_III_low}, following~\cite[\S 2.4, Lemma 4.1]{temam2012infinite}. Let $\tilde{\psi}: \R \mapsto H_0^1(\Omega)$ be defined by
	\begin{equation}
	\begin{aligned}
	\tilde{\psi}= \begin{cases}
	\theta \psi, \ &\text{ on } (0,T), \\
	0, \ &\text{ on } \R \setminus [0,T],
	\end{cases}
	\end{aligned}
	\end{equation}
	where $\theta: \R \mapsto [0,1]$ is a $C^\infty$ truncation function, equal to $0$ on $\R \setminus [0,T]$ and to $1$ on some sub-interval of $(0,T)$. Then
	\begin{equation} \label{reg3}
	\begin{aligned}
	&\tilde{\psi}_t \in L^2(0, \infty; H_0^1(\Om)), \quad \tilde{\psi}_{tt} \in L^2(0, \infty; L^2(\Om)) \cap H^{-\alpha/2}(0,\infty; H_0^1(\Omega)),\\
	&\, \tau \tilde{\psi}_{ttt}-\tau c^2\Delta \tilde{\psi}_t-c^2 \Delta \tilde{\psi}-\delta \Doal \D \tilde{\psi}_t \in L^2(0, \infty; L^2(\Omega)).
	\end{aligned}
	\end{equation}
	We regularize $\tilde{\psi}$ by $\tilde{\psi}_\varepsilon= \varrho_\varepsilon * \tilde{\psi}$, with $\varrho_\varepsilon$ being a $C^\infty$ mollifier. Then $\tilde{\psi}_\varepsilon: \R \mapsto H_0^1(\Omega)$ is a $C^\infty$ function, which satisfies
	\begin{equation} 
	\begin{aligned}
	&(\tau \tilde{\psi}_{\varepsilon, ttt}-\tau c^2 \Delta \tilde{\psi}_{\varepsilon, t}-c^2 \Delta \tilde{\psi}_\varepsilon- \delta \Doal \Delta \tilde{\psi}_{\varepsilon, t}, \tilde{\psi}_{\, \varepsilon, tt})_{L^2} \\
	=&\,  \begin{multlined}[t]\ddt \left\{\frac12\tau \nLtwo{\tilde{\psi}_{\varepsilon, tt}}^2+\frac12\tau c^2\nLtwo{\nabla \tilde{\psi}_{\varepsilon, t}}^2+c^2(\nabla \tilde{\psi}_\varepsilon, \nabla \tilde{\psi}_{\varepsilon, t})_{L^2} \right\} \\- c^2\nLtwo{\nabla \tilde{\psi}_{\varepsilon,t}}^2+\delta (\nabla \Doal \tilde{\psi}_{\varepsilon, t}, \nabla \tilde{\psi}_{\varepsilon,tt})_{L^2}. \end{multlined}
	\end{aligned}
	\end{equation}
	Similarly to \eqref{Alikhanov_Galerkin}, we have
	\[
	\begin{aligned}
 (\nabla \Doal \tilde{\psi}_{\varepsilon, t}, \nabla \tilde{\psi}_{\varepsilon,tt})_{L^2}=\,(\nabla \Doal \tilde{\psi}_{\varepsilon, t}, \Dt^\alpha\nabla \Dt^{1-\alpha} \tilde{\psi}_{\varepsilon, t})_{L^2}\geq \frac12\ddt \textup{I}^{1-\alpha} \|\nabla \Doal \tilde{\psi}_{\varepsilon, t}\|^2_{L^2}.
 \end{aligned}
	\]
	 Therefore,
		\begin{equation} \label{ineq_tilde_eps}
	\begin{aligned}
	&(\tau \tilde{\psi}_{\varepsilon, ttt}-\tau c^2 \Delta \tilde{\psi}_{\varepsilon, t}-c^2 \Delta \tilde{\psi}_\varepsilon- \delta \Doal \Delta \tilde{\psi}_{\varepsilon, t}, \tilde{\psi}_{\, \varepsilon, tt})_{L^2} \\
	\gtrsim&\, \begin{multlined}[t] \ddt \left\{\frac12\tau \nLtwo{\tilde{\psi}_{\varepsilon, tt}}^2+\frac12\tau c^2\nLtwo{\nabla \tilde{\psi}_{\varepsilon, t}}^2 +c^2(\nabla \tilde{\psi}_\varepsilon, \nabla \tilde{\psi}_{\varepsilon, t})_{L^2}\right.\\ \left.+\frac{\delta}{2}\textup{I}^{1-\alpha} \|\nabla \Doal \tilde{\psi}_{\varepsilon, t}\|^2_{L^2}\right\} - c^2\nLtwo{\nabla \tilde{\psi}_{\varepsilon,t}}^2.\end{multlined}
	\end{aligned}
	\end{equation}
	Thanks to \eqref{reg3}, we know that
	\[
	\begin{aligned}
	&\lim_{\varepsilon \rightarrow 0}\ (\tau \tilde{\psi}_{\varepsilon, ttt}-\tau c^2 \Delta \tilde{\psi}_{\varepsilon, t}-c^2 \Delta \tilde{\psi}_\varepsilon- \delta \Doal \Delta \tilde{\psi}_{\varepsilon, t}, \tilde{\psi}_{\varepsilon, tt})_{L^2} \\
	=&\,(\tau \tilde{\psi}_{ttt}-\tau c^2 \Delta \tilde{\psi}_{t}-c^2 \Delta \tilde{\psi}- \delta \Dal \Delta \tilde{\psi}_{t}, \tilde{\psi}_{tt})_{L^2}. 
	\end{aligned}
	\]
	We can thus pass to the limit $\varepsilon \rightarrow 0$ in \eqref{ineq_tilde_eps} to arrive at 
	\begin{equation}
	\begin{aligned}
	&(\tau \tilde{\psi}_{ttt}-\tau c^2 \Delta \tilde{\psi}_t-c^2 \Delta \tilde{\psi}- \delta \Doal \Delta \tilde{\psi}_t, \tilde{\psi}_{tt})_{L^2} \\
\gtrsim&\, \begin{multlined}[t]\ddt \left\{\frac12\tau \nLtwo{\tilde{\psi}_{tt}}^2+\frac12\tau c^2\nLtwo{\nabla \tilde{\psi}_{t}}^2 +c^2(\nabla \tilde{\psi}, \nabla \tilde{\psi}_t)_{L^2}+\frac{\delta}{2}\textup{I}^{1-\alpha} \|\nabla \Doal \tilde{\psi}_{t}\|^2_{L^2}\right\}\\ - c^2\nLtwo{\nabla \tilde{\psi}_{t}}^2.\end{multlined}
	\end{aligned}
	\end{equation}
	By restriction to $(0,T)$, the same holds for $\theta \psi$, from which the claim follows.
\end{proof}

To formulate the second well-posedness result, we introduce a higher-order solution space:
\begin{equation}
\begin{aligned}
\spacehighIII=   W^{1, \infty}(0,T; \Honetwo) \cap W^{2, \infty}(0,T; H_0^1(\Om))\cap H^3(0,T; L^2(\Om)). 
\end{aligned}
\end{equation}
We denote by $\|\cdot\|_{\spacehighIII}$ the corresponding norm on this space. Under stronger regularity assumptions on the data and the coefficient $\sigma$, the fMGT III equation \eqref{fMGT_III_sigma} has a unique solution in this space.
\begin{proposition}[Higher regularity for the fMGT III equation]  \label{Prop:fMGT_III_higher} Let $\alpha \in (0,1]$ and $\sigma \in L^\infty(0,T; L^\infty(\Omega)\cap W^{1, 4}(\Omega))$. Given $f \in L^2(0,T; H_0^1(\Om))$ and \[(\psi_0, \psi_1, \psi_2) \in\Honetwo \times \Honetwo \times H_0^1(\Om),\] there exists a unique solution $\psi \in \spacehighIII$, which solves \eqref{fMGT1_sigma} in the $L^2(0,T; L^2(\Omega))$ sense, and satisfies
	\begin{equation} \label{energy_est2_fMGT_W_III}
	\begin{aligned}
	\|\psi\|^2_{\spacehighIII}
	\lesssim\,\begin{multlined}[t] \nLtwoLtwo{\nabla f}^2+\nLtwo{\D \psi_0}^2+\nLtwo{\D \psi_1}^2+\nLtwo{\nabla \psi_2}^2. \end{multlined}
	\end{aligned}
	\end{equation}
\end{proposition}
\begin{proof}
	The statement when $\alpha=1$ follows analogously to~\cite[Theorem 3.1]{KaltenbacherNikolic}. The proof in the case $\alpha \in (0,1)$ can again be conducted by employing a Galerkin analysis in space. We only outline the derivation of the energy estimate, which follows by testing the semi-discrete problem by $-\D \psi^n_{tt}$. We omit the superscript $n$ in the notation below. After integrating over $(0,t)$, we first obtain the identity
	\begin{equation} \label{energy_id}
	\begin{aligned}
	&\begin{multlined}[t]\frac12 \tau\nLtwo{\nabla \ptt(t)}^2 \Big \vert_0^t
	+ \frac12\tau c^2\nLtwo{\D \psi_{t}(t)}^2 \Big \vert_0^t+\delta  \int_0^t \prodLtwo{\Dt^{2-\alpha}\Delta \psi}{\Delta \psi_{tt}} \ds
	\end{multlined}\\
	=&\,\begin{multlined}[t] \int_0^t \prodLtwo{\nabla f}{\nabla \psi_{tt}}\ds- \int_0^t \prodLtwo{(1+\sigma) \nabla \psi_{tt}}{\nabla \psi_{tt}} \ds-\int_0^t (\psi_{tt}\nabla \sigma, \nabla \psi_{tt})\ds\\ - c^2 \prodLtwo{\D \psi}{\D \psi_t} \Big \vert_0^t 
	+  c^2 \int_0^t \nLtwo{\D \psi_t}^2 \ds. \end{multlined}
	\end{aligned}
	\end{equation}
	We can rely on the following estimate:
	\begin{equation}
	\begin{aligned}
	&\int_0^t \prodLtwo{\nabla f}{\nabla \psi_{tt}}\ds-\int_0^t (\psi_{tt}\nabla \sigma, \nabla \psi_{tt})\ds - c^2 \prodLtwo{\D \psi}{\D \psi_t} \Big \vert_0^t \\
	\leq&\, \begin{multlined}[t]  \nLtwoLtwo{
		\nabla
		f}\|\nabla \psi_{tt}\|_{L^2_t(L^2)}+\|\psi_{tt}\|_{L^2_t(L^4)}\|\nabla \sigma\|_{L^\infty(L^4)}\|\nabla \psi_{tt}\|_{L_t^2(L^2)}\\
	+ c^2 \|\D \psi(t)\|_{L^2}\|\D \psi_t(t)\|_{L^2}+ c^2 \|\D \psi_0\|_{L^2}\|\D \psi_1\|_{L^2}; \end{multlined}
	\end{aligned}
	\end{equation}
	see also~\cite[Theorem 3.1]{kaltenbacher2021inviscid}. We further note that
	\begin{equation}
	\begin{aligned}
	\|\psi_{tt}\|_{L^2_t(L^4)}\|\nabla \sigma\|_{L^\infty(L^4)}\|\nabla \psi_{tt}\|_{L_t^2(L^2)} \leq \CHone \|\nabla \sigma\|_{L^\infty(L^4)}\|\nabla \psi_{tt}\|_{L_t^2(L^2)}^2
	\end{aligned}
	\end{equation}
	and that
	\begin{equation}
	\begin{aligned}
	\|\D \psi(t)\|_{L^2}\|\D \psi_t(t)\|_{L^2} \leq \frac{1}{\epsilon}(\sqrt{T}\|\D \psi_t\|_{L^2_t(L^2)}+\|\D \psi_0\|_{L^2})^2+\epsilon \|\D \psi_t(t)\|_{L^2}^2.
	\end{aligned}
	\end{equation}
	For fixed, small enough $\epsilon>0$, an application of Gronwall's inequality thus yields \eqref{energy_est2_fJMGT_W_III}, at first in a discrete setting. Additionally, we obtain 
	\begin{equation}
	\begin{aligned}
	\nLtwotLtwo{\psi_{ttt}}^2 \lesssim& \, \begin{multlined}[t] \nLtwotLtwo{\psi_{tt}}^2 +\nLtwotLtwo{\Delta\psi}^2+ \nLtwotLtwo{\Delta \psi_t}^2\\+ \nLtwotLtwo{\Dtal \Delta \psi}^2 +\nLtwoLtwo{f}^2,
	\end{multlined}
	\end{aligned}
	\end{equation}
	where, similarly to \eqref{est_Dtal}, we can further estimate the fractional term as follows: 
	\begin{equation}
	\nLtwotLtwo{\Dtal \Delta \psi} \lesssim \|\Delta \psi_{tt}\|_{H_t^{-\alpha/2}(L^2)}.
	\end{equation}
	The rest of the arguments follow analogously to the proof of Proposition~\ref{Prop:fMGT_III_lower}. We point out that in this higher-order setting, we are allowed to test the homogeneous problem ($f=\psi_0=\psi_1=\psi_2=0$) directly with $\ptt$ to prove uniqueness. \\
\indent	Note that for $\psi \in \spacehighIII$, thanks to the embedding $W^{1, \infty}(0,T; \Honetwo) \hookrightarrow C([0,T]; \Honetwo)$, we know that $\psi \in C([0,T]; \Honetwo)$. Likewise, we have \[\psi_t \in L^\infty(0,T; \Honetwo) \cap C([0,T]; H^1_0(\Omega)).\] According to~\cite[\S2, Lemma 3.3]{temam2012infinite}, this implies that $ \psi_t$ is weakly continuous from $[0,T]$ into $\Honetwo$. Similarly, we can prove that $\psi_{tt} \in C_{w}([0,T]; H_0^1(\Om))$. 
\end{proof}
\indent We are now ready to prove a well-posedness result for the nonlinear fJMGT--W III equation.
\begin{theorem}[Local well-posedness of the fJMGT--W III equation] \label{Thm:fJMGT_W_III} Let $\alpha \in (0,1]$, $\tilde{T}>0$, and $\varrho>0$. Further, assume that $f \in L^2(0,\tilde{T}; H_0^1(\Om))$ and that
	\begin{equation}
	\|f\|^2_{L^2(H^1)}+\|\psi_0\|_{H^2}^2+\|\psi_1\|_{H^2}^2+\|\psi_2\|_{H^1}^2 \leq \varrho^2. 
	\end{equation}
	Then there exists $T=T(\varrho) \leq \tilde{T}$, such that the initial boundary-value problem
	\begin{equation}\label{ibvp_fJMGT_W_III}
	\left \{
	\begin{aligned}
	\tau \psi_{ttt} + &(1+2k \psi_t) \psi_{tt} - c^2\Delta\psi -\tau c^2\Delta \psi_t - \delta \Doal \Delta \psi_t =f\hspace*{-2mm}&&\text{in }\Omega\times(0,T), \\[1mm]
	&\psi=\,0&&\text{on }\partial\Omega\times(0,T),\\[1mm]
	&(\psi, \psi_t, \psi_{tt})=\,(\psi_0, \psi_1, \psi_2)&&\mbox{in }\Omega\times \{0\},
	\end{aligned} \right.
	\end{equation}
	has a unique solution $\psi \in X^{\textup{high}}_{\textup{fMGT III}}$, which satisfies	
	\begin{equation} \label{energy_est2_fJMGT_W_III}
	\begin{aligned}
	\begin{multlined}[t] \|\psi\|^2_{\spacehighIII}
	\end{multlined}
	\lesssim\,\begin{multlined}[t] 	\|f\|^2_{L^2(H^1)}+\|\psi_0\|_{H^2}^2+\|\psi_1\|_{H^2}^2+\|\psi_2\|_{H^1}^2. \end{multlined}
	\end{aligned}
	\end{equation}
\end{theorem}
\begin{proof}
	The proof follows by applying the Banach Fixed-point theorem to the mapping $\mathcal{T}: w \mapsto \psi$, where $\psi$ solves the linearized equation \eqref{fMGT1_sigma} with $\sigma= 2k w_t$ and
	\begin{equation} \label{defBR}
	w \in B_R:=\{w \in \spacehighIII\, : \|w\|_{ X^{\textup{high}}_{\textup{fMGT III}}}\leq R, \ w(0)=\psi_0, \, w_t(0)=\psi_1, \, w_{tt}(0)=\psi_2  \},
	\end{equation}
	with $R>0$ specified below. Note that
	\[
	\begin{aligned}
	\nLinfLinf{\sigma}+\|\sigma\|_{L^\infty(W^{1,4})} \leq&\, 2\CHtwo |k| \|w_t\|_{L^\infty(H^2)}+2C_{H^1, L^4} |k|\|w_t\|_{L^\infty(H^2)}\\ \lesssim&\, R.
	\end{aligned}
	\]
	Thus by employing estimate \eqref{energy_est2_fMGT_W_III}, where the hidden constant has the form $C_1 \exp(C_2 (R+1)\tilde{T})$, it immediately follows that $\mathcal{T}$ is a well-defined self-mapping on $B_R^{\textup{W}}$, provided $R>0$ is chosen so that 
	\[
	\sqrt{C_1 \exp(C_2 (R+1)\tilde{T})} \, \varrho \leq R. 
	\]
	\indent Next, we prove that $\mathcal{T}$ is strictly contractive. {Note that we will prove contractivity with respect to the weaker norm $\|\cdot\|_{\spacelowIII}$; recall the definition of the space $X^{\textup{low}}_{\textup{fMGT III}}$ in \eqref{def_X_fMGTIII} for $\alpha \in (0,1)$ and \eqref{def_X_fMGTIII_alphaone} for $\alpha=1$}.\\
	\indent We take any $w^{(1)}$ and $w^{(2)}$ in $B_R^W$ and set  $\psi^{(1)}=\mathcal{T} w^{(1)}$ and $\psi^{(2)}=\mathcal{T} w^{(2)} $. We also introduce the short-hand notation for the differences \[\overline{\psi}=\psi^{(1)} -\psi^{(2)}, \qquad \overline{w}= w^{(1)} -w^{(2)}.\] Then we know that $\opsi$ solves the linear equation
	\begin{equation} \label{West_contract_eq}
	\tau\opsi_{ttt}+(1+2k w_t^{(1)})\opsi_{tt}-c^2 \Delta \opsi-\tau c^2 \Delta \opsi_t-\delta \Doal \Delta \opsi_t +2k \overline{w}_t\psi^{(2)} _{tt}=0
	\end{equation}
	and has zero initial conditions. Employing the lower-order estimate \eqref{energy_est1_fJMGT_W_III} with $\sigma=2k w_t^{(1)}$ and $f= -2k \overline{w}_t\psi^{(2)} _{tt}$ yields the bound
	\begin{align}
	\|\opsi\|_{\spacelow}\leq&\, \sqrt{C_1\exp(C_2(R+1)\tilde{T})} \nLtwoLtwo{f}\\
	\leq&\, 
	\begin{multlined}[t]2\sqrt{C_1\exp(C_2(R+1)\tilde{T})}|k| \nLinfLfour{\psi_{tt}^{(2)}}\sqrt{\tilde{T}}\nLinfLfour{\overline{w}_t} \end{multlined}\\
	\leq&\, \theta \|\overline{w}\|_{\spacelow}. 
	\end{align}
	Thus we can guarantee that $\theta \in (0,1)$ and obtain strict contractivity of $\mathcal{T}$ by decreasing $\tilde{T}$. \\
	\indent We note that the space $B_R$ with the metric induced by the norm $\|\cdot\|_{\spacelow}$ is a closed subset of a complete normed space; cf.~\cite[Theorem~4.1]{kaltenbacher2021inviscid}. Existence of a unique solution in $B_R$ then follows by Banach's Fixed-point theorem. 
\end{proof}

\subsection{Limiting behavior of the fJMGT--W III equation} \label{Sec:Limit}
We next discuss the limit with respect to the order of differentiation. Given $\alpha \in (0,1)$, under the assumptions of Theorem~\ref{Thm:fJMGT_W_III}, let $\psi^\alpha$ be the solution of the fJMGT--W III equation:
{\[
	\tau \psi^\alpha_{ttt}+(1+2k\psi^\alpha_t)\psi^\alpha_{tt}-c^2 \Delta \psi^\alpha -\tau c^2\Delta \psi^\alpha_{t}- \delta \Dtal\D\psi^\alpha=f.
	\]}
 Let $\psi$ solve the corresponding JMGT--Westervelt equation obtained by setting $\alpha=1$ above. Then the difference $\overline{\psi}=\psi^{\alpha}-\psi$ solves 
\begin{equation} \label{fJMGT_W_III_diff}
\begin{aligned}
&\tau \opsi_{ttt} + (1+2k\psi_t^{\alpha})\opsi_{tt} - c^2\Delta\opsi -\tau c^2\Delta \opsi_t - \delta \Dt^{1-\alpha} \Delta \opsi_t+2k\opsi_t\psi_{tt} \\
=&\, \delta (\Dt^{1-\alpha}\Delta \psi_t-\Delta \psi_t).
\end{aligned}
\end{equation}
Similarly to the proof of Proposition~\ref{Prop:fMGT_III_lower}, testing with $\overline{\psi}_{tt}$ (which we are allowed to do in this higher-regularity setting) leads to 
\begin{equation} 
\begin{aligned}
\nLtwo{\opsi_{tt}(t)}^2  
+ \nLtwo{\nabla \opsi_{t}(t)}^2 \lesssim\,\nLtwoLtwo{\Dt^{1-\alpha}\nabla \psi_t-\nabla \psi_t}^2.
\end{aligned}
\end{equation}
By recalling Lemma~\ref{Lemma:Limit}, we find that if $\pt(0)=0$, then
\begin{equation}
\lim_{\alpha \rightarrow1^-} \nLtwoLtwo{\Dt^{1-\alpha}\nabla \psi_t-\nabla \psi_t}=0,
\end{equation}
and thus arrive at the following result. 
\begin{proposition} Let the assumptions of Theorem~\ref{Thm:fJMGT_W_III} hold with $\psi_1=0$. Let $\{\psi^\alpha\}_{\alpha \in (0,1)}$ be the family of solutions to the \textup{fJMGT--W III} equation and let $\psi$ solve the corresponding \textup{JMGT} equation with $\alpha=1$. Then $\psi^{\alpha}$ converges to $\psi$ in $W^{1, \infty}(0,T; H_0^1(\Omega)) \cap W^{2, \infty}(0,T; L^2(\Om))$ as $\alpha \rightarrow 1^-$.
\end{proposition}

%%%%%%%%%%%%%%%%%%%%%%%%%%%%%%%%%%%%%%%%%%%%%%%%%%%%%%%
\section{Analysis of the equations with the fractional leading term}\label{Sec:Analysis_fJMGT_W_others}
We next discuss to what extent the analysis we have performed for the fJMGT--W III equation carries over to the other versions. To this end, the crucial question is whether the Galerkin approximation procedure based on energy estimates is feasible for a linearized equation:
\begin{equation} \label{linearized_fMGTI}
\tau^\alpha \Dt^{2+\alpha}\psi + (1+\sigma) \psi_{tt} - c^2\Delta\psi -\tau^\alpha c^2\Dt^\alpha\Delta \psi - \delta \Dt^\beta \Delta \psi = f\,, 
\end{equation}
with 
\begin{equation} \label{def_beta}
\beta = \begin{cases} 1 &\mbox{ for fMGT},\\2-\alpha &\mbox{ for fMGT I},\\\alpha &\mbox{ for fMGT II},
\end{cases}
\end{equation}
where for the respective nonlinear versions of Westervelt type we have $\sigma=2k\psi_t$ in mind. These three cases have in common the fact that, unlike with the fMGT III equation, the leading-order time derivative is fractional and varies with $\alpha$. \\
\indent The analysis of the fMGT and fMGT I equations follows similar lines of reasoning. We thus present the proof for the fMGT I equation with details and only outline the main arguments in the analysis of the fMGT model. To facilitate the analysis, we assume that $\alpha \in (1/2, 1)$, which, at least for the equation based on employing (GFE I),  appears to be the physically justified range according to the numerical experiments performed in~\cite{zhang2014time}. \\
\indent The analysis of the fMGT II equation cannot be carried out in the same manner; we explain why in Remark~\ref{Remark:fMGT_II} below and a offer different way of analyzing it when $\sigma=0$ in Section~\ref{Sec:WaveEq_Memory}.
\subsection{Analysis of the fJMGT--W I equation} 
In this section we consider the equation 
\begin{equation}
\tau^\alpha \Dal \psi_{tt}+(1+2k\psi_t)\psi_{tt}-c^2 \Delta \psi -\tau^\alpha c^2 \Dal \Delta \psi- \delta \Dtal \D\psi=f.
\end{equation}
To carry out the analysis starting from a linearization, we need to assume that the coefficient $\sigma$ is small enough in a suitable norm. To this end, for $\alpha \in (1/2, 1)$, we introduce the space
\[
X^\sigma_\fmgti = L^2(0,T;(W^{1,3}\cap L^\infty)(\Om))
\]
{equipped with the norm $\|\cdot\|_{X^\sigma_\fmgti}$.} Further, the solution space for $\psi$ will be
\begin{equation} \label{Xfmgti}
\begin{aligned}
\quad \spaceI=\,  \left\{ \psi\in H^{2+\alpha}(0,T; L^2(\Om)): \right. & \left. \Dal \psi \in L^\infty(0,T; \Honetwo),\right. \\ & \left. \Dt^{1+\alpha} \psi \in L^\infty(0,T; H_0^1(\Om)) \right\}, 
\end{aligned}
\end{equation}
{equipped with the norm $\|\cdot\|_{\spaceI}$}, which is the space induced by the part of the energy that can be bounded by a uniform constant independent of $\alpha$; see estimate \eqref{energy_est_I} in Proposition~\ref{Prop:fMGT_I} below.

Since the coefficient $\sigma$ acts as a placeholder for $2k \psi_t$, we note that
\[ \|2k \psi_t\|_{X^\sigma_\fmgti} \lesssim \|\psi\|_{\spaceI},\]
due to interpolation 
\[ 
\begin{aligned}
H^\alpha(0,T;\Honetwo)\cap H^{1+\alpha}(0,T;H_0^1(\Om))\subseteq&\, H^{\alpha+\theta}(0,T;H^{1+(1-\theta)}(\Om))\\
=&\, H^1(0,T;H^{1+\alpha}(\Om))
\end{aligned}
\]
with $\theta=1-\alpha$ and continuity of the embedding $H^{1+\alpha}(\Om)\hookrightarrow (W^{1,3}\cap L^\infty)(\Om)$.

We are now ready to analyze equation \eqref{linearized_fMGTI} for $\beta=2-\alpha$.
\begin{proposition}[Well-posedness of the \textup{fMGT I} equation]\label{Prop:fMGT_I}
	Let  $\alpha\in [\alpha_0,1)$ for some $\alpha_0 > 1/2$. Assume that $f \in H^{\alpha-1/2}(0,T;L^2(\Om))$, $\sigma \in X^\sigma_\fmgti$, and \[(\psi_0, \psi_1, \psi_2) \in (\Honetwo, \Honetwo, H_0^1(\Om)).\] There exists $\varrho>0$, independent of $\alpha$, 
	such that if \[\|\sigma\|_{X^\sigma_\fmgti} \leq \varrho,\] then there is a unique $\psi \in X_\fmgti$, which satisfies the \textup{fMGT I} equation in the $L^2(0,T; L^2(\Om))$ sense with $(\psi, \psi_t, \psi_{tt})\vert_{t=0}=(\psi_0, \psi_1, \psi_2)$.  Furthermore, this solution fulfills the following estimate:
	\begin{equation} \label{energy_est_I}
	\begin{aligned}
	& \begin{multlined}[t]\|\Dt^{2+\alpha}\psi\|_{L^2_t(L^2)}^2
	+ \nLtwo{\nabla \Dt^{1+\alpha}\psi(t)}^2
	+\nLtwo{\D \Dt^{\alpha}\psi(t)}^2\\
	+ C(\alpha)\Bigl(\|\nabla\ptt\|_{{H_t^{-(1-\alpha)/2}}(L^2)}^2
	+ \|\Dt^{3/2}\D \psi\|_{L^2_t(L^2)}^2\Bigr) \end{multlined}
	\\
	\lesssim&\, \|f\|_{H^{\alpha-1/2}(L^2)}^2+\nLtwo{\D \psi_0}^2+\nLtwo{\D \psi_1}^2+\nLtwo{\nabla \psi_2}^2,
	\end{aligned}
	\end{equation}
	where $C(\alpha)\to0$ as $\alpha\to1^-$. 
\end{proposition}
\begin{proof}
	The proof follows by discretizing the problem with respect to the spatial variable, using smooth eigenfunctions of the Dirichlet-Laplacian as the basis.\\
	
	\noindent (I) \emph{Existence of an approximate solution.} For $n \in \N$ fixed, we first prove that the semi-discrete problem has a unique solution. We employ the same notation as in the proof of Proposition~\ref{Prop:fMGT_III_lower}; that is,
	\begin{equation}
	\begin{aligned}
	\psi^n(x,t) =\sum_{i=1}^n \xi^n_i(t)\phi_i(x),\quad \psi_{j}^n(x)= \sum_{i=1}^n \xi^n_{j, i} \phi_i(x), \quad j\in\{0,1,2\}.
	\end{aligned}
	\end{equation}
	Using the mass matrices $M$ and $M_\sigma=M_\sigma(t)$, the stiffness matrix $K$, and the source vector $\boldsymbol{f}$ defined in \eqref{matrices}, the next step is to rewrite the discretized problem as a system of integral Volterra equations. To this end, let 
	\[\boldsymbol{\mu}=\Dt^{2+\alpha}\boldsymbol{\xi} \qquad \text{and} \qquad  p^\gamma(t)=\frac{1}{\Gamma(\gamma+1)}t^\gamma.\]
	We can rely on the following identities:
	\[\Igamma w=p^{\gamma-1}*w,\qquad \Igamma p^i= p^{\gamma+i}, \qquad \Igamma \Is = \Igammas\]  to rewrite the vector solution and its derivatives as
	\[
	\begin{aligned}
	\boldsymbol{\xi}_{tt}&= \Ione\bxi_{ttt}+ p^0 \bxi_2 
	= \Ialpha \bmu+ p^0 \boldsymbol{\xi}_2
	= p^{\alpha-1}*\bmu + p^0 \bxi_2,\\
	\boldsymbol{\xi}_t&= \Ione \bxi_{tt}+ p^0 \bxi_1  
	= \textup{I}^{1+\alpha}\boldsymbol{\mu}+ p^1 \boldsymbol{\xi}_2+ p^0 \bxi_1
	= p^{\alpha}*\bmu + p^1 \boldsymbol{\xi}_2+ p^0 \boldsymbol{\xi}_1,\\
	\boldsymbol{\xi}&= \Ione \boldsymbol{\xi}_{t}+ p^0 \boldsymbol{\xi}_0 
	= \textup{I}^{2+\alpha}\boldsymbol{\mu}+ p^2 \boldsymbol{\xi}_2+ p^1 \boldsymbol{\xi}_1+ p^0 \boldsymbol{\xi}_0
	= p^{\alpha+1}*\bmu + p^2 \boldsymbol{\xi}_2+ p^1 \boldsymbol{\xi}_1+ p^0 \boldsymbol{\xi}_0.
	\end{aligned}
	\]
	Furthermore, we can rewrite the fractional derivatives as
	\[
	\begin{aligned}
	\Dt^{2-\alpha}\boldsymbol{\xi}&= \Ialpha\boldsymbol{\xi}_{tt} 
	= \Italpha \boldsymbol{\mu}+ \Ialpha p^0 \boldsymbol{\xi}_2
	= p^{2\alpha-1}*\bmu + p^\alpha \boldsymbol{\xi}_2,\\
	\Dt^{\alpha}\boldsymbol{\xi}&= \textup{I}^{1-\alpha}\boldsymbol{\xi}_t
	=I^{2}\bmu+ \textup{I}^{1-\alpha} p^1 \bxi_2+ \textup{I}^{1-\alpha} p^0 \boldsymbol{\xi}_1
	= p^{1}*\bmu +p^{2-\alpha} \boldsymbol{\xi}_2+  p^{1-\alpha} \boldsymbol{\xi}_1.\\
	\end{aligned}
	\]
	Therefore, the semi-discrete problem can be equivalently rewritten as a system of Volterra integral equations:
	\begin{equation} \label{Volterra_system_I}
	\begin{aligned}
	&\begin{multlined}[t]\tau^\alpha M \bmu 
	+ M_\sigma(t)\Bigl(p^{\alpha-1}*\bmu + p^0 \bxi_2\Bigr)\\
	+c^2 K \Bigl(p^{\alpha+1}*\bmu + p^2 \boldsymbol{\xi}_2+ p^1 \boldsymbol{\xi}_1+ p^0 \boldsymbol{\xi}_0\Bigr)\\
	+\tau^\alpha c^2 K \Bigl(p^{1}*\bmu +p^{2-\alpha} \boldsymbol{\xi}_2+  p^{1-\alpha} \boldsymbol{\xi}_1\Bigr)
	+\delta K (p^{2\alpha-1}*\bmu + p^\alpha \boldsymbol{\xi}_2)=f.\end{multlined}
	\end{aligned}
	\end{equation}
	Thus, unique solvability of this system in $L^2(0,T)$ follows from~\cite[Theorem 4.2, p. 241 in \S 9]{GLS90}. Then from	\[ \left \{
	\begin{aligned}
	&\Dt^{\alpha}\boldsymbol{\xi}_{tt} =\bmu \in L^2(0,T), \quad \alpha \in (\tfrac12, 1)\\
	& \boldsymbol{\xi}_{tt}(0)=\bxi_{2}
	\end{aligned} \right.
	\]
	we have a unique $\bxi_{tt} \in H^\alpha(0,T)$; cf.~\cite[\S 3.3]{kubica2020time}. Combined with the initial conditions $(\bxi_0, \bxi_1)$, this yields a unique $\bxi \in H^{2+\alpha}(0,T)$ and further implies the existence of a unique $\pn \in H^{2+\alpha}(0,T; V_n)$. \\
	
	\noindent (II) \emph{A priori energy analysis.}  We next focus on deriving a uniform energy estimate for $\pn$. We will make use of estimate \eqref{coercivityI} to treat the fractional terms; that is,
	\begin{equation}
	\int_0^t \langle \textup{I}^\rho w(s),  w(s) \rangle \ds \geq \cos ( \tfrac{\pi\rho}{2} ) \| w \|_{H^{-\rho/2}(0,t)}^2 
	\end{equation} 
	for $\rho\in(0,1)$, as well as the identity $\int_0^t \langle w_t(s),  w(s) \rangle \ds = \frac12|w|^2\,\big\vert_0^t$. Thus, the rule of thumb is that for a coercivity estimate on $\int_0^t \Dt^r(s) w\, \Dt^\rho w(s)\ds$  to yield a non-negative lower bound (up to initial data), the difference $|r-\rho|$ between the fractional orders must not exceed one. We will consider the multiplier \[-\Delta \Dt^{1+\alpha}\pn(t)= \sum_{i=1}^n \Dt^{1+\alpha} \xi^n_i(t) \,\lambda_i \phi_i(x) \in V_n,\] for which this rule applies and yields non-negative contributions on the left-hand side for the terms containing $\Dt^{2+\alpha}\pn$, $\pttn$, $-\Dt^\alpha\Delta \pn$, and $-\Dt^{2-\alpha} \Delta \pn$. Multiplying the semi-discrete equation with $-\Delta \Dt^{1+\alpha} \pn$ and integrating over space and $(0,t)$ at first leads to
	\begin{equation}
	\begin{aligned}
	&\, \begin{multlined}[t]-\tau^\alpha \intt \prodLtwo{ \Dt^{2+\alpha} \psi^n}{\D \Dt^{1+\alpha} \pn}\ds-\intt \prodLtwo{ \pttn}{\D \Dt^{1+\alpha} \pn}\ds \\
	+ \tau^\alpha c^2\intt \prodLtwo{\Dal \D \pn}{\D \Dt^{1+\alpha} \pn}\ds+\delta \intt \prodLtwo{\Dt^{2-\alpha} \D \pn}{\D \Dt^{1+\alpha} \pn}\ds
	\end{multlined} \\
	=&\, \begin{multlined}[t] -\intt (f-\sigma \pttn+c^2 \D \pn, \Delta \Dt^{1+\alpha} \pn)\ds.
	\end{multlined}
	\end{aligned}
	\end{equation}
	We next exchange the order of differentiation in the first and third term on the left as follows:
	\begin{equation}\label{exchange_derivative}
	\Dt^{2+\alpha}\psi^n = \Dt \Dt^{1+\alpha}\psi^n-p^{-\alpha}\psi^n_{tt}(0)\,, \qquad
	\Dt^{1+\alpha}\psi^n = \Dt \Dt^\alpha\psi^n-p^{-\alpha}\psi^n_{t}(0)\,, 
	\end{equation}
	and integrate by parts to obtain 
	\[
	\begin{aligned}
	&\intt \prodLtwo{ \Dt^\alpha\D\pn(s)}{p^{-\alpha}(s)\D \ptn(0)}\ds\\
	=&\, \begin{multlined}[t]-\intt \prodLtwo{ \Dt^{1+\alpha}\D\pn(s)+p^{-\alpha}(s)\D \ptn(0)}{p^{1-\alpha}\D \ptn(0)}\ds\\
	+p^{1-\alpha}(t)\prodLtwo{ \Dt^\alpha\D\pn(t)}{\D \ptn(0)} \end{multlined}\\
	=&\, -\intt \prodLtwo{ \Dt^{1+\alpha}\D\pn(s)}{p^{1-\alpha}(s)\D \ptn(0)}\ds + h_0(t),
	\end{aligned}
	\]
	where we have introduced
	\[h_0(t)=\frac{t^{2-2\alpha}}{2\Gamma(2-2\alpha)^2} \|\D \ptn(0)\|_{L^2}^2
	+p^{1-\alpha}(t)\prodLtwo{ \Dt^\alpha\D\pn(t)}{\D \ptn(0)}.\]
	Thus, we arrive at the following identity:
	\begin{equation}
	\begin{aligned}
	&\, \begin{multlined}[t]-\tau^\alpha \intt \prodLtwo{ \Dt \Dt^{1+\alpha}\psi^n}{\D \Dt^{1+\alpha} \pn}\ds-\intt \prodLtwo{ \pttn}{\D \Dt^{1+\alpha} \pn}\ds \\
	+ \tau^\alpha c^2\intt \prodLtwo{\Dal \D \pn}{\D \Dt \Dt^\alpha\psi^n}\ds+\delta \intt \prodLtwo{\Dt^{2-\alpha} \D \pn}{\D \Dt^{1+\alpha} \pn}\ds
	\end{multlined} \\
	=&\, \begin{multlined}[t] -\intt \prodLtwo{f-\sigma \pttn+c^2 \D \pn+\tau^\alpha p^{-\alpha}\pttn(0)-\tau^\alpha c^2 p^{1-\alpha}\D \ptn(0)}
	{\Delta \Dt^{1+\alpha} \pn}\ds\\
	\hspace*{-5cm}+\tau^\alpha c^2 h_0(t).
	\end{multlined}
	\end{aligned}
	\end{equation}
	We note that \[\Dt^{1+\alpha}\pn=\pttn=0 \text{ on } \partial \Om\] with our choice of the basis functions. Additionally using $\Dt^{1+\alpha}\pn=\textup{I}^{1-\alpha}\psi^n_{tt}$ in the second term on the left and integrating by parts in space and time yields 
	\begin{equation}\label{enid_rem}
	\begin{aligned}
%	\textup{lhs}:=	
&\begin{multlined}[t] \,\frac{\tau^\alpha}{2}\nLtwo{\nabla \Dt^{1+\alpha}\pn(s)}^2 \big\vert_0^t+\int_0^t\prodLtwo{\nabla\pttn(s)}{\textup{I}^{1-\alpha}  \nabla\pttn(s)}\ds
	\\	+\frac{\tau^\alpha c^2}{2}\nLtwo{\D \Dt^{\alpha}\pn(s)}^2 \Big \vert_0^t+\delta \dmp \end{multlined} \\
	=&\, 
	-\int_0^t\prodLtwo{\tilde{f}(s)}{\D \Dt^{1+\alpha}\pn(s)}\ds
	+\tau^\alpha c^2 h_0(t),  
	\end{aligned}
	\end{equation}
	where we have introduced the short-hand notation
	\begin{equation}\label{ftilde}
	\tilde{f}=f+c^2\D\pn-\sigma\pttn +\tau^{\alpha}\left(p^{-\alpha}\pttn(0)-c^2p^{1-\alpha}\D\ptn(0)\right)
	\end{equation}
	and
	\[
	d=\intt \prodLtwo{\Dt^{2-\alpha} \D \pn}{\D \Dt^{1+\alpha} \pn}\ds.
	\]
	By the identity
	\begin{equation}
	\begin{aligned}
	\Dt^{2-\alpha}\pn=&\, \Ialpha\psi_{tt}=\textup{I}^{2\alpha-1}\textup{I}^{1-\alpha}\pttn=\textup{I}^{2\alpha-1}\Dt^{1+\alpha}\pn,
	\end{aligned}
	\end{equation}
	we know that
	\begin{align} \label{damping_fMGT_I}
	\dmp = 	\int_0^t\prodLtwo{\textup{I}^{2\alpha-1} \Dt^{1+\alpha}\Delta\pn(s)}{\Dt^{1+\alpha}\Delta\pn(s)}\ds. 
	\end{align} 
	Since $\alpha>1/2$, this term can be estimated from below using the fact that $\textup{I}^\gamma:H^{-\gamma}(0,t)\to L^2(0,T)$ is an isomorphism for $\gamma \in [0, 1/2)$; see~\cite[Theorem 1]{gorenflo1999operator}. Therefore,
	\[
	\begin{aligned}
	\frac{\dmp}{\cos(\pi(\alpha-1/2))} \geq&\, 
	\|\Dt^{1+\alpha}\D \pn\|_{H_t^{1/2-\alpha}(L^2)}^2\\
	\sim&\, 
	\|\textup{I}^{\alpha-1/2}\Dt^{1+\alpha}\D \pn\|_{L^2_t(L^2)}^2\sim 
	\|\Dt^{3/2}\D \pn\|_{L_t^2(L^2)}^2. 
	\end{aligned}
	\]
	We estimate the $\tilde{f}$ term on the right-hand side of \eqref{enid_rem} with a view on the possibility of bounding it by $\dmp$ as follows: 
	\[
	\begin{aligned}
	\left|\int_0^t\prodLtwo{\tilde{f}(s)}{\Dt^{1+\alpha}\D \pn(s)}\ds\right|
	\leq 
	\frac{1}{2\epsilon}\|\tilde{f}\|_{H^{\alpha-1/2}_t(L^2)}^2+\frac{\epsilon}{2}\|\Dt^{1+\alpha}\D \pn\|_{H^{1/2-\alpha}_t(L^2)}^2.
	\end{aligned}
	\]
	It remains to estimate the terms within $\tilde{f}$ in this norm; that is, to bound
	\begin{equation}
	\|f+c^2\D\pn-\sigma\pttn +\tau^{\alpha}\left(p^{-\alpha}\pttn(0)-c^2p^{1-\alpha}\D\ptn(0)\right)\|_{H^{\alpha-1/2}_t(L^2)}.
	\end{equation}
	For the $c^2$ term, it is readily checked that the respective contribution of $c^2\D \pn$ to the above norm of $\tilde{f}$ (cf. \eqref{ftilde}) can be bounded by means of $c^2 \|\Dt^\alpha\D\pn\|_{L^2(L^2)}$ as follows:
	\[
	\begin{aligned}
	\|\D \pn \|_{H^{\alpha-1/2}_t(L^2)}
	\lesssim&\, \|\Dt^{\alpha-1/2}\D \pn \|_{L^2_t(L^2)} + \nLtwo{\D \pn_0}\\
	=&\, \|\kersing_{1/2}*\Dt^\alpha\D \pn \|_{L^2_t(L^2)} + \nLtwo{\D \pn_0}\\
	\leq&\, \|\kersing_{1/2}\|_{L^1(0,T)} \|\Dt^\alpha\D \pn \|_{L^2_t(L^2)} + \nLtwo{\D \pn_0}\,,
	\end{aligned}
	\]
	and therefore tackled by the second term on the left-hand side of \eqref{enid_rem} together with Gronwall's inequality.\\
	\indent By the Kato--Ponce inequality \eqref{prodruleest} with 
	\[\rho=\alpha-1/2,\qquad (p_1, q_1)=\left(\frac{2}{2 \alpha -1}, \frac{1}{1-\alpha}\right), \qquad (p_2, q_2)=(2, \infty),\]
	we obtain 
	\begin{equation}\label{estsigmapsitt_I}
	\begin{aligned}
	&\|\sigma\pttn\|_{H_t^{\alpha-1/2}(L^2)}
	\\
	\lesssim&\,
	\|\sigma\|_{W_t^{\alpha-1/2,\frac{2}{2\alpha-1}}(L^\infty)}\|\pttn\|_{L_t^{\frac{1}{1-\alpha}}(L^2)}
	+\|\sigma\|_{L_t^{2}(L^\infty)}\|\pttn\|_{W_t^{\alpha-1/2, \infty}(L^2)}.
	\end{aligned}
	\end{equation}
	By the Sobolev embedding $L^2(0,T)\hookrightarrow W^{\alpha-1/2,\frac{2}{2 \alpha-1}}(0,T)$, we have
	\[
	\|\sigma\|_{W_t^{\alpha-1/2, \frac{2}{2 \alpha-1}}(L^\infty)}
	\leq C_{H^2,L^\infty}^\Om\|\sigma\|_{X^\sigma_\fmgti}\,, \quad
	\|\sigma\|_{L_t^2(L^\infty)}
	\leq C_{H^2,L^\infty}^\Om\|\sigma\|_{X^\sigma_\fmgti}\,.
	\]
	To further estimate the norms of $\pttn$ in \eqref{estsigmapsitt_I}, we will use the leading time derivative term $\Dt^{2+\alpha}\pn$ as well as its representation via the PDE. That is, we rely on the following Sobolev embeddings:
	\begin{equation}
	\begin{aligned}
	\|\pttn\|_{L_t^{\frac{1}{1-\alpha}}(L^2)}\lesssim&\,
	\|\pttn\|_{H^{\alpha-1/2}_t(L^2)}\lesssim
	\|\Dt^{2+\alpha}\pn\|_{H^{-1/2}_t(L^2)}+\|\pttn(0)\|_{L^2}
	\end{aligned}
	\end{equation}
	and
	\begin{equation}
	\begin{aligned}
	\|\pttn\|_{W_t^{\alpha-1/2, \infty}(L^2)}\lesssim&\, \|\pttn\|_{H^\alpha_t(L^2)}\lesssim \|\Dt^{2+\alpha}\pn\|_{L^2_t(L^2)}+\|\pttn(0)\|_{L^2},
	\end{aligned}
	\end{equation}
	where
	$\pttn$ satisfies the fractional ODE 
	\[
	\tau^\alpha\Dt^\alpha \pttn + \pttn = -r\mbox{ with }r=\sigma \pttn - c^2\Delta\psi -\tau^\alpha c^2\Dt^\alpha\Delta \pn - \delta \Dt^{2-\alpha} \Delta \pn -f
	\]
	and therefore 
	\begin{equation}\label{fracODE}
	\Dt^\alpha \pttn = \tau^{-\alpha}\Bigl(-E_{\alpha,1}(-(\tfrac{t}{\tau})^\alpha)\pttn(0)+\int_0^t E_{\alpha,\alpha}(-(\tfrac{t-s}{\tau})^\alpha) r(s)\ds-r(t)\Bigr);
	\end{equation} 
	see, e.g.,~\cite[\S 3]{kubica2020time}. Thus, we have
	\begin{equation}\label{Dt2plusalphapn}
	\begin{aligned}
	&\|\Dt^{2+\alpha}\pn\|_{H^{-1/2}_t(L^2)}\lesssim
	\|\Dt^{2+\alpha}\pn\|_{L^2_t(L^2)}\lesssim \|r\|_{L^2_t(L^2)}+\|\pttn(0)\|_{L^2}\\
	\leq&\,\begin{multlined}[t]\|\sigma \pttn\|_{L^2_t(L^2)} + c^2\|\Delta\psi\|_{L^2_t(L^2)} +\tau^\alpha c^2\|\Dt^\alpha\Delta \pn\|_{L^2_t(L^2)}\\
	+ \delta \|\Dt^{2-\alpha} \Delta \pn\|_{L^2_t(L^2)} +\|f\|_{L^2(L^2)}+\|\pttn(0)\|_{L^2}.\end{multlined}
	\end{aligned}
	\end{equation}
	In here, the terms with factors $c^2$, $\tau^\alpha c^2$, and $\delta$ can be controlled -- in a (generalized) Gronwall inequality fashion --  by left-hand side terms in \eqref{enid_rem}; to see this for the latter, consider  
	\[\|\Dt^{2-\alpha} \Delta \pn\|_{L^2_t(L^2)}=\|\textup{I}^{\alpha-1/2}\Dt^{3/2} \Delta \pn\|_{L^2_t(L^2)}\lesssim \kersing_{3/2-\alpha}*\dmp.\]
	Thus, from \eqref{estsigmapsitt_I} to \eqref{Dt2plusalphapn}, we have obtained an estimate of the form
	\[
	\begin{aligned}
	\|\sigma\pttn\|_{H_t^{\alpha-1/2}(L^2)}\lesssim&\, \|\sigma\|_{X^\sigma_\fmgti} \left(\|\Dt^{2+\alpha}\pn\|_{L^2_t(L^2)}+\|\pttn(0)\|_{L^2}\right)\\
	\lesssim&\,  \|\sigma\|_{X^\sigma_\fmgti} \left(\|\sigma\pttn\|_{H_t^{\alpha-1/2}(L^2)} 
+ \textup{rhs} \right),
	\end{aligned}
	\]
where 
\[\textup{rhs}:=
 c^2\|\Delta\pn\|_{L^2_t(L^2)} +\tau^\alpha c^2\|\Dt^\alpha\Delta \pn\|_{L^2_t(L^2)}
	+ \kersing_{3/2-\alpha}*\dmp +\|f\|_{L^2(L^2)} +\|\pttn(0)\|_{L^2}.
\]
	Thus, provided $\|\sigma\|_{X^\sigma_\fmgti}$ is sufficiently small (where the bound can be chosen independent of $\alpha\in[\alpha_0,1)$ for $\alpha_0>1/2$),
	the term $\|\sigma\pttn\|_{H_t^{\alpha-1/2}(L^2)}$ is bounded by 
a multiple of \textup{rhs}. By combining this with \eqref{enid_rem}, \eqref{Dt2plusalphapn}, and Gronwall's inequality in its generalized version, see, e.g., \cite[Lemma 7.2]{kubica2020time}, we therefore obtain the following estimate:
	\begin{equation} \label{discrete_est_I}
	\begin{aligned}
	& \begin{multlined}[t]\|\Dt^{2+\alpha}\pn\|_{L^2_t(L^2)}^2
	+ \nLtwo{\nabla \Dt^{1+\alpha}\pn(t)}^2
	+\nLtwo{\D \Dt^{\alpha}\pn(t)}^2\\
	+ C(\alpha)(\|\nabla\pttn\|_{{H_t^{-(1-\alpha)/2}}(L^2)}^2
	+ \|\Dt^{3/2}\D \pn\|_{L^2_t(L^2)}^2) \end{multlined}
	\\
	\lesssim&\, \|f\|_{H^{\alpha-1/2}(L^2)}^2+\nLtwo{\D \psi_0}^2+\nLtwo{\D \psi_1}^2+\nLtwo{\nabla \psi_2}^2,
	\end{aligned}
	\end{equation}
	where we have also relied on the uniform boundedness of the approximate data. 
{Here the constant $C(\alpha)$ tends to zero as $\alpha\to1^-$, since it contains the factor $\cos ( \tfrac{\pi(1-\alpha)}{2} )$ from the coercivity estimate \eqref{coercivityI}.
}
\\
	
	\noindent (III) \emph{Passing to the limit.} Thanks to the uniform bound \eqref{discrete_est_I}, there exists a subsequence, which we do not relabel, such that
	\begin{equation} \label{weak_limits_I}
	\begin{alignedat}{4} 
	\Dt^{2+\alpha} \pn  &\relbar\joinrel\rightharpoonup \Dt^{2+\alpha} \psi  &&\text{ weakly}  &&\text{ in } &&L^2(0,T; L^2(\Omega)),  \\
	\Dal \D \pn &\relbar\joinrel\rightharpoonup \Dal  \D  \psi &&\text{ weakly-$\star$} &&\text{ in } &&L^\infty(0,T; L^2(\Om)), \\
	\Dt^{1+\alpha}\nabla \pn  &\relbar\joinrel\rightharpoonup \Dt^{1 +\alpha} \nabla \psi &&\text{ weakly-$\star$}  &&\text{ in } &&L^\infty(0,T; L^2(\Omega)),  \\
	\Dt^{3/2} \D \pn &\relbar\joinrel\rightharpoonup \Dt^{3/2}\D\psi &&\text{ weakly} &&\text{ in } &&L^2(0,T; L^2(\Omega)).
	\end{alignedat} 
	\end{equation} 
	Furthermore,
	\begin{equation} \label{weak_limits_I_1}
	\begin{alignedat}{4} 
	(1+\sigma)\pttn  &\relbar\joinrel\rightharpoonup  (1+\sigma)\ptt  &&\text{ weakly} &&\text{ in } &&L^2(0,T; L^2(\Omega)),\\
	\D \pn &\relbar\joinrel\rightharpoonup  \D \psi &&\text{ weakly} &&\text{ in } &&L^2(0,T; L^2(\Omega)), \\
	\Dtal \D \pn  &\relbar\joinrel\rightharpoonup  \Dtal \D \psi &&\text{ weakly} &&\text{ in } &&L^2(0,T; L^2(\Omega)).  
	\end{alignedat} 
	\end{equation}
	Thus, we can pass to the limit in the usual way in the semi-discrete problem. Further, weak/weak-$\star$ lower semi-continuity of norms implies that the solution we constructed satisfies \eqref{energy_est_I} a.e. in time.\\

	\noindent (IV) \emph{Attainment of the initial conditions.} Similarly to step (IV) in the proof of Proposition~\ref{Prop:fMGT_III_lower}, we show that $\psi$ attains its initial conditions by, on one hand concluding from \eqref{weak_limits_I} that
	\[
	\begin{aligned}
	&\pn (0) \relbar\joinrel\rightharpoonup \psi(0) \text{ weakly}  \text{ in } \Honetwo,\\  
	&\ptn (0) \relbar\joinrel\rightharpoonup \psi_t(0) \text{ weakly}  \text{ in } H_0^1(\Om),\\  
	&\pttn (0) \relbar\joinrel\rightharpoonup \psi_{tt}(0) \text{ weakly}  \text{ in } L^ 2(\Om),
	\end{aligned}	\]
	and, on the other hand, $\pn(0) \rightarrow \psi_0$ in $\Honetwo$, $\ptn(0) \rightarrow \psi_1$ in $H_0^1(\Om)$, $\pttn(0) \rightarrow \psi_1$ in $L^2(\Om)$ based on our choice of the approximate data. Thus the initial data are attained in an $\Honetwo\times H_0^1(\Om)\times L^2(\Om)$ sense.\\

	\noindent (IV) \emph{Uniqueness.} The fact that the obtained solution is unique follows by testing the homogeneous problem
	\[
	\tau^\alpha \Dt^{2+\alpha}\psi + (1+\sigma) \psi_{tt} - c^2\Delta\psi -\tau^\alpha c^2\Dt^\alpha\Delta \psi - \delta \Dt^{2-\alpha} \Delta \psi = 0
	\] (with zero initial data) with $\Dt^{1+\alpha} \pt$. Analogously to above, but replacing $\D\to\nabla$ and $\nabla\to\mbox{id}$, we obtain
	\[
	\begin{aligned}
	& \begin{multlined}[t]\frac{\tau^\alpha}{2}\nLtwo{\Dt^{1+\alpha}\psi(t)}^2
	+ \frac{\tau^\alpha c^2}{2}\nLtwo{\nabla\Dt^{\alpha}\psi(t)}^2+C(\alpha)(\|\psi_{tt}\|_{H_t^{-(1-\alpha)}(L^2)}^2
	+  \|\Dt^{3/2} \nabla \psi\|_{L^2_t(L^2)}^2)\end{multlined}
	\\
	\leq&\, \left |\int_0^t\prodLtwo{c^2\D\psi-\sigma\psi_{tt}}{\Dt^{1+\alpha}\psi}\ds \right| \\
	\leq&\, \frac{1}{2\epsilon}\Bigl(\|c^2\nabla\psi\|_{H_t^{\alpha-1/2}(L^2)}+\|\sigma\psi_{tt}\|_{H_t^{\alpha-1/2}(H^{-1})}\Bigr)^2+\frac{\epsilon}{2} \|\Dt^{3/2}\nabla\psi\|_{L^2_t(L^2)}^2,
	\end{aligned}
	\]
	where 
	\[
	\begin{aligned}
	\|c^2\nabla\psi\|_{H^{\alpha-1/2}(L^2)}&\lesssim \|\nabla\Dt^{\alpha}\psi\|_{L^2(L^2)}.
	\end{aligned}
	\]
	Further, on account of the following estimate:
	\begin{equation} \label{ab_H-1}
	\begin{aligned}
	\|ab\|_{H^{-1}(\Om)}
	=& \, \|a\|_{H^{-1}(\Om)}\sup_{v\in H_0^1(\Om)\setminus\{0\}} \|v\|_{H_0^1(\Om)}^{-1} \nLtwo{v \nabla b + b\nabla v}\\
	\leq&\, \|a\|_{H^{-1}(\Om)}(C_{H^1,L^6}\|\nabla b\|_{L^3}+\|b\|_{L^\infty}),
	\end{aligned}
	\end{equation}
	we have, similarly to \eqref{estsigmapsitt_I},
	\begin{equation}\label{sigmaptt}
	\begin{aligned}
	\|\sigma\ptt\|_{H_t^{\alpha-1/2}(H^{-1})}
	%	\lesssim\|\partial_t^{\alpha-1/2}(\sigma\ptt)\|_{L_t^2(H^{-1})}\\
	\lesssim\,
	\|\sigma\|_{X^\sigma_{\textup{fMGT I}}}\|\ptt\|_{L_t^{\frac{1}{1-\alpha}}(H^{-1})}
	+\|\sigma\|_{X^\sigma_{\textup{fMGT I}}}\|\ptt\|_{W_t^{\alpha-1/2, \infty}(H^{-1})}.
	\end{aligned}
	\end{equation}
	Again using \eqref{fracODE} with $\psi$ in place of $\pn$ yields
	\begin{equation}
	\begin{aligned}
	&\|\ptt\|_{L_t^{\frac{1}{1-\alpha}}(H^{-1})}	+\|\ptt\|_{W_t^{\alpha-1/2, \infty}(H^{-1})} \lesssim\, \|\Dt^{2+\alpha} \psi\|_{L_t^2(H^{-1})}\\
	\lesssim&\, \|\sigma \psi_{tt} - c^2\Delta\psi -\tau^\alpha c^2\Dt^\alpha\Delta \psi - \delta \Dt^{2-\alpha} \Delta \psi\|_{L_t^2(H^{-1})},
	\end{aligned}
	\end{equation}
	where $\|\sigma \psi_{tt}\|_{L_t^2(H^{-1})}\lesssim \|\sigma \psi_{tt}\|_{H_t^{\alpha-1/2}(H^{-1})}$. Therefore, these terms can be absorbed for small enough $\sigma$ by the left-hand side or handled by Gronwall's inequality to conclude that $\psi=0$.
\end{proof}
%%%%%%%%%%%%%%%%%%%%%%%%%%%%%%%%%%%%%%%%%
We next prove a well-posedness result for the corresponding nonlinear problem. To guarantee that the coefficient $\sigma$ is small enough in the fixed-point iteration, we impose a smallness condition on the data.
\begin{theorem}[Local well-posedness of the fJMGT--W I equation] \label{Thm:fJMGT_W_I} Let $\alpha \in [\alpha_0,1)$ for some $\alpha_0 > 1/2$ and $T>0$. Further, assume that $ f \in H^{\alpha-1/2}(0,{T};L^2(\Om))$. There exists $\varrho=\varrho(\alpha, T)>0$, such that if
	\begin{equation}
	\|f\|^2_{H^{\alpha-1/2}(L^2)}+\|\psi_0\|_{H^2}^2+\|\psi_1\|_{H^2}^2+\|\psi_2\|_{H^1}^2 \leq \varrho^2,
	\end{equation}
	then the initial boundary-value problem
	\begin{equation}
	\left \{
	\begin{aligned}
	\tau^\alpha \Dt^{2+\alpha}\psi + &(1+2k\pt) \psi_{tt} - c^2\Delta\psi -\tau^\alpha c^2\Dt^\alpha\Delta \psi - \delta \Dt^{2-\alpha} \Delta \psi = f\hspace*{-2mm}&&\text{in }\Omega\times(0,T), \\[1mm]
	&\psi=\,0&&\text{on }\partial\Omega\times(0,T),\\[1mm]
	&(\psi, \psi_t, \psi_{tt})=\,(\psi_0, \psi_1, \psi_2)&&\mbox{in }\Omega\times \{0\},
	\end{aligned} \right.
	\end{equation}
	has a unique solution $\psi \in X_{\textup{fMGT I}}$, which satisfies	
	\begin{equation} \label{energy_est2_fJMGT_W_I}
	\begin{aligned}
	\begin{multlined}[t] \|\psi\|^2_{X_{\textup{fMGT I}}}
	\end{multlined}
	\lesssim\,\begin{multlined}[t] 	\|f\|^2_{H^{\alpha-1/2}(L^2)}+\|\psi_0\|_{H^2}^2+\|\psi_1\|_{H^2}^2+\|\psi_2\|_{H^1}^2. \end{multlined}
	\end{aligned}
	\end{equation}
\end{theorem}
\begin{proof}
	The proof follows by setting up a fixed-point mapping $\mathcal{T}:w \mapsto \psi$, which associates
	\[
	w\in B_R:=\{w \in X_{\textup{fMGT I}}\, : \|w\|_{ X_{\textup{fMGT I}}}\leq R, \ w(0)=\psi_0, \, w_t(0)=\psi_1, \, w_{tt}(0)=\psi_2  \} 
	\]
	with the solution $\psi$ of the linearized problem \eqref{linearized_fMGTI} with $\sigma=2k w_t$. We recall that
	\[
	X^\sigma_\fmgti = L^2(0,T;(W^{1,3}\cap L^\infty)(\Om)),
	\]
	and so
	\[
	\|\sigma\|_{X^\sigma_{\textup{fMGT I}}}= 2|k|\|w_t\|_{X^\sigma_{\textup{fMGT I}}} \leq C \|w\|_{ X_{\textup{fMGT I}}} \lesssim R.
	\]
	Thus, $\sigma$ can be made small enough by decreasing $R>0$. The self-mapping is thus an immediate consequence of the energy estimate \eqref{energy_est_I}, provided we choose $\varrho$ small enough, so that
	\[
	C(R, T) (\|f\|^2_{H^{\alpha-1/2}(L^2)}+\|\psi_0\|_{H^2}^2+\|\psi_1\|_{H^2}^2+\|\psi_2\|_{H^1}^2) \leq C(R, T) \varrho^2 \leq R^2.
	\]
	\indent We prove strict contractivity of this mapping next. Let $w^{(1)}$, $w^{(2)} \in B_R$. Denote $\psi^{(1)}=\mathcal{T} w^{(1)}$ and $\psi^{(2)}=\mathcal{T} w^{(2)}$. Contractivity of $\mathcal{T}$ follows by considering the difference equation for $\opsi= \psi^{(1)}-\psi^{(2)}$:
	\begin{equation} \label{fJMGT_W_I_diff_contract}
	\begin{aligned}
	&\tau^\alpha \Dt^{2+\alpha}\opsi + (1+2kw_t^{(1)})\opsi_{tt} - c^2\Delta\opsi -\tau^\alpha c^2\Delta \Dt^\alpha\opsi - \delta \Dt^{2-\alpha} \Delta \opsi_t+2k\psi^{(2)}_{tt}\overline{w}_t=0,
	\end{aligned}
	\end{equation}
	which is supplemented by zero initial conditions. Similarly to the proof of uniqueness in Proposition~\ref{Prop:fMGT_I}, testing with $\opsi_{tt}$ yields
	\[
	\begin{aligned}
	&\,\begin{multlined}[t] \frac{\tau^\alpha}{2}\nLtwo{\Dt^{1+\alpha}\opsi(t)}^2
	+ \frac{\tau^\alpha c^2}{2}\nLtwo{\nabla\Dt^{\alpha}\opsi(t)}^2\\\hspace*{2cm}+C(\alpha)(\|\psi_{tt}\|_{H_t^{-(1-\alpha)}(L^2)}^2
	+  \|\Dt^{3/2} \nabla \opsi\|_{L^2_t(L^2)}^2) \end{multlined}
	\\
	\leq&\, \left|\int_0^t\prodLtwo{-2k\psi^{(2)}_{tt}\overline{w}_t+c^2\D\opsi-2k w_t^{(1)}\opsi_{tt}}{\Dt^{1+\alpha}\opsi}\ds \right| \\
	\leq&\, \begin{multlined}[t]\frac{\epsilon}{2} \|\Dt^{1+\alpha}\opsi\|_{L_t^\infty(L^2)}^2 
	+\frac{1}{2\epsilon} \|-2k\psi^{(2)}_{tt}\overline{w}_{t}\|_{L_t^1(L^2)}^2
	+\frac{\epsilon}{2} \|\Dt^{3/2}\nabla\opsi\|_{L^2_t(L^2)}^2 \\
	+\frac{1}{2\epsilon}\Bigl(\|c^2\nabla\opsi\|_{H_t^{\alpha-1/2}(L^2)}+\|2k w_t^{(1)}\opsi_{tt}\|_{H_t^{\alpha-1/2}(H^{-1})}\Bigr)^2.\end{multlined}
	\end{aligned}
	\]
	We can then rely on the following bound: 
	\[
	\begin{aligned}
	\|c^2\nabla \opsi\|_{H^{\alpha-1/2}(L^2)}&\lesssim \|\nabla\Dt^{\alpha}\opsi\|_{L^2(L^2)}
	\end{aligned}
	\]
	and, by \eqref{estsigmapsitt_I}--\eqref{Dt2plusalphapn}  with $2k w_t^{(1)}$ in place of $\sigma$, we have
	\[
	\begin{aligned}
	&\|2k w_t^{(1)}\psi_{tt}\|_{H^{\alpha-1/2}(H^{-1})}\lesssim \| w_t^{(1)}\|_{X^\sigma_{\textup{fMGT I}}} \|\Dt^{2+\alpha} \psi\|_{L^2(H^{-1})}\\
	\lesssim&\, \| w_t^{(1)}\|_{X^\sigma_{\textup{fMGT I}}}\|2k w_t^{(1)} \psi_{tt}- c^2\Delta\opsi -\tau^\alpha c^2\Delta \Dt^\alpha\opsi - \delta \Dt^{2-\alpha} \Delta \opsi_t\|_{L^2(H^{-1})}.
	\end{aligned}
	\]
	Thus for $\| w_t^{(1)}\|_{X_\sigma} \lesssim R$ small enough, similarly to the proof of Proposition~\ref{Prop:fMGT_I}, we obtain
	\begin{equation}
	\begin{aligned}
	& \frac{\tau^\alpha}{2}\nLtwo{\Dt^{1+\alpha}\opsi(t)}^2
	+ \frac{\tau^\alpha c^2}{2}\nLtwo{\nabla\Dt^{\alpha}\opsi(t)}^2+\|\opsi_{tt}\|_{H_t^{-(1-\alpha)}(L^2)}^2
	+  \|\Dt^{3/2} \nabla \opsi\|_{L^2_t(L^2)}^2
	\\
	\lesssim& \, R^2\|2k \psi_{tt}^{(2)}\overline{w}_t\|_{L^2(H^{-1})}^2+\|2k \psi_{tt}^{(2)}\overline{w}_{t}\|_{L^1(L^2)}.
	\end{aligned}
	\end{equation}
	Furthermore,
	\begin{equation}
	\begin{aligned}
	R^2\|2k \psi_{tt}^{(2)}\overline{w}_t\|_{L^2(H^{-1})}^2+\|2k \psi_{tt}^{(2)}\overline{w}_{t}\|_{L^1(L^2)} 
	\lesssim&\, R^2 \|\psi_{tt}^{(2)}\|^2_{L^\infty L^4}\|\overline{w}_{t}\|^2_{L^2(L^4)} \\
	\lesssim&\, R^2 \varrho^2 \|\overline{w}_{t}\|^2_{L^2(L^4)}.
	\end{aligned}
	\end{equation}
	We can further bound the last term as follows:
	\begin{equation}
	\begin{aligned}
	\|\nabla \overline{w}_t\|^2_{L^2(L^2)}=\|\textup{I}^{\gamma} \Dt^{1+\gamma}\nabla \overline{w}\|^2_{L^2(L^2)} \leq \|\kersing_{1-\gamma}\|_{L^1(0,T)} \|\Dt^{1+\gamma}\nabla \overline{w}\|^2_{L^2(L^2)},
	\end{aligned}
	\end{equation}
	choosing $\gamma=1/2$. Thus, by decreasing $\varrho>0$, we can guarantee that $\mathcal{T}$ is strictly contractive in the following norm:
	\[
	|||\psi|||= \nLinfLtwo{\Dt^{1+\alpha}\opsi}^2+ \nLinfLtwo{\nabla\Dt^{\alpha}\opsi}^2+\|\opsi_{tt}\|_{H^{-(1-\alpha)}(L^2)}^2+ \|\nabla \opsi\|_{L^2(L^2)}^2.
	\]
	The rest of the arguments follow as in Theorem~\ref{Thm:fJMGT_W_III} and complete the proof.	
\end{proof}
\subsection{Analysis of the fJMGT--W equation} To formulate the corresponding result for the fMGT-W equation
\[	
\tau^\alpha \Dal \psi_{tt}+(1+2k\psi_t)\psi_{tt}-c^2 \Delta \psi -\tau^\alpha c^2 \Dal \Delta \psi- \delta \D\psi_{t}=f        
\]
we again need smallness of the coefficient $\sigma$ in a suitable norm. To this end, let
\[
X^\sigma_\fmgt = H^{\alpha/2}(0,T;(W^{1,3}\cap L^\infty)(\Om))
\]
for $\alpha \in (1/2, 1)$, {and denote the corresponding norm by $\|\cdot\|_{X^\sigma_\fmgt}$}. We also introduce the solution space by
\begin{equation}
\begin{aligned}
\quad X_{\textup{fMGT}}=\,  \left\{ \psi\in H^{2+\alpha}(0,T; L^2(\Om)):\right.&\, \Dt^{1+\alpha/2} \psi \in L^\infty(0,T; \Honetwo), \\& \left. \Dt^{1+\alpha} \in L^\infty(0,T; H_0^1(\Om)) \right\},
\end{aligned}
\end{equation}
{equipped with the norm $\|\cdot\|_{X_{\textup{fMGT}}}$}. Note that with this choice, again  
\[
\|2k \psi_t\|_{X^\sigma_\fmgt} \lesssim \|\psi\|_{X_{\textup{fMGT}}}
\] 
holds but the energy term $\|\Dt^{1+\alpha/2}\D \psi\|_{L^2(L^2)}$ needed for this purpose comes with an $\alpha$-dependent coefficient in \eqref{est_fMGT}.
For this reason, while still being able to show well-posedness also of the nonlinear \textup{fJMGT--W} equation for each $\alpha\in(0,1)$, we will not obtain a uniform bound quantifying smallness of the initial data. That is, we will not be able to show that for fixed small enough $(\psi_0, \psi_1, \psi_2)$, there exists a family of solutions to the nonlinear problem. Hence, concerning limits as $\alpha\to1-$, we will restrict ourselves to the linear \textup{fMGT} equation:
\[
\tau^\alpha \Dal \psi_{tt}+(1+\sigma(x,t))\psi_{tt}-c^2 \Delta \psi -\tau^\alpha c^2 \Dal \Delta \psi- \delta \D\psi_{t}=f   .
\]
We next prove the well-posedness of the linear time-fractional problem. {Note that under the same regularity conditions on the initial and boundary data, the fMGT equation allows us to prove slightly better regularity of the solution as compared to fMGT I; cf. Proposition~\ref{Prop:fMGT_I}.}
\begin{proposition}[Well-posedness of the fMGT equation] \label{Prop:fMGT}
	Let  $\alpha \in [\alpha_0,1)$ for some $\alpha_0 > 1/2$. Assume that $f \in H^{\alpha-1/2}(0,T;L^2(\Om))$, $\sigma \in X^\sigma_\fmgti$, and \[(\psi_0, \psi_1, \psi_2) \in (\Honetwo, \Honetwo, H_0^1(\Om)).\] Then there exists $\varrho=\varrho(\alpha)>0$, such that if $\|\sigma\|_{X^\sigma_\fmgt} \leq \varrho$, there is a unique $\psi \in X_\fmgt$, which satisfies the problem in the $L^2(0,T; L^2(\Om))$ sense with $(\psi, \psi_t, \psi_{tt})\vert_{t=0}=(\psi_0, \psi_1, \psi_2)$.  Furthermore, this solution fulfills the following estimate:
	\begin{equation} \label{est_fMGT}
	\begin{aligned}
	& \begin{multlined}[t]\|\Dt^{2+\alpha}\psi\|_{L^2(L^2)}^2
	+ \nLtwo{\nabla \Dt^{1+\alpha}\psi(t)}^2
	+\nLtwo{\D \Dt^{\alpha}\psi(t)}^2\\
	+ C(\alpha)\Bigl(\|\nabla\psi_{tt}\|_{{H^{-(1-\alpha)/2}}(L^2)}^2
	+ \|\Dt^{1+\alpha/2}\D \psi\|_{L^2(L^2)}^2\Bigr) \end{multlined}\\
	\lesssim&\, \|f\|_{H^{\alpha-1/2}(L^2)}^2+\nLtwo{\D \psi_0}^2+\nLtwo{\D \psi_1}^2+\nLtwo{\nabla \psi_2}^2.
	\end{aligned}
	\end{equation}
	where $C(\alpha)\to0$ as $\alpha\to1^-$. 
\end{proposition}
\begin{proof}
{The proof follows similarly to the proof of Proposition~\ref{Prop:fMGT_I} with the main changes contained in the energy analysis, on which we focus here.} Note that now the semi-discrete problem can be equivalently rewritten as a system of Volterra integral equations:
	\[
	\begin{aligned}
	&\begin{multlined}[t]\tau^\alpha M \bmu 
	+ M_\sigma(t)\Bigl(p^{\alpha-1}*\bmu + p^0 \bxi_2\Bigr)
	+c^2 K \Bigl(p^{\alpha+1}*\bmu + p^2 \boldsymbol{\xi}_2+ p^1 \boldsymbol{\xi}_1+ p^0 \boldsymbol{\xi}_0\Bigr)\\
	+\tau^\alpha c^2 \Bigl(p^{1}*\bmu +p^{2-\alpha} \boldsymbol{\xi}_2+  p^{1-\alpha} \boldsymbol{\xi}_1\Bigr)
	+\delta K (	p^{\alpha}*\bmu + p^1 \boldsymbol{\xi}_2+ p^0 \boldsymbol{\xi}_1) =f\end{multlined}
	\end{aligned}
	\]
	in place of \eqref{Volterra_system_I}; the existence of an approximate solutions follows by the same arguments. We present the energy analysis of the semi-discrete problem here, but omit the superscript $n$ below for simplicity.  Multiplying the semi-discrete equation with $-\Delta \Dt^{1+\alpha} \psi$, yields the energy identity \eqref{enid_rem}, where now
	\begin{equation} \label{damping_fMGT}
	\begin{aligned} 
	\dmp =&\, \intt \prodLtwo{\Dt \D \psi}{\D \Dt^{1+\alpha} \psi}\ds\\
	=&\, d_0+	\int_0^t\prodLtwo{\Ialpha \Dt^{1+\alpha}\Delta\psi(s)}{\Dt^{1+\alpha}\Delta\psi(s)}\ds,
	\end{aligned}
	\end{equation}
	with $d_0=\prodLtwo{\D\psi_t(0)}{\Dt^{\alpha} \D \psi(t)}-p^{1-\alpha}(t)\|\D\psi_t(0)\|_{L^2}^2$ instead of \eqref{damping_fMGT_I}. Here, we have used the identities
	\[\Dt\psi={\textup{I}}^1\psi_{tt}+\psi_t(0)=\Ialpha\Dt^{1+\alpha}\psi+\psi_t(0)\]
	and \[
	\begin{aligned}
	\intt\prodLtwo{\D\psi_t(0)}{\Dt^{1+\alpha} \D \psi(s)}\ds
	=&\,\prodLtwo{\D\psi_t(0)}{\intt ((\Dt \Dt^{\alpha} \D \psi)(s)-p^{-\alpha}(s)\D\psi_t(0))\ds}\\
	=&\, \prodLtwo{\D\psi_t(0)}{\Dt^{\alpha} \D \psi(t)}-p^{1-\alpha}(t)\|\D\psi_t(0)\|_{L^2}^2.
	\end{aligned}\]
	The damping term can now be estimated from below as follows
	\[
	\frac{\dmp+d_0}{\cos(\pi\gamma/2)} \geq 
	\|\Dt^{1+\alpha}\D \psi\|_{H_t^{-\alpha/2}(L^2)}^2\sim \|\Dt^{1+\alpha/2}\D \psi\|_{L_t^2(L^2)}^2. 
	\]
	Furthermore, we have
	\[
	\begin{aligned}
	\left|\int_0^t\prodLtwo{\tilde{f}(s)}{\Dt^{1+\alpha}\D \psi(s)}\ds\right|
	\leq 
	\frac{1}{2\epsilon}\|\tilde{f}\|_{H^{\alpha/2}(L^2)}^2+\frac{\epsilon}{2}\|\Dt^{1+\alpha/2}\D \psi\|_{L^2(L^2)}^2
	\end{aligned}
	\]
	with
	\begin{equation}
	\tilde{f}=f+c^2\D\pn-\sigma\pttn +\tau^{\alpha}\left(p^{-\alpha}\pttn(0)-c^2p^{1-\alpha}\D\ptn(0)\right)
	\end{equation}
	as before. 	By the Kato--Ponce inequality \eqref{prodruleest} with $\rho=\alpha/2$, we then have
	\begin{equation}\label{estsigmapsitt}
	\begin{aligned}
	&\|\sigma\psi_{tt}\|_{H_t^{\alpha/2}(L^2)}
	\\
	\lesssim&\,
	\|\sigma\|_{W_t^{\alpha/2,p_1}(L^\infty)}\|\psi_{tt}\|_{L_t^{q_1}(L^2)}
	+\|\sigma\|_{L_t^{p_2}(L^\infty)}\|\psi_{tt}\|_{W_t^{\alpha/2,q_2}(L^2)}
	\end{aligned}
	\end{equation}
	for $\frac{1}{p_1}+\frac{1}{q_1}=\frac{1}{p_2}+\frac{1}{q_2}=\frac12$. 	Similarly to before, to further estimate the norms of $\psi_{tt}$ in \eqref{estsigmapsitt} we will use the leading time derivative term $\Dt^{2+\alpha}\psi$ and its representation via the fractional ODE
	\[
	\tau^\alpha\Dt^\alpha \ptt + \ptt = -r, \qquad r=\sigma \ptt - c^2\Delta\psi -\tau^\alpha c^2\Dt^\alpha\Delta \psi - \delta \Dt \Delta \psi -f.
	\]
	Therefore,
	\[
	\begin{aligned}
	&\|\psi_{tt}\|_{L^{q_1}(0,t;L^2)}+\|\psi_{tt}\|_{W^{\alpha/2, q_2}(0,t;L^2)}\lesssim \|\Dt^{2+\alpha}\psi\|_{L^2(L^2)}\\
	\lesssim&\, \|\sigma \ptt - c^2\Delta\psi -\tau^\alpha c^2\Dt^\alpha\Delta \psi - \delta \Dt \Delta \psi -f
	\|_{L^2(L^2)}+\|\psi_{tt}(0)\|_{L^2},
	\end{aligned}
	\]
	where again we take care of the highest order term by using the damping term $\dmp$. Thus, we require
	\begin{equation}\label{cond_qomega}
	2-\frac{1}{q_1}\leq 2+\alpha-\frac12, \quad	2+\alpha/2-\frac{1}{q_2}\leq 2+\alpha-\frac12. %, \quad
	%	1\leq 1+\alpha/2 .
	\end{equation}
	Additionally, we aim at choosing the available parameters $p_i$ (yielding $q_i=2p_i/(p_i-2)$), such that (having in mind that $\sigma = 2k \psi_t$ in the fixed point argument later on)
	\[
	\|\D\psi_t(s)\|_{W^{\alpha/2, p_1}_t(L^2)}\lesssim\dmp, \quad 
	\|\D\psi_t(s)\|_{L^{p_2}_t(L^2)}\lesssim\dmp\,,
	\]
	which leads to 
	\begin{equation}\label{cond_pomega}
	1+\frac{\alpha}{2}-\frac{1}{p_1}\leq 1+\frac{\alpha}{2}-\frac12 \mbox{ and }
	1-\frac{1}{p_2}\leq 1+\frac{\alpha}{2}-\frac12.
	\end{equation}
	It is readily checked that all conditions in \eqref{cond_qomega}, \eqref{cond_pomega} can be satisfied with the choice 
	\begin{equation}\label{p12omega12}
	(p_1, q_1)=(2, \infty)\,, \qquad (p_2, q_2)=\left(\frac{2}{1-\alpha}, \frac{2}{\alpha}\right).
	\end{equation}
	Thus analogously to the proof of Proposition~\ref{Prop:fMGT_I}, we arrive at \eqref{est_fMGT}.
\end{proof}
%%%%%%%%%%%%%%%%%%%%%%%%%%%%%%%%%%%%%%%%%
\begin{theorem}[Local well-posedness of the fJMGT--W equation] \label{Thm:fJMGT_W}
	Let $\alpha \in [\alpha_0,1)$ for some $\alpha_0 > 1/2$ and $T>0$. Further, assume that $f \in H^{\alpha-1/2}(0,{T}; L^2(\Om))$. Then there exists  
	$\varrho=\varrho(\alpha,T)>0$
	such that if
	\begin{equation}
	\|f\|^2_{H^{\alpha-1/2}(L^2)}+\|\psi_0\|_{H^2}^2+\|\psi_1\|_{H^2}^2+\|\psi_2\|_{H^1}^2 \leq \varrho ^2,
	\end{equation}
	then the initial boundary-value problem
	\begin{equation}
	\left \{
	\begin{aligned}
	\tau^\alpha \Dt^{2+\alpha}\psi + &(1+2k\pt) \psi_{tt} - c^2\Delta\psi -\tau^\alpha c^2\Dt^\alpha\Delta \psi - \delta \Delta \psi_t = f\hspace*{-2mm}&&\text{in }\Omega\times(0,T), \\[1mm]
	&\psi=\,0&&\text{on }\partial\Omega\times(0,T),\\[1mm]
	&(\psi, \psi_t, \psi_{tt})=\,(\psi_0, \psi_1, \psi_2)&&\mbox{in }\Omega\times \{0\},
	\end{aligned} \right.
	\end{equation}
	has a unique solution $\psi \in X_{\textup{fMGT}}$, which satisfies	
	\begin{equation} \label{energy_est2_fJMGT_W}
	\begin{aligned}
	\begin{multlined}[t] \|\psi\|^2_{X_{\textup{fMGT}}}
	\end{multlined}
	\lesssim\,\begin{multlined}[t] 	\|f\|^2_{H^{\alpha-1/2}(L^2)}+\|\psi_0\|_{H^2}^2+\|\psi_1\|_{H^2}^2+\|\psi_2\|_{H^1}^2. \end{multlined}
	\end{aligned}
	\end{equation}
\end{theorem}
\begin{proof}
	The proof follows in an analogous manner to the proof of Theorem~\ref{Thm:fJMGT_W_I}, combined with the results of Proposition~\ref{Prop:fMGT}. We therefore omit the details here.
\end{proof}
\begin{remark}[On the analysis of the fMGT II equation with $\boldsymbol{\sigma \neq 0}$]\label{Remark:fMGT_II}
	We note that the \textup{fMGT II} equation {
	\[
	\tau^\alpha \Dal \psi_{tt}+(1+\sigma(x,t))\psi_{tt}-c^2 \Delta \psi -(\tau^\alpha c^2 +\delta)\Dal \Delta \psi=f
	\]}
	 does not seem to be tractable this way with $\sigma\not=0$. In particular, we would have 
	\[
	\dmp =\nLtwo{\D\Dt^\alpha\psi}^2 \Big \vert_0^t 
	\]
	in place of \eqref{damping_fMGT_I} and \eqref{damping_fMGT}.  Thus, the damping term $\delta\dmp$ is obviously too weak. \\
	\indent Note that the multiplier $-\D\psi_{tt}$ that we have successfully used for the \textup{fMGT III} equation in Section~\ref{Sec:Analysis_fJMGT_W_III} does to work out either since then the $\delta$ term cannot be proven to be nonnegative due to the fact that the difference $2-\alpha$ of the differentiation orders is larger than one. We provide an analysis of the \textup{fMGT II} equation with $\sigma=0$ in Section~\ref{Sec:LinearAnalysis} by rewriting it as a second-order wave equation with memory.
\end{remark}

%%%%%%%%%%%%%%%%%%%%%%%%%%%%%%%%%%%%%%%%%%%%%%%%%%%%%%%
\subsection{Limiting behavior of the fMGT--W and fJMGT--W I equations}
The difference $\overline{\psi}=\psi^{\alpha}-\psi$ solves 
\begin{equation} \label{fJMGT_W_III_diff_lin}
\begin{aligned}
&\tau^\alpha \Dt^{2+\alpha}\opsi + \sigma\opsi_{tt} - c^2\Delta\opsi -\tau^\alpha c^2\Delta \Dt^\alpha\opsi - \delta \Dt^\beta \Delta \opsi_t \\
&=\,(\tau\psi_{ttt}-\tau^\alpha\Dt^{2+\alpha}\psi_{tt}) - c^2\D(\tau\psi_t-\tau^\alpha\Dt^\alpha\psi) - \delta \D(\psi_t-\Dt^\beta\psi_t)=:\tilde{f}
\end{aligned}
\end{equation}
in the linear case and 
\begin{equation} \label{fJMGT_W_III_diff}
\begin{aligned}
&\tau^\alpha \Dt^{2+\alpha}\opsi + (1+2k\psi_t^{\alpha})\opsi_{tt} - c^2\Delta\opsi -\tau^\alpha c^2\Delta \Dt^\alpha\opsi - \delta \Dt^\beta \Delta \opsi+2k\psi_{tt}\opsi_t =\tilde{f}
\end{aligned}
\end{equation}
in the nonlinear case with vanishing initial data and $\beta=\beta(\alpha)$ as in \eqref{def_beta} (which implies that the $\beta$ difference term on the right-hand side just vanishes in case of the \textup{fJMGT--W} equation).

Multiplication with $\Dt^{1+\alpha}\opsi$ with the abbreviations 
$\sigma^\alpha=\sigma$, $\mu=0$ in the linear case  
and $\sigma^\alpha=2k\psi_t^{\alpha}$, $\mu=2k\psi_{tt}\opsi_t$ in the nonlinear case yields, analogously to the proof of uniqueness in Proposition~\ref{Prop:fMGT_I},
\begin{equation} \label{limit_diff_I}
\begin{aligned}
\textup{lhs}:=& \, \frac{\tau^\alpha}{2}\nLtwo{\Dt^{1+\alpha}\opsi(t)}^2
+ \frac{\tau^\alpha c^2}{2}\nLtwo{\nabla\Dt^{\alpha}\opsi(t)}^2\\
&\quad+ \cos(\pi(1-\alpha)/2)\|\opsi_{tt}\|_{H_t^{-(1-\alpha)}(L^2)}^2
+ \delta  \cos(\pi\gamma/2) \|\Dt^m \nabla \opsi\|_{L_t^2(L^2)}^2
\\
=& \int_0^t\prodLtwo{\tilde{f}+c^2\D\opsi-\sigma^\alpha\opsi_{tt}+\mu}{\Dt^{1+\alpha}\opsi} \ds\\
\leq&\, \begin{multlined}[t]\frac{\epsilon}{2} \|\Dt^{1+\alpha}\opsi\|_{L_t^\infty(L^2)}^2 
+\frac{1}{2\epsilon} \|\tilde{f}+2k\psi_{tt}\opsi_{t}\|_{L_t^1(L^2)}^2
+\frac{\epsilon}{2} \|\Dt^m\nabla\opsi\|_{L_t^2(L^2)}^2 \\
+\frac{1}{2\epsilon}\left(\|c^2\nabla\opsi\|_{H_t^\rho(L^2)}+\|\sigma^\alpha\opsi_{tt}\|_{H_t^\rho(H^{-1})}\right)^2,\end{multlined}
\end{aligned}
\end{equation}
where 
\[
\begin{cases}
\gamma=2\alpha-1\,, \quad m=3/2\,, \quad \rho=\alpha-1/2&\mbox{ for fMGT I,}\\
\gamma=\alpha\,, \quad m=1+\alpha/2\,, \quad \rho=\alpha/2&\mbox{ for fMGT.}
\end{cases}
\]
We know that
\[
\begin{aligned}
\|2k\psi_{tt}\opsi_{t}\|_{L^1(L^2)}
&{\leq 2|k|C_{H^1,L^6}^\Om C_{H^\alpha,L^3}^\Om \|\nabla\psi_{tt}\|_{L_t^2(L^2)} \|\opsi_t\|_{L_t^2(H^\alpha)}}\\
&{\lesssim \|\nabla\psi_{tt}\|_{L^2(L^2)}
	\|\Dt^{1+\alpha}\opsi\|_{L_t^2(L^2)}^{1-\alpha} \|\Dt^\alpha\nabla\opsi\|_{L_t^2(L^2)}^\alpha,}\\
\|c^2\D\opsi\|_{H^\rho(L^2)}&\lesssim \|\nabla\Dt^{\alpha}\opsi\|_{L^2(L^2)}.
\end{aligned}
\]
where we have used interpolation; cf.~\cite[Chapter 7]{Adams}. Let $X^\sigma$ be either $X^\sigma_\fmgti$ or $X^\sigma_\fmgt$, depending on the equation. By proceeding similarly to \eqref{estsigmapsitt_I}--\eqref{Dt2plusalphapn}, we find that
\[
\begin{aligned}
\|\sigma^\alpha\opsi_{tt}\|_{H^\rho(H^{-1})}\lesssim&\, \begin{multlined}[t] \|\sigma^\alpha\|_{X^\sigma} \left(\|\sigma^\alpha\opsi_{tt}\|_{H_t^{\rho}(H^{-1})} 
+ \textup{rhs}\right), \end{multlined}
\end{aligned}
\]
where 
\[
\begin{aligned}
\textup{rhs}:=&
 c^2\|\nabla\opsi\|_{L^2_t(L^2)} +\tau^\alpha c^2\|\Dt^\alpha\nabla \opsi\|_{L^2_t(L^2)}
	+\|\tilde{f}\|_{L^2(H^{-1})} \\
&\qquad+ \delta
\begin{cases}
\kersing_{3/2-\alpha}*\|\Dt^m\nabla\opsi\|_{L^2_t(L^2)}^2&\mbox{ for fMGT I,}\\
\|\Dt^m\nabla\opsi\|_{L^2_t(L^2)}^2&\mbox{ for fMGT.}
\end{cases}
\end{aligned}
\]
We can therefore tackle all terms by generalized Gronwall in the \textup{fMGT--I} case, 
whereas for \textup{fMGT} we need to absorb the $\delta$ term by the lhs $\delta$ term in \eqref{limit_diff_I} and therefore need to impose smallness of $\|\sigma^\alpha\|_{X^\sigma}$ with an $\alpha$ dependent bound. 

It remains to estimate the contribution arising from \[\tilde{f}= \tau\psi_{ttt}-\tau^\alpha\Dt^{2+\alpha}\psi - c^2\D(\tau\psi_t-\tau^\alpha\Dt^\alpha\psi) - \delta \D(\psi_t-\Dt^\beta\psi),\] which we do for each of the difference terms separately, 
\[
\begin{aligned}
\|\tau\psi_{ttt}-\tau^\alpha\Dt^{2+\alpha}\psi_{tt}\|_{L^1(L^2)}
\leq&\,|\tau-\tau^\alpha|\, \|\psi_{ttt}\|_{L^1(L^2)}+\tau^\alpha \|(\Dt-\Dt^\alpha)\psi_{tt}\|_{L^1(L^2)},
\\
\|\D(\tau\psi_t-\tau^\alpha\Dt^\alpha\psi)\|_{L^1(L^2)}
\leq&\,|\tau-\tau^\alpha|\,\|\D\psi_t\|_{L^1(L^2)}+\tau^\alpha \|(\Dt-\Dt^\alpha)\D\psi\|_{L^1(L^2)},
\end{aligned}
\]
and, in case of \textup{fJMGT--W I} with $\beta=2-\alpha$,
\[
\begin{aligned}
\|\D(\psi_t-\Dt^{2-\alpha}\psi)\|_{L^1(L^2)}
=\|(\mbox{id}-\Dt^{1-\alpha})\D\psi_t\|_{L^1(L^2)}.
\end{aligned}
\]
Thus, to be able to apply the limits \eqref{limitalphaL2_1} and \eqref{limitalphaL2}, we need 
\[\psi_{tt},\, \D\psi
{\in W^{2,1}(0,T;L^2(\Om)) \,, \ }
 \D\psi_t \in W^{1,1}(0,T;L^2(\Om)),\] and, in case of \textup{fJMGT--W I} additionally $\psi_t(0)=0$.\\
\indent Note that the required smoothness of $\psi$ follows, e.g., from Theorem~\ref{Thm:fJMGT_K_III} below under restrictive regularity conditions on the initial data. We expect that these assumptions might be relaxed in view of the fact that the $W^{1, \infty}(0,T; \Honetwo) \cap W^{2, \infty}(0,T; H_0^1(\Om))\cap W^{3, \infty}(0,T; L^2(\Om))$ regularity that we obtain from Theorem~\ref{Thm:fJMGT_W_III} is already very close to what is needed here. Altogether, with \[
\begin{aligned}
X^{\textup{low}} = \left\{ \psi\in H^{1/2}(0,T; H_0^1(\Om)): 
\Dt^{3/2}\psi \in L^\infty(0,T; L^2(\Om)) \right\}
\end{aligned}
\]
{and the corresponding norm denoted by $\|\cdot\|_{X^{\textup{low}}}$}, we have the following results.
\begin{proposition}[Limit of the fMGT--I equation] 
	Let $f \in H^{1/2}(0,T;L^2(\Om))$, $\sigma \in X^\sigma_\fmgti$, and \[(\psi_0, \psi_1, \psi_2) \in (\Honetwo, \{0\}, H_0^1(\Om)).\]
	Further, let $\varrho>0$ be as in Proposition ~\ref{Prop:fMGT_I} and 
	$\|\sigma\|_{X^\sigma_\fmgti} \leq \varrho$; let $\{\psi^\alpha\}_{\alpha \in (0,1)}$ be the family of solutions to the \textup{fMGT--I} equation, let $\psi$ solve the corresponding equation with $\alpha=1$ and assume that 
	\begin{equation}\label{assumedregularity}
	\begin{aligned}
&\nabla\psi_{tt}\in L^2(0,T;L^2(\Om)),\ \psi_{tt}, \D\psi 
{\in W^{2,1}(0,T;L^2(\Om)), }\\
&\D\psi_t \in W^{1,1}(0,T;L^2(\Om)).
\end{aligned}
\end{equation} 
Then $\psi^{\alpha}$ converges to $\psi$ in the $\|\cdot\|_{X^{\textup{low}}}$ norm as $\alpha \rightarrow 1^{-}$.
\end{proposition}
\begin{proposition}[Limit of the fJMGT--W I equation] 
Let $f \in H^{1/2}(0,T;L^2(\Om))$ and $(\psi_0, \psi_1, \psi_2) \in (\Honetwo, \{0\}, H_0^1(\Om))$.
	Furthermore, let $\varrho>0$ be as in Theorem~\ref{Thm:fJMGT_W_I} and
	\[\|f\|^2_{H^{1/2}(H^1)}+\|\psi_0\|_{H^2}^2
	+\|\psi_2\|_{H^1}^2 \leq \varrho^2;\] let $\{\psi^\alpha\}_{\alpha \in (0,1)}$ be the family of solutions to the \textup{fJMGT--W I} equation, let $\psi$ solve the corresponding equation with $\alpha=1$, and assume that \eqref{assumedregularity} holds. Then $\psi^{\alpha}$ converges to $\psi$ in the $\|\cdot\|_{X^{\textup{low}}}$ norm as $\alpha \rightarrow 1^{-}$.
\end{proposition}
\begin{proposition}[Limit of the fMGT equation] 
	Assume that $f \in H^{1/2}(0,T;L^2(\Om))$, $\sigma=0$, and $(\psi_0, \psi_1, \psi_2) \in (\Honetwo, \Honetwo, H_0^1(\Om))$. Let $\{\psi^\alpha\}_{\alpha \in (0,1)}$ be the family of solutions to the \textup{fMGT} equation, let $\psi$ solve the corresponding equation with $\alpha=1$ and assume that \eqref{assumedregularity} holds. Then $\psi^{\alpha}$ converges to $\psi$ in the $\|\cdot\|_{X^{\textup{low}}}$ norm as $\alpha \rightarrow 1^{-}$.
\end{proposition}

%%%%%%%%%%%%%%%%%%%%%%%%%%%%%%%%%%%%%%%%%%%%%%%%%%%%%%%
\section{Equations with the quadratic gradient nonlinearity} \label{Sec:Analysis_fJMGT_K_III}
Unlike its Westervelt version, the Kuznetsov versions of these time-fractional equations contain a quadratic gradient nonlinearity, and so their analysis requires the use of higher-order energy estimates. We thus limit our presentation to the analysis of the fJMGT--K III equation, which has the integer-order leading term. \\
\indent The study of well-posedness for this equation follows by combining the ideas from the previous section concerning the fractional term with the ideas used in the analysis of its integer-order counterpart considered in~\cite{kaltenbacher2021inviscid, KaltenbacherNikolic}. To formulate the well-posedness result we introduce the solution space
\begin{equation}
\begin{aligned}
\spaceKhigh=&\, \begin{multlined}[t] \left\{\vphantom{H^{-\alpha/2}(0,T; \Honethree)}\psi \in L^{\infty}(0,T; \Honethree): \psi_t \in L^{\infty}(0,T; \Honethree),\right. \\ \left. \hspace*{-0.3cm} \psi_{tt} \in L^\infty(0,T; \Honetwo) \cap H^{-\alpha/2}(0,T; \Honethree), \ \psi_{ttt} \in L^2(0,T; H_0^1(\Omega)) \right\}  
\end{multlined}
\end{aligned}
\end{equation}
for $ \alpha \in (0,1)$, and
\begin{equation}
\begin{aligned}
\spaceKhigh=  H^{3}(0,T; H_0^1(\Om)) \cap W^{2, \infty}(0,T; \Honetwo) \cap W^{1, \infty}(0,T; \Honethree),
\end{aligned}
\end{equation}
for $\alpha=1$, {equipped with the norm $\|\cdot\|_{\spaceKhigh}$.}
\begin{theorem}[Local well-posedness of the fJMGT--K III equation] 
	\label{Thm:fJMGT_K_III} Let $\alpha \in (0,1]$, $\tilde{T}>0$, and $\varrho>0$. Further, assume that $f \in H^1(0,\tilde{T}; H_0^1(\Om))$ and that
	\begin{equation}
	\|f\|^2_{H^1(H^1)} +\|\psi_0\|_{H^3}^2+ \|\psi_1\|_{H^3}^2+\|\psi_2\|_{H^2}^2   \leq \varrho^2. 
	\end{equation}
	Then there exists $T=T(\varrho) \leq \tilde{T}$, such that the initial boundary-value problem
	\begin{equation}\label{ibvp_fJMGT_K_III}
	\left \{
	\begin{aligned}
	\tau \psi_{ttt} + &(1+2\tilde{k} \psi_t) \psi_{tt} - c^2\Delta\psi -\tau c^2\Delta \psi_t  && \\[1mm]
	&\hspace*{2.5cm} - \delta \Doal \Delta \psi_t+\elltil\partial_t |\nabla \psi|^2= f &&\text{ in }\Omega\times(0,T),\\[1mm]
	&\psi=\,0&&\text{ on } \partial \Omega\times(0,T),\\[1mm]
	&(\psi, \psi_t, \psi_{tt})=\,(\psi_0, \psi_1, \psi_2)&&\mbox{ in }\Omega\times \{0\},
	\end{aligned} \right.
	\end{equation}
	has a unique solution $\psi \in \spaceKhigh$, which satisfies	
	\begin{equation} \label{energy_est2_fJMGT_K_III}
	\begin{aligned}
	\begin{multlined}[t] \|\psi\|^2_{\spaceKhigh}
	\end{multlined}
	\lesssim\,\begin{multlined}[t] 	\|f\|^2_{H^1(H^1)}+\|\psi_0\|_{H^3}^2+ \|\psi_1\|_{H^3}^2+\|\psi_2\|_{H^2}^2. \end{multlined}
	\end{aligned}
	\end{equation}
\end{theorem}
\begin{proof}
	The proof follows by employing the Banach Fixed-point theorem to the mapping $\mathcal{T}: w \mapsto \psi$, where $\psi$ solves
	\begin{equation} \label{fMGT_K_III_linearized}
	\begin{aligned}
	\tau \psi_{ttt}+(1+2\tilde{k}w_t)\psi_{tt} - c^2\D \psi -\tau c^2\Delta \psi_t- \delta \Doal \Delta \psi_t   + 2 \elltil\nabla w \cdot\nabla\psi_t=f,
	\end{aligned}
	\end{equation}
	and
	\begin{equation} \label{defBR}
	\begin{aligned}
	w \in B_R:=\{w \in \spaceKhigh\, :&\, \|w\|_{\spaceKhigh}\leq R,\\ &w(0)=\psi_0, \, w_t(0)=\psi_1, \, w_{tt}(0)=\psi_2  \}.
	\end{aligned}
	\end{equation}	
	\noindent (I) The energy estimates for the linear equation can be rigorously derived by a Galerkin procedure with a sufficiently smooth basis; here we present only the derivation of the bound for the semi-discrete solution and omit the superscript $n$ below.  We denote \[p=1+2\tilde{k}w_t, \quad \phi=2 \tilde{k}w,\]
	then test the semi-discrete version of \eqref{fMGT_K_III_linearized} with $\D^2 \psi_{tt}$ and integrate in space. We can estimate the resulting non-fractional terms and those not involving $f$ in a similar manner to~\cite[Theorem 6.1]{kaltenbacher2021inviscid}. We include the derivation of these bounds below for completeness. \\
	\indent Note that $\psi_{tt}= \D \psi=\Delta \psi_{tt}=0$ and $\Dtal \Delta \psi=0$ on $\partial \Om$ for smooth Galerkin approximations based on the eigenfunctions of the Dirichlet-Laplacian. Therefore, the following identities hold:
	\begin{equation}
	\begin{aligned}
	\prodLtwo{p \psi_{tt}}{\D^2 \psi_{tt}} =&\,\prodLtwo{p \D \psi_{tt}+\psi_{tt} \D p +2\nabla p\cdot\nabla \psi_{tt}}{\D \psi_{tt}},\\
	-c^2\prodLtwo{\D \psi}{\D^2 \psi_{tt}} =&\, c^2\ddt \prodLtwo{\nabla \D \psi}{\nabla \D \psi_t}-c^2\nLtwo{\nabla \D \psi_t}^2.
	\end{aligned}
	\end{equation}
	We thus have 
	\begin{equation} \label{LinKuzn_id}
	\begin{aligned}
	&\begin{multlined}[t]\frac12 \tau\ddt \nLtwo{\D \psi_{tt}}^2 +\frac12\tau c^2\ddt \nLtwo{\nabla \D \psi_t}^2+\delta  \prodLtwo{\Dt^{2-\alpha}\nabla \Delta \psi}{\nabla \Delta \psi_{tt}}
	\end{multlined}\\
	=&\,\begin{multlined}[t] -\prodLtwo{p \D \psi_{tt}+ \psi_{tt} \D p +2\nabla p\cdot\nabla \psi_{tt}}{\D \psi_{tt}}-c^2\ddt \prodLtwo{\nabla \D \psi}{\nabla \D \psi_t}\\[1mm]+c^2\nLtwo{\nabla \D \psi_t}^2 
	- \prodLtwo{\nabla \phi \cdot \nabla \psi_t}{\Delta^2 \psi_{tt}}+\prodLtwo{\nabla f}{\nabla \Delta \psi_{tt}}.\end{multlined}
	\end{aligned}
	\end{equation}
	We next integrate in time and estimate the resulting terms. Note first that
	\[
	\begin{aligned}
	&\int_0^t \prodLtwo{p \D \psi_{tt}+\psi_{tt} \D p +\nabla p\cdot\nabla \psi_{tt}}{\D \psi_{tt}}\ds \\
	\lesssim&\,\begin{multlined}[t] \left \{(1+R) \nLtwotLtwo{\D \psi_{tt}}+\nLtwotLinf{\psi_{tt}}\nLinftLtwo{\D p}\right. \\ \left. +\nLinfLfour{\nabla p}\nLtwotLfour{\nabla \psi_{tt}} \right \} \nLtwotLtwo{\D \psi_{tt}} \end{multlined} \\
	\lesssim&\, \left \{(1+R) \nLtwotLtwo{\D \psi_{tt}}+R\nLtwotLinf{\psi_{tt}}+R\nLtwotLfour{\nabla \psi_{tt}} \right \} \nLtwotLtwo{\D \psi_{tt}},
	\end{aligned}
	\]
	where we have utilized the uniform boundedness of $p=1+2 \tilde{k}w_t$, which follows from the fact that $w \in B_R$. Furthermore,
	\[
	\begin{aligned}
	-c^2\int_0^t\ddt \prodLtwo{\nabla \D \psi}{\nabla \D \psi_t}\ds =&\,-c^2\prodLtwo{\nabla \D \psi(t)}{\nabla \D \psi_t(t)}+c^2\prodLtwo{\nabla \D \psi_0}{\nabla \D \psi_1}\\
	\leq&\,\begin{multlined}[t] \frac{1}{2 \epsilon} T c^4\nLtwotLtwo{\nabla \D \psi_{t}}^2+\frac{\epsilon}{2} \nLtwo{\nabla \D \psi_{t}(t)}^2\\
	+\frac12c^4 \nLtwoLtwo{\nabla \D \psi_0}^2+\frac12 \nLtwoLtwo{\nabla \D \psi_1}^2.
	\end{multlined}
	\end{aligned}
	\]
	By the semi-discrete PDE, we know that $\nabla \phi \cdot \nabla \psi_t=0$ on $\partial \Om$. Thus
	\[
	\begin{aligned}
	\prodLtwo{\nabla \phi \cdot \nabla \psi_t}{\Delta^2 \psi_{tt}} =&\,\prodLtwo{\D[\nabla \phi \cdot \nabla \psi_t]}{\Delta \psi_{tt}} \\
	=&\,  \prodLtwo{\nabla \D \phi\cdot\nabla \psi_t+ 2 D^2 \phi:D^2 \psi_t+ \nabla \phi\cdot\nabla \D \psi_t}{\Delta \psi_{tt}}
	\end{aligned}
	\]	
	where $D^2 v=(\partial_{x_i} \partial_{x_j} v)_{i,j}$ is the Hessian, which satisfies
	\[
	\nLfour{D^2 v}\leq \CHone (\nLtwo{D^3 v}+\nLtwo{D^2 v})\leq \CHone C_{\textup{H}}(\nLtwo{\nabla\D v}+\nLtwo{\D v}).
	\]
	This further implies that
	\[
	\begin{aligned}
	&\int_0^t \prodLtwo{\nabla \phi \cdot \nabla \psi_t}{\Delta^2 \psi_{tt}}\ds\\
	\lesssim&\, \begin{multlined}[t] \nLtwotLtwo{\D \psi_{tt}} \left\{\vphantom{\CHone^2}\nLinfLtwo{\nabla \D \phi}\nLtwotLinf{\nabla \psi_{t}}+\nLinfLinf{\nabla \phi}\nLtwotLtwo{\nabla \D \psi_t}\right.\\ \left. + (\nLinfLtwo{\nabla \D \phi}+\nLinfLtwo{\D \phi})(\nLtwotLtwo{\nabla \D \psi_{t}}+\nLtwotLtwo{\D \psi_{t}})\right\}.\end{multlined}
	\end{aligned}
	\]
	Since $w \in B_R$, the function $\phi$ is uniformly bounded, and so
	\[
	\begin{aligned}
	\int_0^t \prodLtwo{\nabla \phi \cdot \nabla \psi_t}{\Delta^2 \psi_{tt}}\ds 
	\lesssim&\, \begin{multlined}[t] R \nLtwotLtwo{\D \psi_{tt}} \left\{\vphantom{\CHone^2}\nLtwotLinf{\nabla \psi_{t}}+\nLtwotLtwo{\nabla \D \psi_t}\right.\\ \left. +\nLtwotLtwo{\nabla \D \psi_{t}}+\nLtwotLtwo{\D \psi_{t}}\right\}.\end{multlined}
	\end{aligned}
	\]
	Integration by parts with respect to time yields
	\begin{equation}
	\begin{aligned}
	\int_0^t \prodLtwo{\nabla f}{\nabla \Delta \psi_{tt}}\ds=&\, \prodLtwo{\nabla f}{\nabla \Delta \psi_{t}}\big \vert_0^t-\int_0^t \prodLtwo{\nabla f_t}{\nabla \Delta \psi_{t}}\ds \\
	\leq&\, \begin{multlined}[t]
	\frac{1}{4 \epsilon}\nLtwo{\nabla f(t)}^2+\epsilon \nLtwo{\nabla \Delta \psi_{t}(t)}^2-\prodLtwo{\nabla f(0)}{\nabla \Delta \psi_1} \\
	+\frac12 \nLtwoLtwo{\nabla f_t}^2+\frac12 \nLtwotLtwo{\nabla \D \psi_t}^2. 
	\end{multlined}
	\end{aligned}
	\end{equation}
	The $\delta$ fractional term can be handled by relying on estimate \eqref{fractional_est_w} similarly to before. Fixing $\epsilon>0$ small enough and combining the derived bounds leads to
	\begin{equation} 
	\begin{aligned}
	&\begin{multlined}[t]\frac12 \tau \nLtwo{\D \psi_{tt}(t)}^2 + (\frac12\tau c^2-2\epsilon) \nLtwo{\nabla \D \psi_t(t)}^2
	+\int_0^t \nLtwo{\Dtal\nabla \Delta \psi}^2\ds \end{multlined}\\
	\lesssim&\,\begin{multlined}[t]\nLtwo{\D \psi_2}^2 + \nLtwo{\nabla \D \psi_1}^2 +\nLtwo{\nabla \D \psi_0}^2 \\
	+ R^2(\|\psi_{tt}\|_{L^2_t(H^2)}^2+\nLtwotLfour{\nabla \psi_{tt}}^2) +\nLtwotLtwo{\D \psi_{tt}}^2\\
	+\left. R^2\nLtwotLinf{\nabla \psi_{t}}^2+(1+R^2+T)\nLtwotLtwo{\nabla \D \psi_t}^2\right. +R^2\nLtwotLtwo{\D \psi_{t}}^2+\|\nabla f\|^2_{H^1(L^2)}. \end{multlined}
	\end{aligned}
	\end{equation}
	Note that by elliptic regularity, we have
	\[
	\nLtwo{\psi_{tt}(t)}  \leq \CPF \nLtwo{\nabla \psi_{tt}(t)} \leq \nHtwo{\psi_{tt}(t)} \lesssim \nLtwo{\D \psi_{tt}(t)}.
	\]
	An application of Gronwall's inequality yields 
	\begin{equation} \label{interim_bound_linKuzn}
	\begin{aligned}
	&\sup_{t \in (0,\tilde{T})}\nLtwo{\D \psi_{tt}(t)}^2+\sup_{t \in (0,\tilde{T})}\nLtwo{\nabla \D \psi_{t}(t)}^2 +\int_0^t \nLtwo{\Dtal\nabla \Delta \psi}^2\ds\\
	\leq&\, C(T, R) (\|\nabla f\|^2_{H^1(H^1)}+\nLtwo{\D \psi_2}^2 + \nLtwo{\nabla \D \psi_1}^2  + \nLtwo{\nabla \D \psi_0}^2).
	\end{aligned}
	\end{equation}
	The uniqueness follows by using $\psi_{tt}$ as the test function in the homogeneous problem. \\

	\noindent (II) It is straightforward to check now that $\TK$ is a well-defined self-mapping. We thus focus on proving strict contractivity. Take  $w^{(1)}, w^{(2)} \in B_R$ and set  $\psi^{(1)}=\TK \phi^{(1)}$ and $\psi^{(2)}=\TK \phi^{(2)}$. Then the difference $\opsi=\psi^{(1)}-\psi^{(2)}$ solves
	\begin{equation} \label{Kuzn_contractivity_eq}
	\begin{aligned}
	\begin{multlined}[t]
	\tau \opsi_{ttt}+(1+2 \tilde{k}w_t^{(1)} )\opsi_{tt}-c^2 \Delta \opsi-\tau c^2 \Delta \opsi_t -\delta \Dtal \D \opsi 
	+ 2 \elltil\nabla w^{(2)} \cdot \nabla \opsi_t=\tilde{f}. 
	\end{multlined}
	\end{aligned}
	\end{equation}
	with the right-hand side
	\[
	\tilde{f} = -2\tilde{k} \overline{w}_t \psi^{(2)}_{tt}- 2 \elltil\nabla \overline{w} \cdot \nabla \psi^{(1)}_t
	\]
	and satisfies zero initial conditions.  Testing with $\opsi_{tt}$ yields, after standard manipulations,
	\begin{equation}  \label{Contractivity_identity}
	\begin{aligned}
	\nLtwo{\opsi_{tt}(t)}^2+ \nLtwo{\nabla \opsi_t(t)}^2+\nLtwo{\nabla \opsi(t)}^2 \leq C(\tilde{T}, R) \nLtwoLtwo{\tilde{f}}^2;
	\end{aligned}
	\end{equation}
	see also estimate (5.5) in~\cite{kaltenbacher2021inviscid}. It remains to bound the source term. By H\"older's inequality, we have
	\begin{equation}
	\begin{aligned}
	\nLtwoLtwo{\tilde{f}}
	\lesssim&\,\begin{multlined}[t] \nLinfLfour{\psi_{tt}^{(2)}}\nLtwoLfour{\overline{w}_t}  + \nLinfLinf{\nabla \psi_t^{(1)}}\nLtwoLtwo{\nabla \overline{w} }.
	\end{multlined}
	\end{aligned}
	\end{equation}
	The first term on the right can be further bounded as follows:
	\begin{equation}
	\begin{aligned}
	\nLinfLfour{\psi_{tt}^{(2)}}\nLtwoLfour{\overline{w}_t}
	\leq\,  \CHone^2 \nLinfLtwo{\nabla \psi_{tt}^{(2)}}T\nLinfLtwo{\nabla \overline{w}_t}.
	\end{aligned}
	\end{equation}
	By noting that $\nLtwoLtwo{\nabla \overline{w} } \leq \tilde{T} \nLtwoLtwo{\nabla \overline{w}_t}$, it further follows that
	\begin{equation}
	\begin{aligned}
	\nLtwoLtwo{\tilde{f}}^2
	\lesssim&\,\begin{multlined}[t]  \nLinfLtwo{\nabla \psi_{tt}^{(2)}}^2 \tilde{T}^2\nLinfLtwo{\nabla \overline{w}_t}^2+\nLinfLinf{\nabla \psi_t^{(1)}}^2\tilde{T}^2\nLtwoLtwo{\nabla \overline{w_t} }^2.
	\end{multlined}
	\end{aligned}
	\end{equation}
	Employing this bound in \eqref{Contractivity_identity} and relying on Gronwall's inequality leads to
	\begin{equation} 
	\begin{aligned}
	&\sup_{t \in (0,\tilde{T})}\nLtwo{\opsi_{tt}(t)}+\sup_{t \in (0,\tilde{T})}\nLtwo{\nabla \opsi_t(t)} \\
	\leq&\,\begin{multlined}[t] C(T, R)\tilde{T}(\nLinfLtwo{\nabla \psi_{tt}^{(2)}}+\nLinfLinf{\nabla \psi_t^{(1)}})\sup_{t \in (0,\tilde{T})}\nLtwo{\nabla \overline{w}_t(t)}. \end{multlined}
	\end{aligned}
	\end{equation}
	By the energy estimate for the linear problem, we know that
	\[
	\nLinfLtwo{\nabla \psi_{tt}^{(2)}}+\nLinfLinf{\nabla \psi_t^{(1)}} \leq \sqrt{\tilde{C}(\tilde{T}, R)}\, r
	\]
	for some $\tilde{C}(\tilde{T}, R)>0$, independent of $\alpha$. Thus we can achieve strict contractivity of $\TK$ with respect to the $\|\cdot\|_{\spacelow}$ norm by reducing the final time $\tilde{T}$. The claim then follows by the Banach Fixed-point theorem.
\end{proof}
We next discuss the limit of this equation with respect to the order of differentiation. Given $\alpha \in (0,1)$, under the assumptions of Theorem~\ref{Thm:fJMGT_W_III}, let $\psi^\alpha$ be the solution of the fJMGT--K III equation and let $\psi$ solve the corresponding JMGT--Kuznetsov equation obtained by setting $\alpha=1$ in \eqref{ibvp_fJMGT_K_III}. Then the difference $\overline{\psi}=\psi^{\alpha}-\psi$ solves 
\begin{equation} \label{fJMGT_W_III_diff}
\begin{aligned}
&\begin{multlined}[t]\tau \opsi_{ttt} + (1+2\tilde{k}\psi_t^{\alpha})\opsi_{tt} - c^2\Delta\opsi -\tau c^2\Delta \opsi_t - \delta \Dt^{1-\alpha} \Delta \opsi_t\\\hspace*{2cm}+2\tilde{k}\psi_{tt}\opsi_t+2\elltil \nabla \opsi \cdot \nabla \psi_t^{\alpha}+2\elltil \nabla \psi \cdot \nabla \opsi_t 
=\, \delta (\Dt^{1-\alpha}\Delta \psi_t-\Delta \psi_t).\end{multlined}
\end{aligned}
\end{equation}
Similarly to the proof of the previous theorem, testing with $\overline{\psi}_{tt}$ and using the uniform boundedness of $\|\psi^\alpha\|_{\spaceKhigh}$ for $\alpha \in (0,1]$ leads to the following bound:
\begin{equation} 
\begin{aligned}
\|\opsi_{tt}(t)\|^2_{L^2}+\|\nabla\opsi_{t}(t)\|^2_{L^2}\lesssim\,\nLtwoLtwo{\Dt^{1-\alpha}\nabla \psi_t-\nabla \psi_t}^2.
\end{aligned}
\end{equation}
On account of Lemma~\ref{Lemma:Limit}, we then have the following result.            
\begin{proposition}[Limit of the fJMGT--K III equation] Let the assumptions of Theorem~\ref{Thm:fJMGT_K_III} hold with $\psi_1=0$. Let $\{\psi^\alpha\}_{\alpha \in (0,1)}$ be the family of solutions to the \textup{fJMGT--K III} equation and let $\psi$ solve the corresponding \textup{JMGT--Kuznetsov} equation. Then $\psi^{\alpha}$ converges to $\psi$ in $W^{1, \infty}(0,T; H_0^1(\Omega)) \cap W^{2, \infty}(0,T; L^2(\Om))$ as $\alpha \rightarrow 1^{-}$.
\end{proposition}

%%%%%%%%%%%%%%%%%%%%%%%%%%%%%%%%%%%%%%%%%%%%%%%%%%%%%%%
\section{Reformulation of linear models as wave equations with memory} \label{Sec:WaveEq_Memory}
By neglecting the nonlinear terms in the equations given in Table~\ref{table:fJMGT}, we arrive at their linear counterparts, which are listed separately in Table~\ref{table:fMGT}. A possible idea to facilitate the linear analysis, which we wish to explore here, is to re-formulate these equations in terms of 
$$z=\tau^\gamma \Dt^\gamma\psi+\psi,$$ 
where $\gamma=1$ in case of fMGT III and $\gamma=\alpha$ otherwise. These linear models can be rewritten as second-order wave equations with memory
\[
z_{tt}-c^2\Delta z -\delta \Dt^\beta\D \psi = f
\]
where $\beta=2-\alpha$ for the fMGT III equation; otherwise it is given by \eqref{def_beta}. By using the Mittag-Leffler functions $E_{\gamma,\gamma}$ and $E_{\gamma,1}$ we can express $\psi$ as
\begin{equation}\label{psifromz}
\begin{aligned}
\psi(t)=&\, {E_{\gamma,1}(-(\tfrac{t}{\tau})^{\gamma})\psi_0+}
\tau^{-\gamma}\int_0^t(t-s)^{\gamma-1}E_{\gamma,\gamma}(-(\tfrac{t-s}{\tau})^{\gamma})z(s)\ds\\
=&\,{E_{\gamma,1}(-(\tfrac{t}{\tau})^{\gamma})\psi_0+}\kerML_\gamma*z.
\end{aligned}
\end{equation}
with $\kerML_\gamma$ as in \eqref{kalpha} below.
That is, 
$$
\Dt^\beta\psi = \Dt^{\beta-\gamma}\Dt^\gamma\psi = \tau^{-\gamma}\Dt^{\beta-\gamma} (z-\psi)= 
\tau^{-\gamma}\Dt^{\beta-\gamma} \left( z-\kerML_\gamma*z {-E_{\gamma,1}(-(\tfrac{t}{\tau})^{\gamma})\psi_0}\right).
$$
The $z$ forms of each of the linear models are also listed in Table~\ref{table:fMGT}, where \[\tilde{f}=f-\delta\tau^{-\gamma}\Dt^{\beta-\gamma} E_{\gamma,1}(-(\tfrac{t}{\tau})^{\gamma})\D\psi_0\] for the fMGT III equation. Note that for $\delta=0$ they are all the same and their analysis can be performed as in Section~\ref{Sec:LinearAnalysis}. We thus focus here on the more challenging case of $\delta>0$.

%%%%%%%%%%%%%%%%%%%%%%%%%%%%%%%%%%%%%%%%%%%%%%%%%%%%%%%
\subsection{Analysis of the fMGT II equation with $\boldsymbol{\sigma=0}$} \label{Sec:LinearAnalysis}
We carry out an analysis of the fMGT II equation
\[
\tau^\alpha \Dt^{2+\alpha}\psi + \psi_{tt} - c^2\Delta\psi -\tau^\alpha c^2\Dt^\alpha\Delta \psi - \delta \Dt^\alpha \Delta \psi = f.
\]
It can be rewritten in terms of \[z=\tau^\alpha \Dt^\alpha\psi+\psi\] as
\begin{equation}\label{zform_fMGT4}
z_{tt}-(c^2+\tfrac{\delta}{\tau^\alpha})\Delta z + \tfrac{\delta}{\tau^\alpha} \, \kerML_\alpha*\Delta z =f
{-\delta\tau^{-\alpha} E_{\alpha,1}(-(\tfrac{t}{\tau})^{\alpha})\D\psi_0}
\end{equation}
with the kernel function
\begin{equation}\label{kalpha}
\kerML_\alpha(t)= \tau^{-\alpha}t^{\alpha-1}E_{\alpha,\alpha}(-(\tfrac{t}{\tau})^{\alpha}) = - \ddt E_{\alpha,1}(-(\tfrac{t}{\tau}^\alpha)).
\end{equation}
We observe that this kernel has the following properties:
\begin{equation}\label{kerprop}
\kerML_\alpha(t)\geq0\,, \quad 
\lim_{t\to0+} \kerML_\alpha(t) = +\infty\,, \quad 
\int_0^\infty \kerML_\alpha(t)\dt =1\,, \quad 
\kerML_\alpha'(t)\leq 0.
\end{equation}
\begin{table}[h]
	\captionsetup{width=.9\linewidth}
	\begin{center} \small
		\begin{tabular}{|m{1.14cm}||m{10.6cm}|}
			\hline \vspace*{2mm}
			{fMGT} &  \vspace*{2mm} \textbf{Linear time-fractional acoustic equations} \\[6pt]
			\Xhline{2\arrayrulewidth} \hline \vspace*{4mm}
			\centering	\hphantom{I}   &  $(1+\tau^\alpha \Dal)(\psi_{tt}-c^2 \Delta \psi)- \delta \D\psi_{t}=f$ \\[2mm]	\Xhline{0.02\arrayrulewidth} \vspace*{5mm}
			\centering	$z$ form  &  {$z_{tt}-c^2\Delta z -\dfrac{\delta}{\tau^\alpha}\Dt^{1-\alpha}\Delta z+ \dfrac{\delta}{\tau^\alpha} \displaystyle \int_0^t \kernel(t-s) \Delta z \ds=\tilde{f}$, \ \ $\kernel=\ddt(\kersing_{1-\alpha}*\kerML_\alpha)$} \\[5mm]
			\hline \hline \vspace*{3mm}	
			\centering	I  &  $(1+\tau^\alpha \Dal)(\psi_{tt}-c^2 \Delta \psi)- \delta \Doal \D\psi_{t}=f$\\[2mm]	\Xhline{0.02\arrayrulewidth} \vspace*{5mm}
			\centering	$z$ form   &  {$z_{tt}-c^2\Delta z -\dfrac{\delta}{\tau^\alpha}\Dt^{2-2\alpha}\Delta z+ \dfrac{\delta}{\tau^\alpha} \displaystyle \int_0^t \kernel(t-s) \Delta z \ds=\tilde{f}$, \ \ $\kernel=\ddt(\kersing_{2(1-\alpha)}*\kerML_\alpha)$}\\[5mm]
			\hline \hline \vspace*{3mm}
			\centering	II   &  $(1+\tau^\alpha \Dal)(\psi_{tt}-c^2 \Delta \psi)-\delta \Dal \D\psi=f$ \\[2mm]	\Xhline{0.02\arrayrulewidth} \vspace*{5mm}
			\centering	$z$ form   &  $z_{tt}-\left(c^2+\dfrac{\delta}{\tau^\alpha}\right)\Delta z + \dfrac{\delta}{\tau^\alpha} \displaystyle \int_0^t \kernel(t-s) \Delta z \ds=f$, \ \ $\kernel=\kerML_\alpha$\\[4mm]
			\hline \hline \vspace*{3mm}
			\centering	III &  $(1+\tau \partial_t)(\psi_{tt}-c^2 \Delta \psi)- \delta \Doal\D\psi_{t}=f$ \\[2mm]	\Xhline{0.02\arrayrulewidth} \vspace*{5mm}	
			\centering	$z$ form  &  {$z_{tt}-c^2\Delta z -\dfrac{\delta}{\tau}\Dt^{1-\alpha}\Delta z+ \dfrac{\delta}{\tau} \displaystyle \int_0^t \kernel(t-s) \Delta z\ds =\tilde{f}$, \ \ $\kernel=\ddt(\kersing_{1-\alpha}*\kerML_1)$} \\[5mm]
			\hline 		\end{tabular}
	\end{center}
	\caption{\small Linear time-fractional models with $\kerML_\alpha(t)= \tau^{-\alpha}t^{\alpha-1}E_{\alpha,\alpha}(-(\tfrac{t}{\tau})^{\alpha})$, 
		$\kerML_1(t)=\tau^{-1}\exp(-(\tfrac{t}{\tau}))$, and $\kersing_\gamma(t)=t^{-\gamma}$.
		\label{table:fMGT}}
\end{table}
~\\
Moreover, since the function $t\mapsto  E_{\alpha,1}(-(\tfrac{t}{\tau}^\alpha))$ is completely monotone, by Schoenberg's theorem \cite[Theorem 7.13]{wendland2004} we conclude that the kernel $\kerML_\alpha$ itself and also the kernel $t\mapsto \int_t^\infty \kerML_\alpha(s)\ds$ is positive definite.
Therefore, the next result is a straightforward consequence of~\cite[Theorem 4.5]{cannarsaSforza2008}, where the regularity of $\psi$ as the unique solution of 
	\begin{equation} \label{ODEpsifromz}
\tau^\alpha \Dt^\alpha\psi+\psi=z
	\end{equation}
a.e. in $(0,T)$ with $\psi(0)=\psi_0$ follows from the fact that  $\textup{I}^\alpha$ maps $L^\infty(0,T)$ to $C^{0,\alpha}(0,T)$; see~\cite[Corollary 2, p.\ 56]{samko1993fractional}.
\begin{proposition}[Well-posedness of the fMGT II equation] \label{Prop:fMGT_4} 
Let $\alpha \in (0,1]$. Given $f \in L^1(0,T; L^2(\Om))$, $(z_0,z_1) \in H_0^1(\Om)\times L^2(\Om)$, and $\psi_0\in\Honetwo$, there exists a unique mild solution 
\[z\in W^{2,\infty}(0,T; H^{-1}(\Om))\cap W^{1,\infty}(0,T; L^2(\Om))\cap L^\infty(0,T; H_0^1(\Om))\] of \eqref{zform_fMGT4} with $(z, z_t)\vert_{t=0}=(z_0, z_1)$.
Correspondingly, for inital data
\[
(\psi_0, \psi_1, \psi_2) \in 
\begin{cases}
\Honetwo\times \{0\}\times H^{-1}(\Om) \hphantom{~} \ \mbox{ if }\alpha<1, \\
\Honetwo\times H_0^1(\Om)\times L^2(\Om) \ \mbox{ if }\alpha=1,
\end{cases}
\] 
there exists a unique solution 
\[\psi
\in C^{2,\alpha}(0,T; H^{-1}(\Om))\cap C^{1,\alpha}(0,T; L^2(\Om))\cap C^{0,\alpha}(0,T; H_0^1(\Om))\] 
according to \eqref{psifromz}; that is, the unique solution of \eqref{ODEpsifromz} with $\psi(0)=\psi_0$ and thus of 
	\begin{equation} \label{fMGT4}
	\begin{aligned}
	\begin{multlined}[t] \langle \tau^\alpha \Dt^{2+\alpha}\psi + \psi_{tt}, v\rangle_{H^{-1}, H_0^1} \\+ \prodLtwo{c^2\nabla\psi +\tau^\alpha c^2\nabla \Dt^\alpha\psi +\delta \Dt^\alpha \nabla \psi}{\nabla v} = \prodLtwo{f}{v}, \end{multlined}
	\end{aligned}
	\end{equation}
	a.e. in $(0,T)$, with $(\psi, \psi_t, \psi_{tt})\vert_{t=0}=(\psi_0, \psi_1, \psi_2)$.
Furthermore, the solution satisfies the estimate
	\begin{equation} \label{energy_est1_fMGT4}
	\begin{aligned}
&\begin{multlined}[t]\|\tau^\alpha D_{tt}\Dt^\alpha \psi+\psi_{tt}\|^2_{L^\infty(H^{-1})}
+\|\tau^\alpha D_{t}\Dt^\alpha\psi+\psi_t\|_{\LinfLtwo}^2 \\ \hspace*{6cm}+ \|\tau^\alpha \Dt^\alpha\nabla\psi+\nabla\psi\|_{\LinfLtwo}^2 \end{multlined} \\
\lesssim&\,  \begin{cases}
\|\psi_2\|_{H^{-1}(\Omega)}^2 + \|\nabla\psi_0\|_{L^2(\Omega)}^2
+\|f\|_{\LoneLtwo}^2 \hphantom{\hspace*{2.8cm}} &\mbox{ for }\alpha<1,\\
\|\tau\psi_2+\psi_1\|_{L^2(\Omega)}^2 +\|\tau\nabla\psi_1+\nabla\psi_0\|_{L^2(\Omega)}^2
+\|f\|_{\LoneLtwo}^2\quad & \mbox{ for }\alpha=1.
\end{cases}
	\end{aligned}
	\end{equation}	
\end{proposition}
Note that the condition $\psi_1=0$ in case $\alpha<1$ is enforced by the singularity at zero in the identity 
\[z_t=\Dt\left(\tau^\alpha \Dt^\alpha\psi+\psi\right)= \tau^\alpha \left(\Dt^{1+\alpha}\psi + \psi_t(0)\frac{t^{-\alpha}}{\Gamma(1-\alpha)}\right) + \psi_t\,.
\]
Spatially higher-order regularity can be obtained with more regular initial data and the source term by using the multiplier $(-\D)^m z_t$ in place of $z_t$ (which led to the energy estimate \eqref{energy_est1_fMGT4}). To study the limiting behavior, we will make use of the following resulting estimate in case $\alpha=1$, $m=1$:
\begin{equation} \label{energy_est2_fMGT4_alpha1}
\begin{aligned}
&\|\tau \psi_{ttt}+\psi_{tt}\|^2_{L^\infty(L^2)}
+\|\tau \nabla\psi_{tt}+\nabla\psi_t\|^2_{L^\infty(L^2)} + \|\tau\D\psi_t+\D\psi\|^2_{\LinfLtwo}\\
\lesssim&\,  \|\tau\nabla\psi_2+\nabla\psi_1\|_{L^2(\Omega)}^2 +\|\tau\D\psi_1+\D\psi_0\|_{L^2(\Omega)}^2
+\|\nabla f\|_{\LoneLtwo}^2.
\end{aligned}
\end{equation}	
\begin{remark}[On the $z$-form of the other linear models]\label{rem:fMGT123}
As already mentioned, a reformulation of the type \eqref{zform_fMGT4} is available for the other linear models \textup{fMGT}, \textup{fMGT I}, and \textup{fMGT III} as well; see Table~\ref{table:fMGT}. However, it is not clear whether properties \eqref{kerprop} still hold for the corresponding kernels. \\
\indent Due to the term $\Dt^{\beta-\gamma} ( z-\kerML_\gamma*z)$ present in these models with \[\epsilon=\beta-\gamma>0\in\{1-\alpha,2-2\alpha\},\] one might consider using the adjoint of $(\Dt^\epsilon)^{-1}$ in the multiplier; that is, test with $((\Dt^\epsilon)^{-1})^*z_t$ in place of $z_t$. Indeed, this leads to tractable terms
\begin{equation}\label{testIepsilon}
\begin{aligned}
&-\tfrac{\delta}{\tau^\alpha}\int_0^t\prodLtwo{\Dt^{\epsilon} \Delta( z-\kerML_\gamma*z)}{((\Dt^\epsilon)^{-1})^*z_t}\dt
= \tfrac{\delta}{\tau^\alpha}\int_0^t\prodLtwo{\nabla( z-\kerML_\gamma*z)}{z_t}\dt,\\
&\int_0^t\prodLtwo{z_{tt}}{((\Dt^\epsilon)^{-1})^*z_t}\dt 
= \int_0^t\prodLtwo{(\Dt^\epsilon)^{-1} \Dt^\epsilon \textup{I}^{\epsilon} z_{tt}}{z_t}\dt
= \int_0^t\prodLtwo{\Dt^{2-\epsilon} z}{z_t}\dt\,.
\end{aligned}
\end{equation}
However, the $c^2$ term does not appear to be amenable to useful estimates since in 
\[
- c^2\int_0^t\prodLtwo{\D z}{((\Dt^\epsilon)^{-1})^*z_t}\dt
= c^2\int_0^t\prodLtwo{\nabla \textup{I}^\epsilon }{\nabla z_t}\dt
\]
the difference between the time differentiation orders of the two factors is $1-(-\epsilon)=1+\epsilon>1$, which leads to an adverse sign, while the norm of $\nabla z_t$ is not controllable by any of the other left-hand side terms resulting from \eqref{testIepsilon}.
\end{remark}
\subsection{Limiting behavior of the fMGT II equation} 
For $\alpha\in(0,1]$, we denote by $\psi^\alpha$ the solution according to Proposition~\ref{Prop:fMGT_4} under the assumptions \[(\psi_0, \psi_1, \psi_2) \in \Honetwo\times \{0\}\times H_0^1(\Om),\quad f, \nabla f\in L^1(0,T; L^2(\Om)).\]
Let $\psi$ be the solution of the corresponding MGT equation. Note that then the corresponding functions $z^\alpha$ and $z$ satisfy the initial conditions 
\[z^\alpha(0)=z(0)=\psi_0,\quad z^\alpha_t(0)=0, \quad z_t(0)=\tau\psi_2.\] Hence to achieve compatibility, besides $\psi_1=0$ (see the proof of Proposition~\ref{Prop:fMGT_4}), we also have to assume $\psi_2=0$.
Then the difference $\overline{z}=z^\alpha-z$ solves 
\begin{equation} \label{fMGT_4_diff}
\begin{aligned}
&\overline{z}_{tt}-(c^2+\tfrac{\delta}{\tau^\alpha})\D \overline{z} + \delta \, \kerML_\alpha*\D \overline{z} 
&=\delta \left( (\tau^{-\alpha}-\tau^{-1})\D z - (\kerML_\alpha-\kerML_1)*\D z\right) 
\end{aligned}
\end{equation}
with homogeneous initial data $\overline{z}(0)=\overline{z}_t(0)=0$. Testing with $\overline{z}_t$ leads to 
\begin{equation}\label{fMGT_4_estdiffz}
\|\overline{z}_t\|_{L^\infty(L^2)}^2 + 2c^2 \|\nabla \overline{z}\|_{L^\infty(L^2)}^2	
\leq 4 \delta\Bigl(|\tau^{-\alpha}-\tau^{-1}| + \|\kerML_\alpha-\kerML_1\|_{L^1(0,T)}\Bigr) \|\D z\|_{\LoneLtwo}^2,
\end{equation}
where we can estimate $\|\D z\|_{\LoneLtwo}^2$ according to \eqref{energy_est2_fMGT4_alpha1} and 
\[
\|\kerML_\alpha-\kerML_1\|_{L^1(0,T)}
=\int_0^T \left| \tau^{-\alpha}t^{\alpha-1}E_{\alpha,\alpha}(-(\tfrac{t}{\tau})^{\alpha}) 
- \tau^{-1}\exp(-(\tfrac{t}{\tau}))\right|\dt\to0\ \mbox{ as }\alpha\to1^{-},
\]
by Lebesgue's Dominated Convergence theorem.
Thus, with 
\[
\begin{aligned}
\psi^\alpha-\psi &= \kerML_\alpha*z^\alpha-\kerML_1*z =(\kerML_\alpha-\kerML_1)*z+\kerML_\alpha*\overline{z}\\
(\psi^\alpha-\psi)_t(t) &=(\kerML_\alpha-\kerML_1)(t)\,z^\alpha(0)+ ((\kerML_\alpha-\kerML_1)*z_t)(t)
+\kerML_\alpha(t)\overline{z}(0)+(\kerML_\alpha*\overline{z}_t)(t)\\
\end{aligned}
\]
we obtain, using $\overline{z}(0)=0$, $z^\alpha(0)=\psi_0$, and Young's Convolution inequality,
\begin{equation} 
\begin{aligned}
&\|(\psi^\alpha-\psi)_t\|_{L^\infty(L^2)}^2 + \|\nabla (\psi^\alpha-\psi)\|_{L^\infty(L^2)}^2\\
\lesssim&\,
\|\kerML_\alpha-\kerML_1\|_{L^p(0,T)}^2\|\psi_0\|_{L^2(\Om)}^2 
+ \|\kerML_\alpha-\kerML_1\|_{L^1(0,T)}^2\|z_t\|_{L^p(L^2)}^2
+\|\kerML_\alpha\|_{L^1(0,T)}^2 \|\overline{z}_t\|_{L^p(L^2)}^2\\
&\quad+ \|\kerML_\alpha-\kerML_1\|_{L^1(0,T)}^2\|\nabla z_t\|_{L^\infty(L^2)}^2
+ \|\kerML_\alpha\|_{L^1(0,T)}^2 \|\nabla\overline{z}\|_{L^\infty(L^2)}^2\,,
\end{aligned}
\end{equation}
where we can use Proposition~\ref{Prop:fMGT_4} and estimate \eqref{fMGT_4_estdiffz} to further bound the right-hand side terms.
Note that since \[\|\kerML_\alpha-\kerML_1\|_{L^\infty(0,T)}=\lim_{t\to0}|\kerML_\alpha-\kerML_1|(t)=+\infty,\] we only get an estimate of $\|(\psi^\alpha-\psi)_t\|_{L^\infty(L^2)}^2$ if we additionally assume $\psi_0=0$. Furthermore,
\[
\begin{aligned}
\|\kerML_\alpha-\kerML_1\|_{L^p(0,T)}
\leq& 
|\tau^{-\alpha}-\tau^{-1}|\|\cdot^{\alpha-1}E_{\alpha,\alpha}(-(\tfrac{\cdot}{\tau})^{\alpha})\|_{L^p(0,T)}\\
&+\tau^{-1}\|\cdot^{\alpha-1}-1\|_{L^p(0,T)}\|E_{\alpha,\alpha}(-(\tfrac{\cdot}{\tau})^{\alpha})\|_{L^\infty(0,T)}\\
&+\tau^{-1}\|E_{\alpha,\alpha}(-(\tfrac{\cdot}{\tau})^{\alpha})-\exp(-(\tfrac{\cdot}{\tau}))\|_{L^p(0,T)}
\to0 \ \ \mbox{ as }\alpha\to1^{-}\,,
\end{aligned}
\]
where the only critical term  is the one containing the singularity, $\int_0 |t^{-(1-\alpha)} -1|^p\dt$. Its convergence to zero follows from Lebesgue's Dominated Convergence theorem with the $L^1$ bound \[2^{p-1}(t^{-p(1-\alpha_0)}+1)\] for $1\geq\alpha\geq\alpha_0>1-\frac{1}{p}$. Thus we arrive at the following result. 
\begin{proposition} Let $\psi_0 \in H_0^1(\Om)\cap H^2(\Om)$ and $f$, $\nabla f\in L^1(0,T; L^2(\Om))$. Further, let $\{\psi^\alpha\}_{\alpha \in (0,1)}$ be the family of solutions to the \textup{fMGT II} equation and let $\psi$ solve the corresponding equation with $\alpha=1$, where the initial data is in both cases given by \[(\psi^\alpha, \psi^\alpha_t, \psi^\alpha_{tt})\vert_{t=0}=(\psi, \psi_t, \psi_{tt})\vert_{t=0}=(\psi_0,0,0).\] Then for any $p\in[1,\infty)$, $\psi^{\alpha}$ converges to $\psi$ in the $W^{1,p}(0,T;L^2(\Om))\cap L^\infty(0,T;H_0^1(\Om))$ norm as $\alpha \rightarrow 1^{-}$.
If additionally $\psi_0=0$, then we also have convergence in $W^{1,\infty}(0,T;L^2(\Om))\cap L^\infty(0,T;H_0^1(\Om))$.
\end{proposition}

%%%%%%%%%%%%%%%%%%%%%%%%%%%%%%%%%%%%%%%%%%%%%%%%%%%%%%%
\section*{Conclusion and Outlook}
In this work, based on physical balance and constitutive laws, we have derived four different fractional-order versions of a well-known third-order in time model of nonlinear acoustics, the JMGT equation. The fractional-order of differentiation $\alpha\in(0,1]$ (sometimes restricted to $\alpha\in(1/2,1]$) appears as a parameter in each of these models. We have studied the well-posedness of these equations and their linearizations in appropriate spaces and justified the respective limits as $\alpha\to1^{-}$, leading to the (J)MGT equation. \\
\indent Formally taking the limit of these equations as the relaxation time $\tau$ vanishes would lead to time-fractional second-order acoustic equations, which are of independent interest as well. An analysis of this limit will be the subject of future research.
\section*{Acknowledgments}
The work of the first author was supported by the Austrian Science Fund {\sc fwf} under the grants P30054 and DOC 78.
\end{document}